\documentclass{article} % For LaTeX2e
\usepackage{mystyle,times}

\usepackage{hyperref}
\usepackage{url}

\iclrfinalcopy
\title{ADMM for Nonsmooth Composite Optimization under Orthogonality Constraints}

\author{Ganzhao Yuan\\
Peng Cheng Laboratory, China\\
\texttt{yuangzh@pcl.ac.cn}}

\usepackage{graphicx}
\usepackage{epstopdf}
\usepackage{multirow}  % 加载 multirow 宏包
\usepackage{enumitem,amsfonts,amssymb,mathtools,amsthm,bm,pifont,bbm,listings,xcolor,threeparttable,xcolor}
\usepackage{ulem,graphicx,subcaption,hyperref,comment,amsmath}
\usepackage[ruled,lined]{algorithm2e}
\usepackage{algpseudocode}
\usepackage{natbib}

\newtheorem{theorem}{Theorem}[section]
\newtheorem{proposition}[theorem]{Proposition}
\newtheorem{lemma}[theorem]{Lemma}

\newtheorem{definition}[theorem]{Definition}
\newtheorem{assumption}[theorem]{Assumption}
\newtheorem{remark}[theorem]{Remark}
\def\a{\mathbf{a}}\def\b{\mathbf{b}}\def\d{\mathbf{d}}
\def\p{\mathbf{p}}\def\s{\mathbf{s}}
\def\v{\mathbf{v}}\def\w{\mathbf{w}}\def\x{\mathbf{x}}\def\y{\mathbf{y}}\def\z{\mathbf{z}}

\def\B{\mathbf{B}}\def\D{\mathbf{D}}\def\G{\mathbf{G}}\def\I{\mathbf{I}}\def\Q{\mathbf{Q}}\def\T{\mathbf{T}}\def\U{\mathbf{U}}\def\V{\mathbf{V}}\def\X{\mathbf{X}}\def\Y{\mathbf{Y}}\def\Z{\mathbf{Z}}\def\v{\mathbf{v}}\def\p{\mathbf{p}}

\def\AA{\mathcal{A}}\def\BB{\mathcal{B}}\def\MM{\mathcal{M}}\def\OO{\mathcal{O}}\def\SS{\mathcal{S}}\def\XX{\mathcal{X}}\def\YY{\mathcal{Y}}\def\ZZ{\mathcal{Z}}

\def\DDD{\mathbb{D}}\def\GGG{\mathbb{G}}\def\PPP{\mathbb{P}}\def\TTT{\mathbb{T}}

\DeclareMathOperator{\tr}{tr}
\DeclareMathOperator{\qf}{qf}
\DeclareMathOperator{\dom}{dom}
\DeclareMathOperator{\Crit}{Crit}
\DeclareMathOperator{\dist}{dist}
\DeclareMathOperator{\Proj}{Proj}
\DeclareMathOperator{\Retr}{Retr}

\DeclareMathOperator{\prox}{\PPP}

\def\trans{^\mathsf{T}} \def\one{\mathbf{1}} \def\zero{\mathbf{0}}  \def\fro{\mathsf{F}}\def\Rn{\mathbb{R}}\def\ts{\textstyle}\def\diag{{\rm{diag}}}\def\Diag{{\rm{Diag}}}

\newcommand{\la}{\langle} \newcommand{\ra}{\rangle}\newcommand{\nn}{\nonumber}\newcommand\noi{\noindent}

\newcommand{\bfit}[1]{\textit{\textbf{#1}}}
\newcommand{\beq}{\begin{eqnarray}}
\newcommand{\eeq}{\end{eqnarray}}
\newcommand{\step}[1]{\text{\ding{\numexpr#1+171\relax}}}

\def\Deltas{\boldsymbol{\Delta}}

\def\cc{{\sf c}}

\def\ww{{\mathcal{W}}}

\def\Aup{\overline{\rm{A}}}
\def\ellup{\overline{\rm{\ell}}}
\def\elldown{\underline{\rm{\ell}}}

\renewcommand{\cite}{\citep}

\begin{document}

\maketitle

\begin{abstract}
We consider a class of structured, nonconvex, nonsmooth optimization problems under orthogonality constraints, where the objectives combine a smooth function, a nonsmooth concave function, and a nonsmooth weakly convex function. This class of problems finds diverse applications in statistical learning and data science. Existing methods for addressing these problems often fail to exploit the specific structure of orthogonality constraints, struggle with nonsmooth functions, or result in suboptimal oracle complexity. We propose {\sf OADMM}, an Alternating Direction Method of Multipliers (ADMM) designed to solve this class of problems using efficient proximal linearized strategies. Two specific variants of {\sf OADMM} are explored: one based on Euclidean Projection ({\sf OADMM-EP}) and the other on Riemannian Retraction ({\sf OADMM-RR}). Under mild assumptions, we prove that {\sf OADMM} converges to a critical point of the problem with an ergodic convergence rate of $\mathcal{O}(1/\epsilon^{3})$. Additionally, we establish a polynomial convergence rate or super-exponential convergence rate for {\sf OADMM}, depending on the specific setting, under the Kurdyka-Lojasiewicz (KL) inequality. To the best of our knowledge, this is \textit{the first non-ergodic convergence result} for this class of nonconvex nonsmooth optimization problems. Numerical experiments demonstrate that the proposed algorithm achieves state-of-the-art performance.

%are conducted to demonstrate the advantages of the proposed method.

%We integrate a Nesterov extrapolation strategy into {\sf OADMM-EP} and a monotone Barzilai-Borwein strategy into {\sf OADMM-RR} to potentially accelerate primal convergence. Additionally, we adopt an over-relaxation strategy in both variants for rapid dual convergence.

\noi \textbf{Keywords:} Orthogonality Constraints; Nonconvex Optimization; Nonsmooth Composite Optimization; ADMM; Convergence Analysis

\end{abstract}

\section{Introduction}

This paper focuses on the following nonsmooth composite optimization problem under orthogonality constraints (`$\triangleq$' means define):
\beq\label{eq:main}
\min_{\X \in \Rn^{n\times r}}\, F(\X)\triangleq f (\X) - g(\X) +  h(\AA(\X)),~s.t.~\X\trans\X = \I_r.
\eeq
\noi Here, $n\geq r$, $\AA(\X)\in \Rn^{m}$ is a linear mapping of $\X$, and $\I_r$ is a $r \times r$ identity matrix. For conciseness, the orthogonality constraints $\X\trans\X = \I_r$ in Problem (\ref{eq:main}) is rewritten as $\X \in\MM \in \Rn^{n\times r}$, with $\MM$ representing the Stiefel manifold in the literature \cite{EdelmanAS98,Absil2008}.

We impose the following assumptions on Problem (\ref{eq:main}) throughout this paper. ($\mathbb{A}$-i) $f(\X)$ is $L_f$-smooth, satisfying $\|\nabla f(\X) - \nabla f(\X')\|_{\fro} \leq L_f \|\X-\X'\|_{\fro}$ holds for all $\X,\X'\in\Rn^{n\times r}$. This implies: $| f(\X) - f(\X') - \la \nabla f(\X'), \X-\X' \ra| \leq \tfrac{L_f}{2}\|\X-\X'\|_{\fro}^2$ (cf. Lemma 1.2.3 in \cite{Nesterov03}). We also assume that $f(\X)$ demonstrates $C_f$-Lipschitz continuity, with $\|\nabla f(\X)\|_{\fro}\leq C_f$ for all $\X\in \MM$. The convexity of $f(\X)$ is not assumed. ($\mathbb{A}$-ii) The function $g(\cdot)$ is convex, proper, and $C_g$-Lipschitz continuous, though it is not necessarily smooth. ($\mathbb{A}$-iii) The function $h(\cdot)$ is proper, lower semicontinuous, $C_h$-Lipschitz continuous, and potentially nonsmooth. Also, it is weakly convexity with constant $W_h\geq0$, which implies that the function $h(\y)+\frac{W_h}{2}\|\y\|_2^2$ is convex for all $\y\in\Rn^{m}$. ($\mathbb{A}$-iv) The proximal operator, $\prox_{\mu}(\y')\triangleq\arg \min_{\y} \tfrac{1}{2\mu}\|\y-\y'\|_2^2 + h(\y)$, can be computed efficiently and exactly for any given $\mu>0$ and $\y'\in\Rn^m$.

Problem (\ref{eq:main}) represents an optimization framework that plays a crucial role in a variety of statistical learning and data science models. These models include sparse Principal Component Analysis (PCA) \cite{journee2010generalized,lu2012augmented}, deep neural networks \cite{cho2017riemannian,xie2017all,bansal2018can,CogswellAGZB15,HuangG23}, orthogonal nonnegative matrix factorization \cite{jiang2022exact}, range-based independent component analysis \cite{selvan2015range}, and dictionary learning \cite{zhai2020complete}.

%phase synchronization \cite{liu2017estimation}

%, $K$-indicators clustering \cite{jiang2016lp}.

%\input{sect2related}
\subsection{Related Work}

$\blacktriangleright$ \textbf{Optimization under Orthogonality Constraints}. Solving Problem (\ref{eq:main}) is challenging due to the computationally expensive and non-convex orthogonality constraints. Existing methods can be divided into three classes. \bfit{(i)} Geodesic-like methods \cite{EdelmanAS98,abrudan2008steepest,Absil2008,jiang2015framework}. These methods involve calculating geodesics by solving ordinary differential equations, which can introduce significant computational complexity. To mitigate this, geodesic-like methods iteratively compute the geodesic logarithm using simple linear algebra calculations. Efficient constraint-preserving update schemes have been integrated with the Barzilai-Borwein (BB) stepsize strategy \cite{wen2013feasible,jiang2015framework} for minimizing smooth functions under orthogonality constraints. \bfit{(ii)} Projection and retractions methods \cite{Absil2008,matcom2013}. These methods maintain orthogonality constraints through projection or retraction. They reduce the objective value by using its current Euclidean gradient direction or Riemannian tangent direction, followed by an orthogonal projection operation. This projection can be computed using polar decomposition or singular value decomposition, or approximated with QR factorization. \bfit{(iii)} Multiplier correction methods \cite{Gau2018,Gau2019,xiao2020class}. Leveraging the insight that the Lagrangian multiplier associated with the orthogonality constraint is symmetric and has an explicit closed-form expression at the first-order optimality condition, these methods tackle an alternative unconstrained nonlinear objective minimization problem, rather than the original smooth function under orthogonality constraints.

%However, computing the Hessian matrix for the Newton subprocedure has a memory complexity of $\mathcal{O}(r^4)$.
$\blacktriangleright$ \textbf{Optimization with Nonsmooth Objectives}. Another challenge in addressing Problem (\ref{eq:main}) stems from the nonsmooth nature of the objective function. Existing methods for tackling this challenge fall into three main categories. \bfit{(i)} Subgradient methods \cite{ferreira1998subgradient,HwangCRIAJS15,li2021weakly}. Subgradient methods, analogous to gradient descent methods, can incorporate various geodesic-like and projection-like techniques. However, they often exhibit slower convergence rates compared to other approaches. \bfit{(ii)} Proximal gradient methods \cite{chen2020proximal}. These methods use a semi-smooth Newton approach to solve a strongly convex minimization problem over the tangent space, finding a descent direction while preserving the orthogonality constraint through a retraction operation. \bfit{(iii)} Operator splitting methods \cite{lai2014splitting,chen2016augmented,zhang2020primal}. These methods introduce linear constraints to break down the original problem into simpler subproblems that can be solved separately and exactly. Among these, ADMM is a promising solution for Problem (\ref{eq:main}) due to its capability to handle nonsmooth objectives and nonconvex constraints separately and alternately. Several ADMM-like algorithms have been proposed for solving nonconvex problems \cite{boct2020proximal,Boct2019SIOPT,wang2019global,li2015global,HeY12,yuanADMM,zhang2020primal}, but these methods fail to exploit the specific structure of orthogonality constraints or cannot be adapted to solve Problem (\ref{eq:main}). \bfit{(iv)} Other methods. {\sf OBCD} \cite{yuan2023OBCD} has been proposed to solve a specific class of our problems, while the inexact augmented Lagrangian method {\sf ManIAL} was introduced in \cite{deng2024oracle}.

$\blacktriangleright$ \textbf{Detailed Discussions on Operator Splitting Methods}. We list some popular variants of operator splitting methods for tackling Problem (\ref{eq:main}). Initially, two natural splitting strategies are used in the literature:
\beq
\ts\min_{\X,\y}\,F_1(\X,\y)\triangleq f(\X)-g(\X)+h(\y) + \iota_{\MM}(\X),~s.t.~\AA(\X)=\y \label{eq:splitting:1}\\
\ts \min_{\X,\Y}\,F_2(\X,\Y)\triangleq f(\X)-g(\X)+h(\AA(\X)) + \iota_{\MM}(\Y),~s.t.~\X=\Y.\label{eq:splitting:2}
\eeq
\noi (\bfit{a}) Smoothing Proximal Gradient Methods ({\sf SPGM}, \cite{beck2023dynamic,BohmW21}) incorporate a penalty (or smoothing) parameter $\mu \rightarrow 0$ to penalize the squared error in the constraints, resulting in the subsequent minimization problem \cite{beck2023dynamic,BohmW21,chen2012smoothing}: $\min_{\X,\y}~F_1(\X,\y) + \tfrac{1}{2\mu}\|\AA(\X)-\y\|_2^2$. During each iteration, {\sf SPGM} employs proximal gradient strategies to alternatively minimize \textit{w.r.t.} $\X$ and $\y$. (\bfit{b}) Splitting Orthogonality Constraints Methods ({\sf SOCM}, \cite{lai2014splitting}) use the following iteration scheme: $\X^{t+1} \thickapprox \arg \min_{\X} F_2(\X,\Y^t) + \la \Z^t, \X - \Y^t \ra + \tfrac{\beta}{2}\|\X-\Y\|_{\fro}^2$, $\Y^{t+1} \in \arg \min_{\Y} F_2(\X^{t+1},\Y) + \la \Z^t, \X^{t+1} - \Y \ra + \tfrac{\beta}{2}\|\X^{t+1}-\Y\|_{\fro}^2$, and $\Z^{t+1} = \Z^t + \beta (\X^{t+1} - \Y^{t+1})$, where $\beta$ is a fixed penalty constant, and $\Z^t$ is the multiplier associated with the constraint $\X=\Y$ at iteration $t$. (\bfit{c}) Similarly, Manifold ADMM ({\sf MADMM}, \cite{kovnatsky2016madmm}) iterates as follows: $\X^{t+1} \thickapprox \arg \min_{\X} F_1(\X,\y^t) + \la \z^t, \AA(\X) - \y^t \ra + \tfrac{\beta}{2}\|\AA(\X) - \y^t\|_{\fro}^2$, $\y^{t+1} \in \arg \min_{\y} F_1(\X^{t+1},\y) + \la \z^t, \AA(\X^{t+1}) - \y \ra + \tfrac{\beta}{2}\|\AA(\X^{t+1}) - \y\|_{\fro}^2$, and $\z^{t+1} = \z^t + \beta (A(\X^{t+1}) - \y^{t+1})$, where $\z^t$ is the multiplier associated with the constraint $A(\X) - \y=\zero$ at iteration $t$. (\bfit{d}) Like {\sf MADMM}, Riemannian ADMM ({\sf RADMM}, \cite{li2022riemannian}) operates using the first splitting strategy in Equation (\ref{eq:splitting:1}). In contrast, it employs a Riemannian retraction strategy to solve the $\X$-subproblem and a Moreau envelope smoothing strategy to solve the $\y$-subproblem.

\textbf{Contributions.} We compare existing methods for solving Problem (\ref{eq:main}) in Table \ref{tab:comp}, and our main contributions are summarized as follows. (\bfit{i}) We introduce {\sf OADMM}, a specialized ADMM designed for structured nonsmooth composite optimization problems under orthogonality constraints in Problem (\ref{eq:main}). Two specific variants of {\sf OADMM} are explored: one based on Euclidean Projection ({\sf OADMM-EP}) and the other on Riemannian Retraction ({\sf OADMM-EP}). Notably, while many existing works primarily address cases where $g(\X)=0$ and $h(\cdot)$ is convex, our approach considers a more general setting where $h(\cdot)$ is weakly convex and $g(\X)$ is convex. (ii) {\sf OADMM} could demonstrate fast convergence by incorporating Nesterov's extrapolation \cite{Nesterov03} into {\sf OADMM-EP} and a Monotone Barzilai-Borwein (MBB) stepsize strategy \cite{wen2013feasible} into {\sf OADMM-RR} to potentially accelerate primal convergence. Both variants also employ an over-relaxation strategy to enhance dual convergence \cite{goncalves2017convergence,yang2017alternating,li2016majorized}. (\bfit{iii}) By introducing a novel Lyapunov function, we establish the convergence of {\sf OADMM} to critical points of Problem (\ref{eq:main}) within an oracle complexity of $\mathcal{O}(1/\epsilon^{3})$, matching the best-known results to date \cite{beck2023dynamic,BohmW21}. This is achieved through a decreasing step size for updating primal and dual variables. In contrast, {\sf RADMM} employs a small constant step size for such updates, resulting in a sub-optimal oracle complexity of $\mathcal{O}(\epsilon^{-4})$ \cite{li2022riemannian}. (\bfit{iv}) We establish a polynomial convergence rate or super-exponential convergence rate for {\sf OADMM}, depending on the specific setting, under the Kurdyka-Lojasiewicz (KL) inequality, providing \textit{the first non-ergodic convergence result} for this class of non-convex nonsmooth optimization problems.

%Riemannian subgradient

%Riemannian proximal gradient

%Smoothing Riemannian proximal gradient

%It's worth noting that similar to ManPG, the algorithms in [23, 22, 58] all necessitate solving a relatively challenging subproblem, which typically requires the use of a semi-smooth Newton algorithm.
%$\mathcal{O}(\cdot)$ notation hides constant factors.
\begin{table*}[!t]
\centering
\caption{Comparison of existing methods for solving Problem (\ref{eq:main}).  }  %\smallskip
%\vspace{1pt}
\scalebox{0.75}{\begin{tabular}{|p{4.9cm}|p{2.76cm}|p{1cm}|p{2.7cm}|p{1.6cm}|p{2.4cm}|}
\hline
\textbf{\footnotesize Reference} &\textbf{\footnotesize $h(\AA(\X))$} &\textbf{\footnotesize $g(\X)$}    & \textbf{\footnotesize Notable Features} & \textbf{\footnotesize Complexity}& \textbf{\footnotesize Conv. Rate}     \\
\hline
{\sf SOCM} \cite{lai2014splitting} & convex $h(\cdot)$ &    empty  & $\sigma=1$, $\alpha=0$ &  \text{unknown}& \text{unknown}   \\
\hline
{\sf MADMM} \cite{kovnatsky2016madmm} & convex $h(\cdot)$ &    empty  & $\sigma=1$, $\alpha=0$ &  \text{unknown}& \text{unknown}   \\
\hline
{\sf RSG} \cite{li2021weakly} & weakly convex $h(\cdot)$ & empty  & -- & $\mathcal{O}(\epsilon^{-4})$   & \text{unknown} \\
\hline
{\sf ManPG} \cite{chen2020proximal} & $h(\AA(\X))=\|\X\|_1$ & empty  & \text{hard subproblem } & $\mathcal{O}(\epsilon^{-2})$ & \text{unknown}   \\
\hline
{\sf OBCD} \cite{yuan2023OBCD} & separable $h(\cdot)$ & empty  &  hard subproblem & $\mathcal{O}(\epsilon^{-2})$ &  \text{unknown}  \\
\hline
{\sf RADMM} \cite{li2022riemannian} & convex $h(\cdot)$ &   empty  &  $\sigma=1$, $\alpha=0$ & $\mathcal{O}(\epsilon^{-4})$ & \text{unknown}   \\
\hline
{\sf ManIAL} \cite{deng2024oracle} & convex $h(\cdot)$ &   empty  &  \text{inexact subproblem} & $\mathcal{O}(\epsilon^{-3})$ & \text{unknown}   \\
\hline
{\sf SPGM} \cite{beck2023dynamic} & convex $h(\cdot)$ &   empty  & -- & $\mathcal{O}(\epsilon^{-3})$& \text{unknown}    \\
\hline
{\sf OADMM-EP}[ours] & weakly convex $h(\cdot)$ &   convex  & $\sigma \in[1,2)$, $\alpha>0$ & $\mathcal{O}(\epsilon^{-3})$  & $\checkmark$ Theorem \ref{theorem:KL:rate:Exponent:P} \\
\hline
{\sf OADMM-RR}[ours] & weakly convex $h(\cdot)$& convex  &  $\sigma \in[1,2)$,  \text{MBB}  & $\mathcal{O}(\epsilon^{-3})$ & $\checkmark$ Theorem \ref{theorem:KL:rate:Exponent:P} \\
\hline
\end{tabular}}
\smallskip
%\vspace{1pt} %&  &
%\scalebox{0.95}{\begin{tabular}{@{}p{\linewidth}@{}}
%\footnotesize Note $\star$: This is known as super-exponential convergence, please refer to Theorem \ref{theorem:KL:rate:Exponent:P}(\bfit{a}) for more details.\\
%\footnotesize Note $\ddagger$: This is known as polynomial convergence, please refer to Theorem \ref{theorem:KL:rate:Exponent:P}(\bfit{b}) for more details.\\
%\footnotesize Note $b$: $N$ is the number of data points for the finite-sum structure (See Equation (\ref{eq:finite:sum})).\\
%\footnotesize Note $c$: The iteration complexity relies on the variational inequality of the convex problem. \\
%\footnotesize Note $e$: $\tilde{\mathcal{O}}$ hides polylogarithmic factors. \\
%\footnotesize Note $d$: We adapt their application model into our optimization framework in Equation (\ref{eq:main}) with $(L,S,Z)=(\x_1,\x_2,\x_3)$, as their model additionally requires the linear operator for the other two blocks to be injective. \\
%\footnotesize Note $e$: Assumption 4 in \cite{wang2019global} claims that the matrix can exhibit either full row rank or full column rank. However, Equation (20) in their analysis relies on the matrix's surjectiveness, while Lemma 7 depends on its injectiveness. \\
%\footnotesize Note $f$: The iteration complexity is contingent on the newly introduced approximate inertial nonconvex proximal point (see Section \ref{sect:iter:com}). \\
%\end{tabular}}
\label{tab:comp}%\vspace{-12pt}
\end{table*}

% \multirow{2}{*}{ $\OO(1/\exp(T^{\dot{u}}))$ or $\OO(1/T^{\ddot{u}})^\ddagger$}
%\input{sect3prel}
\section{Technical Preliminaries} \label{sect:preli}

This section provides some technical preliminaries on Moreau envelopes for weakly convex functions and manifold optimization.

\textbf{Notations.} We define $[n]\triangleq \{1,2,\ldots,n\}$. We use $\AA\trans(\cdot)$ to denote the adjoin operator of $\AA(\cdot)$ with $\la\AA(\X),\z\ra = \la\X,\AA\trans(\z)\ra$ for all $\X\in\Rn^{n\times r}$ and $\z\in\Rn^m$. We define $\Aup\triangleq \max_{\V}\|\AA(\V)\|_{\fro}/\|\V\|_{\fro}$. We use $\iota_{\MM}(\X)$ to denote the indicator function of orthogonality constants. Further notations, technical preliminaries, and relevant lemmas are detailed in Appendix Section \ref{app:sect:notation}.

\subsection{Moreau Envelopes for Weakly Convex Functions}

We provide the following useful definition.

\begin{definition}
For a proper convex, and Lipschitz continuous function $h(\y): \Rn^{m}\mapsto \Rn$, the Moreau envelope of $h(\y)$ with the parameter $\mu>0$ is given by $h_{\mu}(\y)\triangleq \min_{\breve{\y}} h(\breve{\y}) + \frac{1}{2\mu}\|\breve{\y}-\y\|_{2}^2$.
\end{definition}

%\begin{definition}
%A function $h(\y)$ is $W$-weakly convex if $(h(\y)+\tfrac{W}{2}\|\y\|_{2}^2)$ is convex.
%\end{definition}

We show some useful properties of Moreau envelope for weakly convex functions.
%$0\leq h(\y)-g_{\mu}(\y)\leq \mu C_g^2$.

\begin{lemma} \label{lemma:well:knwon:3}
(\cite{BohmW21}) Let $h: \Rn^{m} \mapsto \Rn$ to be a proper, $W_h$-weakly convex, and lower semicontinuous function. Assume $\mu\in(0,W_h^{-1})$. We have the following results. The function $h_{\mu}(\cdot)$ is continuously differentiable with gradient $\nabla h_{\mu}(\y) = \frac{1}{\mu}(\y - \prox_{\mu}(\y))$ for all $\y$, where $\prox_{\mu}(\y) \triangleq \arg \min_{\breve{\y}}  h(\breve{\y}) + \frac{1}{2\mu}\|\breve{\y}-\y\|_2^2$. This gradient is $\max(\mu^{-1}, \frac{W_h}{1-\mu W_h})$-Lipschitz continuous. In particular, when $\mu \in (0,\frac{1}{2 W_h}]$, the condition $\mu^{-1}\geq \frac{W_h}{1-\mu W_h}$ ensures that $h_{\mu}(\y)$ is $(\mu^{-1})$-smooth and $(\mu^{-1})$-weakly convex.

\end{lemma}

\begin{comment}
\begin{lemma}\label{lemma:wc:wc}
(Proof in Appendix \ref{app:lemma:wc:wc}) If $h(\y)$ is $W_h$-weakly convex and $\mu\in(0,\tfrac{1}{2W_h})$, then $h_{\mu}(\y)$ is $\min(\tfrac{W_h}{1-\mu W_h},\mu^{-1})$-weakly convex.
\end{lemma}
\end{comment}

\begin{lemma} \label{lemma:mu:continous}
(Proof in Appendix \ref{app:lemma:mu:continous}) Assume $0<\mu_2<\mu_1<\frac{1}{W_h}$, and fixing $\y\in\Rn^{m}$. We have: $0 \leq h_{\mu_2}(\y) - h_{\mu_1}(\y) \leq \min\{  \tfrac{\mu_1}{2 \mu_2}, 1  \} \cdot (\mu_1-\mu_2) C_h^2$.
\end{lemma}

\begin{lemma} \label{lemma:lip:mu}
(Proof in Appendix \ref{app:lemma:lip:mu}) Assume $0<\mu_2<\mu_1 \leq \frac{1}{2 W_h}$, and fixing $\y\in\Rn^{m}$. We have: $\|\nabla h_{\mu_1}(\y) - \nabla h_{\mu_2}(\y)\| \leq (\tfrac{\mu_1}{\mu_2} - 1) C_h$.
\end{lemma}

\begin{lemma} \label{lemma:smoothing:problem:prox}
(Proof in Appendix \ref{app:lemma:smoothing:problem:prox}) Assume that $h(\y)$ is $W_h$-weakly convex, $\mu\in (0,\tfrac{1}{2 W_h}]$, $\beta>\mu^{-1}$. Consider the following strongly convex optimization problem: $\bar{\y} = \arg \min_{\y}  h_{\mu}(\y) + \frac{\beta}{2}\| \y - \mathbf{b} \|_{2}^2$, which is equivalent to: $(\bar{\y},\breve{\y})= \arg \min_{\y,\y'} h(\y') + \tfrac{1}{2\mu} \|\y' - \y\|_{2}^2 + \tfrac{ \beta}{2}\| \y - \mathbf{b} \|_{2}^2$. We have: (\bfit{a}) $\bar{\y} = \frac{(\breve{\y}  +\mu\beta \mathbf{b} )}{1+\mu\beta} $, where $\breve{\y}=\arg \min_{\y}~  h(\y) + \tfrac{\beta}{2 (1+\mu\beta)}\|\y-\mathbf{b}\|_2^2=\prox(\b;\mu+1/\beta)$. (\bfit{b}) $\beta (\mathbf{b}- \bar{\y}) \in \partial h(\breve{\y})$. (\bfit{c}) $\|\bar{\y}- \breve{\y}\|\leq \mu C_h$.

%$\y^{t+1} = \frac{1}{1/\mu^t+\beta^t} ( \tfrac{1}{\mu^t} \breve{\y}^{t+1}  + \beta^t \mathbf{b} )$, where $\breve{\y}^{t+1} = \arg \min_{\breve{\y}}~  h(\breve{\y}) + \tfrac{\beta^t}{2 (1+\mu^t \beta^t)}\|\breve{\y} - \mathbf{b}\|_2^2$

\end{lemma}

%\begin{remark}
%Lemma \ref{lemma:smoothing:problem:prox} parallels Lemma 1 in \cite{li2022riemannian}, though its results cannot be directly applied in this scenario since $W_h$ is not necessarily zero.
%\end{remark}

\begin{remark}
(\bfit{i}) Lemmas \ref{lemma:mu:continous} and \ref{lemma:lip:mu} presented in this paper are novel. (\bfit{ii}) The upper bound in Lemma \ref{lemma:mu:continous} is slightly better than the bound established in Lemma 4.1 of \cite{BohmW21}. (\bfit{iii}) Lemma \ref{lemma:smoothing:problem:prox} is very critical in our algorithm development and theoretical analysis.
\end{remark}

\subsection{Manifold Optimization}

We define the $\epsilon$-stationary point of Problem (\ref{eq:main}) as follows.

\begin{definition}
 (\rm{First-Order Optimality Conditions}, \cite{chen2020proximal,li2022riemannian}) The solution $(\ddot{\X},\ddot{\y},\ddot{\z})$ with $\ddot{\X}\in\MM$ is called an $\epsilon$-stationary point of Problem (\ref{eq:main}) if: $\Crit(\ddot{\X},\ddot{\y},\ddot{\z}) \leq \epsilon$, where $\Crit(\X,\y,\z) \triangleq\|\AA(\X)- \y\|+ \| \partial h(\y) - \z\| + \| \Proj_{\T_{\X} \MM} ( \nabla f(\X) - \partial g(\X)  + \AA\trans (\z) )\|_{\fro}$. Here, according to \cite{absil2008optimization}, for all $\X\in\MM$ and $\Deltas\in\Rn^{n\times r}$, we have: $\Proj_{\T_{\X}\MM}(\Deltas) = \Deltas - \tfrac{1}{2}\X (\Deltas \trans \X+\X\trans \Deltas)$.

\end{definition}

The proposed algorithm is an iterative procedure. After shifting the current iterate $\X\in\MM$ in the search direction, it may no longer reside on $\MM$. Therefore, we must retract the point onto $\MM$ to form the next iterate. The following definition is useful in this context.

\begin{definition}
A retraction on $\MM$ is a smooth map \cite{absil2008optimization}: $\Retr_{\X}(\Deltas) \in \MM$ with $\X\in\MM$ and $\Deltas \in \Rn^{n\times r}$ satisfying $\Retr_{\X}(\zero) = \X$, and $\lim_{ \T_{\X}\MM \ni \Deltas \rightarrow \zero} \tfrac{ \| \Retr_{\X}(\Deltas) - \X - \Deltas  \|_{\fro}}{\|\Deltas\|_{\fro}} = 0$ for any $\X\in\MM$.

\end{definition}

\begin{remark}
Several retractions on the Stiefel manifold have been explored in literature \cite{absil2012projection,Absil2008}. We present two examples below. (\bfit{i}) Polar Decomposition-Based Retraction: $\Retr_{\X}(\Deltas)=(\X+\Deltas) (\I_r + \Deltas\trans \Deltas)^{-1/2}$. (\textit{ii}) QR-Decomposition-Based Retraction: $\Retr_{\X}(\Deltas)=\qf (\X+\Deltas)$, where $\qf(\X)$ is the $\Q$-factor in the thin QR-decomposition of $\X$.
\end{remark}

The following lemma concerning the retraction operator is useful for our subsequent analysis.

\begin{lemma} \label{lemma:M12}
(\cite{boumal2019global}) Let $\X\in\MM$ and $\Deltas \in \T_{\X}\MM$. There exists positive constants $\{\dot{k},\ddot{k}\}$ such that $\|\Retr_{\X}(\Deltas)-\X\|_{\fro}\leq \dot{k} \|\Deltas\|_{\fro}$, and $\|\Retr_{\X}(\Deltas) - \X - \Deltas\|_{\fro} \leq \tfrac{1}{2}\ddot{k} \|\Deltas\|_{\fro}^2$.
\end{lemma}

Furthermore, we present the following three insightful lemmas.

\begin{lemma} \label{lemma:bound:PXX:XX}

(Proof in Appendix \ref{app:lemma:bound:PXX:XX}) Let $\X\in\MM$ and $\Deltas \in \Rn^{n\times r}$, we have $\|\Proj_{\T_{\X}\MM}(\Deltas)\|_{\fro}\leq \|\Deltas\|_{\fro}$. %, where $\Proj_{\T_{\X}\MM}(\Deltas) = \Deltas - \tfrac{1}{2}\X (\Deltas \trans \X+\X\trans \Deltas)$
\end{lemma}

\begin{lemma} \label{lemma:GD:bound}
(Proof in Appendix \ref{app:lemma:GD:bound}) For any $\rho>0$, $\G\in\Rn^{n\times r}$, and $\X\in\MM$, we define $\GGG_{\rho} \triangleq \G-\rho\X\G\trans\X - (1-\rho)\X\X\trans\G$. It follows that: (\bfit{a}) $\max(1,2\rho) \cdot \la \G,\GGG_{\rho}\ra \geq \|\GGG_{\rho}\|_{\fro}^2\geq \min(1,\rho^2) \|\GGG_{1}\|_{\fro}^2$. (\bfit{b}) $\min(1,2 \rho) \|\GGG_{1/2}\|_{\fro}\leq\|\GGG_{\rho}\|_{\fro} \leq \max(1,2\rho)\|\GGG_{1/2}\|_{\fro}$.

\end{lemma}

\begin{lemma} \label{lemma:subgradient:R:bound}
(Proof in Appendix \ref{app:lemma:subgradient:R:bound}) Consider the following optimization problem: $\min_{\X\in\MM} f(\X)$, where $f(\X)$ is differentiable. For all $\X\in \MM$, we have: $\dist(\zero,\partial  I_{\MM}(\X) + \nabla f(\X))\leq \| \nabla f(\X)  - \X \nabla f(\X) \trans \X\|_{\fro}$.
\end{lemma}

\begin{remark}The matrix $-\GGG_{\rho}\in \Rn^{n\times r}$ in Lemma \ref{lemma:GD:bound} is closely related to the search descent direction of the proposed {\sf OADMM-RR} algorithm. While one can set $\rho$ to typical values such as $1$ or $1/2$, we consider the setting $\rho\in (0,\infty)$ to enhance the versatility of {\sf OADMM-RR}, aligning with \cite{liu2016quadratic,jiang2015framework}.%we provide our analysis to allow $\rho \in (0, +\infty)$ for broader discussions.
\end{remark}
%inspired by \cite{liu2016quadratic,jiang2015framework},

%\input{sect4proposed}

\section{The Proposed {\sf OADMM} Algorithm} \label{sect:algorithm}

This section provides the proposed {\sf OADMM} algorithm for solving Problem (\ref{eq:main}), featuring two variants, one is based on Euclidean Projection ({\sf OADMM-EP}) and the other on Riemannian Retraction ({\sf OADMM-RR}).

Using the Moreau envelope smoothing technique, we consider the following optimization problem:
\beq\label{eq:main:equality}
\min_{\X,\y}\, f(\X) - g(\X) +  h_{\mu}(\y) + \iota_{\MM}(\X),~s.t.~\AA(\X) = \y,
\eeq
\noi where $\mu \rightarrow 0$, and $h_{\mu}(\y)$ is the Moreau Envelope of $h(\y)$. Importantly, $h_{\mu}(\y)$ is $(\mu^{-1})$-smooth when $\mu \leq \frac{1}{2W_h}$, according to Lemma \ref{lemma:well:knwon:3}. It is worth noting that similar smoothing techniques have been used in the design of augmented Lagrangian methods \cite{zeng2022moreau}, and minimax optimization \cite{ZhangX0L20}, and ADMMs \cite{li2022riemannian}. We define the augmented Lagrangian function of Problem (\ref{eq:main:equality}) as follows:
\beq \label{eq:L}
\mathcal{L}(\X,\y,\z,\beta) = \underbrace{f(\X)  + \la \z,\AA(\X) - \y \ra + \tfrac{\beta}{2}\|\AA(\X) - \y\|_{2}^2}_{~\triangleq~\SS(\X,\y,\z,\beta)} - g(\X) + h_{\tau/\beta}(\y) + \iota_{\MM}(\X).
\eeq
\noi Here, $\z$ is the dual variable for the equality constraint, $\mu\triangleq \tau/\beta$ is the smoothing parameter linked to the function $h(\y)$, $\beta$ is the penalty parameter associated with the equality constraint, and $\iota_{\MM}(\X)$ is the indicator function of the set $\MM$.

In simple terms, {\sf OADMM} updates are performed by minimizing the augmented Lagrangian function $\mathcal{L}(\X,\y,\z,\beta)$ over the primal variables $\{\X^t,\y^t\}$ at each iteration, while keeping all other primal and dual variables fixed. The dual variables are updated using gradient ascent on the dual problem.

For updating the primal variable $\X$, we use different strategies, resulting in distinct variants of {\sf OADMM}. We first observe that the function $\mathcal{S}(\X,\y^t,\z^t,\beta^t)$ is $\ell(\beta^t)$-smooth \textit{w.r.t.} $\X$, where $\ell(\beta^t) \triangleq \beta^t \Aup^2 + L_f$. In {\sf OADMM-EP}, we adopt a proximal linearized method based on Euclidean projection \cite{lai2014splitting}, while in {\sf OADMM-RR}, we apply line-search methods on the Stiefel manifold \cite{liu2016quadratic}.

\begin{algorithm}[!t]
\DontPrintSemicolon
\caption{ {\bf {\sf OADMM}: The Proposed ADMM for Solving Problem (\ref{eq:main}).}  } \label{alg:main}
\textbf{Initialization:} \\
\quad Choose $\{\X^0,\y^0,\z^0\}$. Choose $p,\xi\in (0,1)$, $\theta \in (1,\infty)$, $\sigma\in[1,2)$.

\quad Choose $\tau\in [\tfrac{4}{2-\sigma},\infty)$. Choose $\beta^0$ with $\beta^0\geq 2\tau W_h$.

\quad For {\sf OADMM-EP}, choose $\alpha \in [0, \frac{\theta-1}{(\theta+1) (\xi+2)})$.

\quad For {\sf OADMM-RR}, choose $\alpha=0$, $\rho \in (0,\infty)$, $\gamma \in (0,1)$, $\delta \in (0,\tfrac{1}{\max(1,2\rho)})$.

\For{$t$ from 0 to $T$}{
S1) Set $\beta^t = \beta^0 (1 + \xi t^p)$.

S2) Update the primal variable $\X$:

\If{{{\sf OADMM-EP}}}{

Set $\X_{\cc}^{t} = \X^{t}+\alpha (\X^{t}  - \X^{t-1})$, $\G^t \in \nabla_{\X} \SS(\X_{\cc}^t,\y^t,\z^t)-\partial g(\X^t)$.

$\X^{t+1} \in \arg \min_{\X \in\MM} \la \X-\X^t,\G^t\ra + \tfrac{\theta \ell(\beta^t)}{2}\| \X - \X_{\cc}^t \|_{\fro}^2$, where $\ell(\beta^t) \triangleq \beta^t \Aup^2 + L_f$.
}

\If{{{\sf OADMM-RR}}}
{

Set $\G^t \in \nabla_{\X} \SS(\X^t,\y^t,\z^t)-\partial g(\X^t)$. Set $\GGG^t_{\rho} \triangleq \G^t-\rho\X^t [\G^t] \trans\X^t - (1-\rho)\X^t[\X^t]\trans \G^t$. Set $b^t\in (\underline{b},\overline{b})$ as the BB step size, where $\underline{b},\overline{b} \in(0,\infty)$. Set $\X^{t+1} = \Retr_{\X^{t}}( -\eta^t \GGG^t_{\rho})$. Here, $\eta^t\triangleq \tfrac{b^t \gamma^{j}}{\beta^t}$, and $j\in\{0,1,\ldots\}$ is the smallest integer such that: $\dot{\mathcal{L}}(\Retr_{\X^{t}}(-\eta^t \GGG^t_{\rho}) ) - \dot{\mathcal{L}}(\X^{t})\leq  - \delta \eta^t\|\GGG^t_{\rho}\|_{\fro}^2$, where $\dot{\mathcal{L}}(\X)\triangleq L(\X,\y^t,\z^t,\beta^t)$.
}

S3) Update the primal variable $\y$: $\y^{t+1} = \arg \min_{\y}  h_{\mu^t}(\y) + \frac{ \beta^t}{ 2  }\| \y - \mathbf{b} \|_{2}^2$, where $\mu^t =  \tau / \beta^t$, $\mathbf{b} \triangleq \y^t - \tfrac{1}{\beta^t}\nabla_{\y} \SS(\X^{t+1},\y^t,\z^t,\beta^t)$. It can be solved using Lemma \ref{lemma:smoothing:problem:prox} as: $\y^{t+1}= \tfrac{\breve{\y}^{t+1}+\mu^t\beta^t \b}{1+\mu^t\beta^t}$, where $\breve{\y}^{t+1} = \prox(\b;\mu^t+1/\beta^t)$.

%This is equivalent to the following problem:
%\beq % $\Leftrightarrow
%\ts (\y^{t+1},\breve{\y}^{t+1})= \arg \min_{\y,\breve{\y}} h(\breve{\y}) + \tfrac{1}{2\mu^t} \|\breve{\y} - \y\|_{2}^2 + \tfrac{ \beta^t}{ 2  }\| \y - \mathbf{b} \|_{2}^2. \label{eq:subprob:y}
%\eeq
S4) Update the dual variable $\z$: $\z^{t+1} = \z^t + \sigma \beta^t (\AA (\X^{t+1}) - \y^{t+1}) $
}
\end{algorithm}

%\begin{enumerate}[label=\textbf{(\alph*)}, leftmargin=17pt, itemsep=0pt, topsep=1pt, parsep=0pt, partopsep=0pt]

We detail iteration steps of {\sf OADMM} in Algorithm \ref{alg:main}, and have the following remarks. (\bfit{i}) To achieve possible faster dual convergence, we apply an over-relaxation step size with $\sigma \in (1,2)$ for updating the dual variable $\z$, as suggested by previous studies \cite{goncalves2017convergence, yang2017alternating, li2016majorized, li2023convergence}. (\bfit{ii}) To accelerate primal convergence in {\sf OADMM-EP}, we incorporate a Nesterov extrapolation strategy with parameter $\alpha \in (0,1)$. (\bfit{iii}) To enhance primal convergence in {\sf OADMM-RR}, we use a Monotone Barzilai-Borwein (MBB) strategy \cite{wen2013feasible} with a dynamically adjusted parameter $b^t$ to capture the problem's curvature \footnote{Following \cite{wen2013feasible}, one can set $b^t = { \la \mathbf{S}^{t}, \mathbf{S}^t \ra } / { \la \mathbf{S}^{t}, \mathbf{Z}^t \ra }$ or $b^t = { \la \mathbf{S}^{t}, \mathbf{Z}^t \ra }/{ \la \mathbf{Z}^{t}, \mathbf{Z}^t \ra }$, where $\mathbf{S}^t = \X^t-\X^{t-1}$ and $\mathbf{Z}^t = \GGG_{1}^{t-1} - \GGG_{1}^t$, with $\GGG_{1}^t$ being the Riemannian gradient.}. The parameters $\{\gamma,\delta\}$ represent the decay rate and sufficient decrease parameter, commonly used in line search procedures \cite{chen2020proximal}. (\bfit{iv}) The $\X$-subproblem is solved as: $\X^{t+1} = \arg \min_{\X\in\MM}\|\X - \X'\|_{\fro}^2=\dot{\U}\dot{\V}\trans$, where $\X'=\X_{\cc}^t - \G^t/(\theta \ell(\beta^t))$, and $\dot{\U}\diag(\dot{\x})\dot{\V}\trans = \X'$ is the using singular value decomposition of $\X'$. (\bfit{v}) For practical implementation, we recommend the following default parameters: $p=1/3$, $\theta=1.01$, $\sigma=1.1$, $\rho=1$, $\gamma=1/2$, $\delta=10^{-3}$, $\xi=1$, $\alpha=\frac{\theta-1}{(\theta+1) (\xi+2)}-10^{-12}$.

%Importantly, a line search strategy is employed to ensure monotonic decrease in the augmented Lagrangian function \cite{liu2016quadratic, li2015accelerated}.
%\end{enumerate}

\section{Oracle Complexity}
\label{sect:iterC}

This section details the oracle complexity of Algorithm \ref{alg:main}.

\textbf{Notations.} We define $\varepsilon_z=\tfrac{1}{4}$, $\varepsilon_{\beta}>0$, $\varepsilon_y\triangleq \tfrac{2-\sigma}{8}$, $\ddot{\sigma} \triangleq \tfrac{12\sigma^2}{p(2-\sigma)^2} C^2_h$, $c\triangleq \varepsilon_{\beta} + \tau C_h^2 + \tfrac{2}{\sigma}\ddot{\sigma}$. We define a sequence associated with the potential function (or Lyapunov function) for all $t\geq 1$, as follows:
\beq
\Theta^t&\triangleq& L(\X^{t},\y^t,\z^t,\beta^{t})+  c/{\beta^{t}} + \PPP^{t}+ \DDD^{t}  ,
\eeq
\noi where $\PPP^{t}\triangleq \tfrac{     \alpha ( \theta + 1   ) \ell(\beta^t)  }{2} \| \X^{t}  - \X^{t-1}\|_{\fro}^2$, $\DDD^t\triangleq 2\beta^{t-1} \tfrac{\sigma-1}{2-\sigma} \| \sigma (\AA(\X^t)-\y^t)\|_2^2$. We define $\mathcal{B}_t\triangleq\sqrt{ (\tfrac{1}{\beta^{t-1}} - \tfrac{1}{\beta^{t}})\tfrac{1}{\beta^{t-1}}}$, $\mathcal{Y}_t\triangleq\|\y^t - \y^{t-1}\|$, and $\mathcal{Z}_t\triangleq\|\AA(\X^{t}) - \y^t\|$. For {\sf OADMM-EP}, we define $\mathcal{X}_t\triangleq\|\X^t - \X^{t-1}\|_{\fro}$, and for {\sf OADMM-RR}, $\mathcal{X}_t\triangleq\|\tfrac{1}{\beta^t} \GGG^{t-1}_{1/2}\|_{\fro}$.

%We define: $e^{t}\triangleq \sqrt{ (\tfrac{1}{\beta^t} - \tfrac{1}{\beta^{t+1}})\tfrac{1}{\beta^t}}  + \|\y^{t}-\y^{t-1}\| + \|\AA(\X^{t})-\y^{t}\|+ {\tiny \left\{
%  \begin{array}{ll}
%    \|\X^{t}-\X^{t-1}\|_{\fro}, & \hbox{{\sf OADMM-EP};} \\
%   \|\tfrac{1}{\beta^t} \GGG^{t-1}_{1/2}\|_{\fro} , & \hbox{{\sf OADMM-RR}.}
%  \end{array}
%\right.}$

We have the following useful lemma, derived using the first-order optimality condition of $\mathbf{y}^{t+1}$.

\begin{lemma} \label{lemma:bounding:dual}
(Proof in Section \ref{app:lemma:bounding:dual}, \rm{Bounding Dual using Primal}) We have: (\bfit{a}) $\forall t\geq 0,\,\z^{t} - \tfrac{1}{\sigma} (\z^{t} - \z^{t+1}) =  \nabla h_{\mu^{t}} (\y^{t+1}) \in \partial h(\breve{\y}^{t+1})$. (\bfit{b})  $\forall t\geq 1,\,\| \z^{t+1}-\z^t\|_{2}^2 \leq \tfrac{\sigma-1}{2-\sigma} (\|\z^{t}-\z^{t-1}\|_2^2 - \|\z^{t+1}-\z^t\|_2^2) + \dot{\sigma}(\beta^t)^2 \|\y^{t+1}-\y^{t}\|_2^2 + \ddot{\sigma} (\tfrac{\beta^0}{\beta^t}-\tfrac{\beta^0}{\beta^{t+1}})$, where $\dot{\sigma}\triangleq \tfrac{2\sigma^2}{(2-\sigma)^2}  \tfrac{1}{\tau^2}$.

\end{lemma}

\begin{remark}
For {\sf OADMM-RR}, we set $\alpha=0$, resulting in $\PPP^{t}=0$ for all $t$. With the choice $\sigma=1$, we have: $\nabla h_{\mu^{t-1}} (\y^{t})=\z^{t}$, and $\|\z^{t+1}-\z^t\|\leq\|\nabla h_{\mu^{t}}(\y^{t+1})-\nabla h_{\mu^{t-1}}(\y^{t})\|$.
\end{remark}
% (\bfit{ii}) Claim (\bfit{a}) of Lemma \ref{lemma:bounding:dual} implies that the optimality condition is independent of the choice of $\alpha_y$. Therefore, although a Nesterov extrapolation strategy is used to update variable $\y$, it reverts to the unaccelerated version with $\alpha_y=0$. This phenomenon is expected, as the original minimization problem \textit{w.r.t.} $\y$ in Equation \ref{eq:L} has a closed-form solution, making the extrapolation strategy unnecessary.

\begin{lemma} \label{lemma:smooth:LL}

(Proof in Appendix \ref{app:lemma:smooth:LL}) (\bfit{a}) It holds that $\beta^{t+1}\leq \beta^t (1+\xi)$. (\bfit{b}) There exists constant $\{\ellup,\elldown\}$ such that $\beta^t\elldown\leq \ell(\beta^t) \leq \beta^t\ellup$.
\end{lemma}

The subsequent lemma demonstrates that the sequence $\{\Theta^t\}_{t=1}^{\infty}$ is always lower bounded.

\begin{lemma} \label{lemma:bound:solution}
(Proof in Section \ref{app:lemma:bound:solution}) For all $t\geq 1$, there exists constants $\{\overline{\rm{X}},\overline{\rm{z}},\overline{\rm{y}},\underline{\rm{\Theta}}\}$ such that $\|\X^t\|_{\fro}\leq \overline{\rm{X}}$, $\|\z^t\| \leq \overline{\rm{z}}$, $\|\y^t\|\leq \overline{\rm{y}}$, and $\Theta^t\geq \underline{\rm{\Theta}}$.
\end{lemma}

The following lemma is useful for our subsequent analysis, applicable to both {\sf OADMM-EP} and {\sf OADMM-RR}.

\begin{lemma}\label{lemma:dec:non:x}
(Proof in Appendix \ref{app:lemma:dec:non:x}, \rm{Sufficient Decrease for Variables} $\{\y,\z,\beta,\mu\})$ We have $\varepsilon_z \beta^t \ZZ_{t+1}^2 +\varepsilon_y \beta^t  \YY_{t+1}^2 + \varepsilon_{\beta} \beta^t \BB_{t+1}^2 + L(\X^{t+1},\y^{t+1},\z^{t+1},\beta^{t+1}) - L(\X^{t},\y^t,\z^t,\beta^{t})  + c/{\beta^{t+1}}-c/{\beta^{t}}+ \DDD^{t+1}-\DDD^t\leq \mathfrak{X}$, where $\mathfrak{X}\triangleq L(\X^{t+1},\y^t,\z^t,\beta^{t}) - L(\X^t,\y^t,\z^t,\beta^{t})$.
\end{lemma}

In the remaining content of this section, we provide separate analyses for {\sf OADMM-EP} and {\sf OADMM-RR}.

\subsection{Analysis for {\sf OADMM-EP}}

Using the optimality condition of $\mathbf{X}^{t+1}$, we derive the following lemma.

\begin{lemma}\label{lemma:dec:x:proj}

(Proof in Appendix \ref{app:lemma:dec:x:proj}, \rm{Sufficient Decrease for Variable} $\X$) We define $\varepsilon_x \triangleq \tfrac{1}{2} \varepsilon_x'\elldown$, where $\varepsilon_x'\triangleq  \theta-1 - \alpha (2+\xi) (1+ \theta)>0$. We have $\mathfrak{X} \leq - \varepsilon_x \beta^t \XX_{t+1}^2 +  \PPP^{t} - \PPP^{t+1} $.
\end{lemma}

Combining the results from Lemmas \ref{lemma:dec:non:x}, and \ref{lemma:dec:x:proj}, we arrive at the following lemma.

\begin{lemma} \label{lemma:bound:Gamma:Pj}
(Proof in Appendix \ref{app:lemma:bound:Gamma:Pj}) We define $\mathcal{X}_t\triangleq\|\X^t - \X^{t-1}\|_{\fro}$. We have:

\begin{enumerate}[label=\textbf{(\alph*)}, leftmargin=20pt, itemsep=1pt, topsep=1pt, parsep=0pt, partopsep=0pt]

\item  $\beta^t \{\varepsilon_{\beta} \BB_{t+1}^2 + \varepsilon_z \ZZ_{t+1}^2 + \varepsilon_y \YY_{t+1}^2 + \varepsilon_x \XX_{t+1}^2 \}   \leq \Theta^t - \Theta^{t+1}$.

    \item  $\tfrac{1}{T}\sum_{t=1}^T \beta^t [\BB_{t+1} + \ZZ_{t+1} + \YY_{t+1} + \XX_{t+1}]\leq \OO(T^{(p-1)/2})$.
\end{enumerate}
%\{  \|\AA(\X^{t+1}) - \y^{t+1}\|_{2}^2 +  \|\X^{t+1}-\X^{t} \|_{\fro}^2 + \|\y^{t+1}-\y^{t} \|_{2}^2 \} \leq (\Theta^1 - \underline{\rm{\Theta}} )/\min(\varepsilon_y,\varepsilon_x,\varepsilon_z )$.
\end{lemma}

Finally, we have the following theorem regarding the oracle complexity of {\sf OADMM-EP}.

\begin{theorem} \label{theorem:OADMM:Pj}
(Proof in Appendix \ref{app:theorem:OADMM:Pj}) Let $p=1/3$. We have: $\frac{1}{T} \sum_{t=1}^T \Crit(\X^{t+1},\breve{\y}^{t+1},\z^{t+1}) \leq \OO(T^{(p-1)/2}) + \OO(T^{-p}) = \mathcal{O}(T^{-1/3})$. In other words, there exists $\bar{t}\leq T$ such that:
$\Crit(\X^{\bar{t}+1},\breve{\y}^{\bar{t}+1},\z^{\bar{t}+1})\leq \epsilon$, provided that $T\geq \mathcal{O}(1/\epsilon^3)$.
\end{theorem}

\begin{remark}
(\bfit{i}) We notice that $\frac{1}{T} \sum_{t=1}^T \Crit(\X^{t+1},\breve{\y}^{t+1},\z^{t+1}) \leq \OO(T^{(p-1)/2}) + \OO(T^{-p})$. Minimizing the worse-case complexity of the right-hand side of this inequality with respect to $p$ yields: $\arg \min_{p\in(0,1)} \max((p-1)/2,-p)=1/3$. Thus, setting $p=1/3$  achieves the optimal trade-off between the two terms, leading to the best complexity bounds. (\bfit{ii}) The oracle complexity of {\sf OADMM-EP} matches the best-known complexities currently available to date \cite{beck2023dynamic,BohmW21}.
\end{remark}
%utilizes the same exponent parameter as in the penalty update rule $\beta^t = \beta^0 (1+\xi t^p)$, where $p=1/3$. Its

\subsection{Analysis for {\sf OADMM-RR}}

%stopping criterion
%(\tfrac{1}{\max(1,2\rho)} -\delta) / (\tfrac{1}{2}\overline{b}  \ddot{k} \overline{g} + \tfrac{1}{2}\overline{b} \dot{k} \ellup )\textbf{}
Using the properties of the line search procedure for updating the variable $\mathbf{X}^{t+1}$, we deduce the following lemma.

\begin{lemma}\label{lemma:dec:x}
(Proof in Appendix \ref{app:lemma:dec:x}, \rm{Sufficient Decrease for Variable} $\X$) We define $\varepsilon_x \triangleq  \delta \overline{\gamma} \gamma \underline{b} \min(1,2 \rho)^2 >0$, where $\overline{\gamma}\triangleq 2(1/{\max(1,2\rho)} -\delta) / (\ellup    \dot{k} \overline{b} + \overline{g} \ddot{k} \overline{b}/\beta^0)>0$. We have: (\bfit{a}) For any $t\geq 0$, if $j$ is large enough such that $\gamma^j \in (0, \overline{\gamma})$, then the condition of the line search procedure is satisfied. (\bfit{b}) It follows that: $L(\X^{t+1},\y^t,\z^t,\beta^t) - L(\X^t,\y^t,\z^t,\beta^t) \leq - \tfrac{\varepsilon_x}{\beta^t}\|\GGG^t_{1/2}\|^2_{\fro}$. Here, $\overline{g}$ is a constant that $\|\G^t\|_{\fro}\leq\overline{g}$, $\{\dot{k},\ddot{k}\}$ are defined in Lemma \ref{lemma:M12}, and $\{\rho,\gamma,\delta,\overline{b},\underline{b}\}$ are defined in Algorithm \ref{alg:main}.
\end{lemma}

\begin{remark}

By Lemma \ref{lemma:dec:x}(\bfit{a}), since $\overline{\gamma}$ is a universal constant and $\gamma^j$ decreases exponentially, the line search procedure of {\sf OADMM-RR} will terminate in $\log(\overline{\gamma}) / \log(\gamma) +1 = \mathcal{O}(1)$ time.

\end{remark}

Combining the results from Lemmas \ref{lemma:dec:non:x}, and \ref{lemma:dec:x}, we obtain the following lemma.

\begin{lemma} \label{lemma:bound:Gamma}
(Proof in Appendix \ref{app:lemma:bound:Gamma}) We define $\mathcal{X}_t\triangleq\|\tfrac{1}{\beta^t} \GGG^{t-1}_{1/2}\|_{\fro}$. We have:

\begin{enumerate}[label=\textbf{(\alph*)}, leftmargin=20pt, itemsep=1pt, topsep=1pt, parsep=0pt, partopsep=0pt]

\item $\beta^t \{\varepsilon_{\beta} \BB_{t+1}^2 + \varepsilon_z \ZZ_{t+1}^2 + \varepsilon_y \YY_{t+1}^2 + \varepsilon_x \XX_{t+1}^2 \} \leq \Theta^t - \Theta^{t+1}$.

\item $\tfrac{1}{T}\sum_{t=1}^T \beta^t [\BB_{t+1} + \ZZ_{t+1} + \YY_{t+1} + \XX_{t+1}] \leq \OO(T^{(p-1)/2})$.

\end{enumerate}

\end{lemma}

Finally, we derive the following theorem on the oracle complexity of {\sf OADMM-RR}.

\begin{theorem}\label{theorem:OADMM:R}
(Proof in Appendix \ref{app:theorem:OADMM:R}) Let $p=1/3$. We have: $\frac{1}{T} \sum_{t=1}^T \Crit(\X^{t+1},\breve{\y}^{t+1},\z^{t+1}) \leq \OO(T^{(p-1)/2}) + \OO(T^{-p}) = \OO(T^{-1/3})$. In other words, there exists $\bar{t}\leq T$ such that:
$\Crit(\X^{\bar{t}+1},\breve{\y}^{\bar{t}+1},\z^{\bar{t}+1})\leq \epsilon$, provided that $T\geq \mathcal{O}(1/\epsilon^3)$.

\end{theorem}

\begin{remark}
Theorem \ref{theorem:OADMM:R} mirrors Theorem \ref{theorem:OADMM:Pj}, and {\sf OADMM-RR} shares the same oracle complexity as {\sf OADMM-EP}.
\end{remark}

%\subsection{On Solving the Sub problem}
%It is
%\beq
%\y^{t+1} &\in& \arg \min_{\y} h_{\mu^t}(\y) + \la \y-\y^t, \beta (\y^t-\AA\x^{t+1}) - \z^t \ra + \frac{\theta_y \LL_y}{2} \|\y - \y^t\|_2^2
%\eeq

%\beq
%\y^{t+1} &\in& \arg \min_{\y} h_{\mu}(\y) + \frac{\alpha}{2} \|\y - \a\|_2^2
%\eeq
%\noi where $\mu=\mu^t$, $\alpha=\theta_y \LL_y$, and $\a= \y^t  -  [\beta (\y^t-\AA\x^{t+1}) - \z^t] / \alpha$. It is equivalent to the following problem:
%\beq
%\min_{\y,\v} \frac{\alpha}{2} \|\y-\a\|_2^2 + h(\v) + \frac{1}{2\mu}\|\v-\y\|_2^2
%\eeq
%Setting the gradient of $\y$, we have:
%\beq
%\alpha (\y-\a) + \frac{1}{\mu} (\y-\v) = 0
%\eeq
%We have:
%\beq
%\y = \frac{\alpha \a + \frac{1}{\mu} \v }{ \alpha + \frac{1}{\mu}}
%\eeq
%\beq
%\min_{\v} \frac{\alpha}{2} \|  \frac{\alpha \a + \frac{1}{\mu} \v }{ \alpha + \frac{1}{\mu}} -\a\|_2^2 + h(\v) + \frac{1}{2\mu}\| \frac{\alpha \a + \frac{1}{\mu} \v }{ \alpha + \frac{1}{\mu}} - \v\|_2^2
%\eeq

%\beq
%\min_{\v} \frac{\alpha}{2} \|  \frac{\alpha \a + \frac{1}{\mu} \v }{ \alpha + \frac{1}{\mu}} -\a\|_2^2 + h(\v) + \frac{1}{2\mu}\| \frac{\alpha \a + \frac{1}{\mu} \v }{ \alpha + \frac{1}{\mu}} - \v\|_2^2
%\eeq

%\subsection{Iteration Complexity for Both {\sf OADMM-EP} and {\sf OADMM-RR}}

%\input{sect6strong}
\section{Convergence Rate}
\label{sect:strong:KL}

This section provides convergence rate of {\sf OADMM-EP} and {\sf OADMM-RR}. Our analyses are based on a non-convex analysis tool called KL inequality \cite{Attouch2010,Bolte2014,li2015accelerated,li2023convergence}.

For simplicity, we only consider the case where $\alpha=0$, $\sigma=1$, and $g(\X)=0$. We only focus on $p=1/3$, as it gives the best oracle complexity.

We define the Lyapunov function as: $\Theta(\X,\y,\z,\beta) \triangleq  L(\X,\y,\z,\beta)  + c/{\beta}$. We define $\ww \triangleq \{\X,\y,\z,\beta\}$, $\ww^t \triangleq \{\X^t,\y^t,\z^t,\beta^t\}$. Consequently, $\Theta^{t} = \Theta(\ww^t)$. We denote $\ww^{\infty}$ as a limiting point of Algorithm \ref{alg:main}. We make the following additional assumption.

\begin{assumption} \label{assumption:KL}
The function $\Theta(\ww)$ is a KL function \textit{w.r.t.} $\ww$.

\end{assumption}

\begin{proposition} \label{proposition:KL}
({\rm{Kurdyka-{\L}ojasiewicz Inequality}} \cite{Attouch2010}). Consider a semi-algebraic function $\Theta(\ww)$ with $\ww \in \dom (\Theta)$. There exist constants $\tilde{\eta}\in(0,+\infty)$, $\tilde{\sigma}\in[0,1)$, a neighborhood $\Upsilon$ of $\ww^{\infty}$, and a continuous and concave desingularization function $\varphi(s)\triangleq \tilde{c} s ^{1-\tilde{\sigma}}$ with $\tilde{c}>0$ and $s\in [0,\tilde{\eta})$ such that, for all $\ww\in \Upsilon$ satisfying $\Theta(\ww) - \Theta(\ww^{\infty}) \in (0,\tilde{\eta})$, it holds that: $\varphi' ( \Theta(\ww) - \Theta(\ww^{\infty}) ) \cdot \dist(\zero,\partial \Theta(\ww))\geq 1$.

\end{proposition}

\begin{remark}
Semi-algebraic functions, including real polynomial functions, finite combinations, and indicator functions of semi-algebraic sets, commonly exhibit the KL property and find extensive use in applications \cite{Attouch2010}.

\end{remark}

We present the following lemma regarding subgradient bounds for each iteration.

\begin{lemma} \label{lemma:subgrad:bound:P}
(Proof in Section \ref{app:lemma:subgrad:bound:P}, \rm{Subgradient Bounds}) For both {\sf OADMM-EP} and {\sf OADMM-RR}, there exists a constant $K>0$ such that: $\dist(\zero,\partial \Theta(\ww^t)) \leq K \beta^t  (\BB_t + \XX_t + \YY_t+ \ZZ_t)$.
\end{lemma}

\begin{remark}
Lemma \ref{lemma:subgrad:bound:P} significantly differs from prior work that used a constant penalty due to the crucial role played by the increasing penalty.
\end{remark}

%Semi-algebraic functions constitute a category of functions that exhibit the K$\L$ property. They find extensive use in various applications, encompassing real polynomial functions, finite combinations and products of semi-algebraic functions, and indicator functions of semi-algebraic sets.

%For Condition $\III$, we define the Lyapunov function as: $ \frac{ C_{\aa} }{\beta}\| \sigma \beta  \AA_n\trans \r  \|_2^2  + \upsilon$. For Condition $\AAA$, we define the Lyapunov function as: $\Theta(\x,\z,\x',\x'';\beta) \triangleq  \frac{ C_{\aa} }{\beta}\| \sigma \beta  \AA_n\trans \r + \sigma \u \|_2^2  + \upsilon$.

The following theorem establishes a finite length property of {\sf OADMM}.

\begin{theorem} \label{theorem:finite:length:P}
(Proof in Section \ref{app:theorem:finite:length:P}, {\rm{A Finite Length Property}}) We define $d^t\triangleq\sum_{i=t}^{\infty} e^{i+1}$, where $e^{t}\triangleq \BB_t + \XX_t + \YY_t+ \ZZ_t$. We define $\varphi^t \triangleq \varphi( \Theta(\ww^{t}) -  \Theta(\ww^{\infty}))$, where $\varphi(\cdot)$ is the desingularization function defined in Assumption \ref{assumption:KL}. We have the following results for both {\sf OADMM-EP} and {\sf OADMM-RR}.

\begin{enumerate}[label=\textbf{(\alph*)}, leftmargin=20pt, itemsep=1pt, topsep=1pt, parsep=0pt, partopsep=0pt]

%\item If $\tilde{\sigma}\in[0,\tfrac{1}{4}]$, then the sequence $\{\X^t\}_{t=1}^{\infty}$ converges in a finite number of steps.

\item $(e^{t+1})^2 \leq  (\varphi^t - \varphi^{t+1}) K' e^t $, where $K' = \tfrac{ 4K }{\min(\varepsilon_z,\varepsilon_y, \varepsilon_x, \varepsilon_{\beta} ) }$, and $K$ is defined in Lemma \ref{lemma:subgrad:bound:P}.

\item It holds that $\forall t\geq 1,\,d^t\leq e^{t} + 2 K' \varphi^t$. The sequence $\{\ww^t\}_{t=1}^{\infty}$ has the finite length property that $d^t \leq e^1  + 2 K' \varphi^1 < +\infty$.

\end{enumerate}

\end{theorem}

\begin{remark} The finite length property in Theorem \ref{theorem:finite:length:P} represents much stronger convergence results compared to those outlined in Theorems \ref{theorem:OADMM:Pj} and \ref{theorem:OADMM:R}.
\end{remark}

We prove a lemma demonstrating that the convergence of $d^t \triangleq \sum_{i=t}^{\infty} e^{i+1}$ is sufficient to establish the convergence of $\|\X^{t} - \X^\infty\|_{\fro}$.

\begin{lemma} \label{lemma:rate:pre}
(Proof in Section \ref{app:lemma:rate:pre}) For both {\sf OADMM-EP} and {\sf OADMM-RR}, we have:

\begin{enumerate}[label=\textbf{(\alph*)}, leftmargin=20pt, itemsep=1pt, topsep=1pt, parsep=0pt, partopsep=0pt]

%\item If $\tilde{\sigma}\in[0,\tfrac{1}{4}]$, then the sequence $\{\X^t\}_{t=1}^{\infty}$ converges in a finite number of steps.

\item There exists a constant $\varpi$ such that $\|\X^{t} - \X^\infty\|_{\fro} \leq \varpi d^t$.

    %    (\bfit{b})
\item
We have $d^{t} \leq \ts d^{t-1} - d^{t} + \nu [ \beta^t (d^{t-1} - d^{t}) ]^{\frac{1-\tilde{\sigma}}{\tilde{\sigma}}}$, where $\nu$ is some universal constant.

\end{enumerate}
\end{lemma}

Finally, we establish the convergence rate of {\sf OADMM} with exploiting the KL exponent $\tilde{\sigma}$.

\clearpage
\begin{theorem}\label{theorem:KL:rate:Exponent:P}

(Proof in Section \ref{app:theorem:KL:rate:Exponent:P}, {\rm{Convergence Rate}}) We fix $p=1/3$. We have:
\begin{enumerate}[label=\textbf{(\alph*)}, leftmargin=20pt, itemsep=1pt, topsep=1pt, parsep=0pt, partopsep=0pt]

%\item If $\tilde{\sigma}\in(0,\tfrac{1}{4}]$, then we have $\|\X^{t} - \X^\infty\|_{\fro}\leq \mathcal{O}(t^{-\zeta})$, where $\zeta = \tfrac{2}{3}\cdot \frac{1-\tilde{\sigma}}{\tilde{\sigma}} \in [2,\infty]$.

%\item If $\tilde{\sigma}\in[0,\tfrac{1}{4}]$, then the sequence $\{\X^t\}_{t=1}^{\infty}$ converges in a finite number of steps.

%

\item If $\tilde{\sigma}\in(0,\tfrac{1}{2}]$, then we have: $\|\X^{t} - \X^\infty\|_{\fro}\leq \mathcal{O}(t^{-\zeta})$, where $\zeta = \tfrac{2}{3}\cdot \frac{1-\tilde{\sigma}}{\tilde{\sigma}} \in [\tfrac{2}{3},\infty]$.

\item If $\tilde{\sigma}\in (\tfrac{1}{2},1)$, then we have: $\|\X^{t} - \X^\infty\|_{\fro}\leq  \mathcal{O}(t^{-\zeta})$, where $\zeta = \tfrac{2}{3}\cdot \tfrac{1-\tilde{\sigma}}{2\tilde{\sigma}-1}\in (0,\infty)$.

\item If $\tilde{\sigma}\in(\tfrac{1}{4},\tfrac{1}{2}]$, then we have $\|\X^{t} - \X^\infty\|_{\fro}\leq \mathcal{O}(1/\exp(t^{\zeta}))$, where $\zeta =1- \tfrac{ p(1-\tilde{\sigma})}{\tilde{\sigma}}  \in (0,\tfrac{2}{3}]$.

%\item In If $\tilde{\sigma}\in (\tfrac{1}{2},1)$, then we have: $\|\X^{t} - \X^\infty\|_{\fro}\leq  \mathcal{O}(t^{-\zeta})$, where $\zeta = \tfrac{2}{3}\cdot \tfrac{1-\tilde{\sigma}}{2\tilde{\sigma}-1}\in (0,\infty)$.

%\item If $\tilde{\sigma}\in [\tfrac{1}{2},1)$, then we have: $\|\X^{t} - \X^\infty\|_{\fro}\leq   \mathcal{O}(1 / (t^{ (1-p)/\zeta} ) )$, where $\zeta = \frac{\tilde{\sigma}}{1-\tilde{\sigma}} -1 \in (0,\infty)$.

\end{enumerate}

\end{theorem}

\begin{remark}

(\bfit{i}) To the best of our knowledge, Theorem \ref{theorem:KL:rate:Exponent:P} represents the first non-ergodic convergence rate for solving this class of nonconvex and nonsmooth problem in Problem (\ref{eq:main}). It is worth noting that the work of \cite{li2023convergence} establishes a non-ergodic convergence rate for subgradient methods with diminishing stepsizes by further exploring the KL exponent. (\bfit{ii}) Under the KL inequality assumption, with the desingularizing function chosen in the form of $\varphi(s)\triangleq \tilde{c} s ^{1-\tilde{\sigma}}$ with $\tilde{\sigma}\in(0,1)$, {\sf OADMM} converges with a polynomial convergence rate when $\tilde{\sigma}\in(0,1)$, and converges with a super-exponential rate when $\tilde{\sigma}\in(\tfrac{1}{4},\tfrac{1}{2}]$ for the gap $\|\X^{t} - \X^\infty\|_{\fro}$. (\bfit{iii}) Our result generalizes the classical findings of \cite{Attouch2010,Bolte2014}, which characterize the convergence rate of proximal gradient methods for a specific class of nonconvex composite optimization problems.

%Notably, super-exponential convergence is faster than polynomial convergence.  
\end{remark}

 %\bfit{(i)}

%\input{sect7exp}

 \section{Applications and Numerical Experiments}

%This section evaluates the effectiveness of {\sf OADMM} on the sparse PCA problem.

In this section, we assess the effectiveness of the proposed algorithm {\sf OADMM} on the sparse PCA problem by comparing it against existing non-convex, non-smooth optimization algorithms.

\noi $\blacktriangleright$ \textbf{Application to Sparse PCA}. Sparse PCA is a method to produce modified principal components with sparse loadings, which helps reduce model complexity and increase model interpretation \cite{chen2016augmented}. It can be formulated as:
\beq
\min_{\X \in \Rn^{n\times r}}  \tfrac{1}{2\dot{m}} \| \X\X\trans \D - \D \|_{\fro}^2 + \dot{\rho} (\|\X\|_1 -  \|\X\|_{[k]}),~s.t.~\X\trans\X=\I_r,\label{eq:SPCA:DC}
\eeq
where $\D \in \Rn^{n\times \dot{m}}$ is the data matrix, $\dot{m}$ is the number of data points, and $\|\X\|_{[k]}$ is the $\ell_1$ norm the the $k$ largest (in magnitude) elements of the matrix $\X$. Here, we consider the DC $\ell_1$-largest-$k$ function \cite{GotohTT18} to induce sparsity in the solution. One advantage of this model is that when $\dot{\rho}$ is sufficient large, we have $\|\X\|_1 \thickapprox  \|\X\|_{[k]}$, leading to a $k$-sparsity solution $\X$. Problem (\ref{eq:SPCA:DC}) coincides with the optimization model in Problem (\ref{eq:main}), where $f(\X)=\tfrac{1}{2\dot{m}} \| \X\X\trans \D - \D \|_{\fro}^2$, $f(\X)=\dot{\rho} \|\X\|_{[k]}$, and $h(\AA(\X)) =\dot{\rho} \|\X\|_1$.

\noi $\blacktriangleright$ \textbf{Compared Methods}. We compare {\sf OADMM-EP} and {\sf OADMM-RR} against four state-of-the-art optimization algorithms: (\bfit{i}) {\sf RADMM}: ADMM using Riemannian retraction with fixed and small stepsizes \cite{li2022riemannian}, tested with two different penalty parameters $\forall t,\,\beta^t\in\{100,10000\}$, leading to two variants: RADMM-I and RADMM-II. (\bfit{ii}) {\sf SPGM-EP}: Smoothing Proximal Gradient Method using Euclidean projection \cite{BohmW21}. (\bfit{iii}) {\sf SPGM-EP}: SPGM utilizing Riemannian retraction \cite{beck2023dynamic}. (\bfit{iv}) {\sf Sub-Grad}: Subgradient methods with Euclidean projection \cite{davis2019stochastic,li2021weakly}. %Note that {\sf RADMM} employs fixed and small stepsizes.

\noi $\blacktriangleright$ \textbf{Experiment Settings}. All methods are implemented in MATLAB on an Intel 2.6 GHz CPU with 64 GB RAM. For all retraction-based methods, we use only polar decomposition-based retraction. We evaluate different regularization parameters $\dot{\rho} \in \{10, 50, 100, 500, 1000\}$. For {\sf OADMM}, default parameters are used, with $\beta^0 = 10 \dot{\rho}$ and corresponding values $\xi = 0.5$ for each $\dot{\rho}$. For simplicity, we omit the Barzilai-Borwein strategy and instead use a fixed constant $b^t=1$ for all iterations. All algorithms start with a common initial solution $\x^0$, generated from a standard normal distribution. Our code for reproducing the experiments is available in the \textbf{supplemental material}.

\noi $\blacktriangleright$ \textbf{Experiment Results}. We report the objective values for different methods with varying parameters $\dot{\rho}$. The experimental results presented in Figures \ref{fig:L0PCA:rho:50} and \ref{fig:L0PCA:rho:500} reveal the following insights: (\bfit{i}) {\sf Sub-Grad} essentially fails to solve this problem, as the subgradient is inaccurately estimated when the solution is sparse. (\bfit{ii}) {\sf SPGM-EP} and {\sf SPGM-RR}, which rely on a variable smoothing strategy, exhibit slower performance than the multiplier-based variable splitting method. This observation aligns with the commonly accepted notion that primal-dual methods are generally more robust and faster than primal-only methods. (\bfit{iii}) The proposed {\sf OADMM-EP} and {\sf OADMM-RR} demonstrate similar results and generally achieve lower objective function values than the other methods.

\section{Conclusions}

This paper introduces {\sf OADMM}, an Alternating Direction Method of Multipliers (ADMM) tailored for solving structured nonsmooth composite optimization problems under orthogonality constraints. {\sf OADMM} integrates either a Nesterov extrapolation strategy or a Monotone Barzilai-Borwein (MBB) stepsize strategy to potentially accelerate primal convergence, complemented by an over-relaxation stepsize strategy for rapid dual convergence. We adjust the penalty and smoothing parameters at a controlled rate. Additionally, we develop a novel Lyapunov function to rigorously analyze the oracle complexity of {\sf OADMM} and establish the first non-ergodic convergence rate for this method. Finally, numerical experiments show that our {\sf OADMM} achieves state-of-the-art performance.

\begin{figure}[!t]
\centering

\centering
\begin{subfigure}{.24\textwidth}\centering\includegraphics[width=1.12\linewidth]{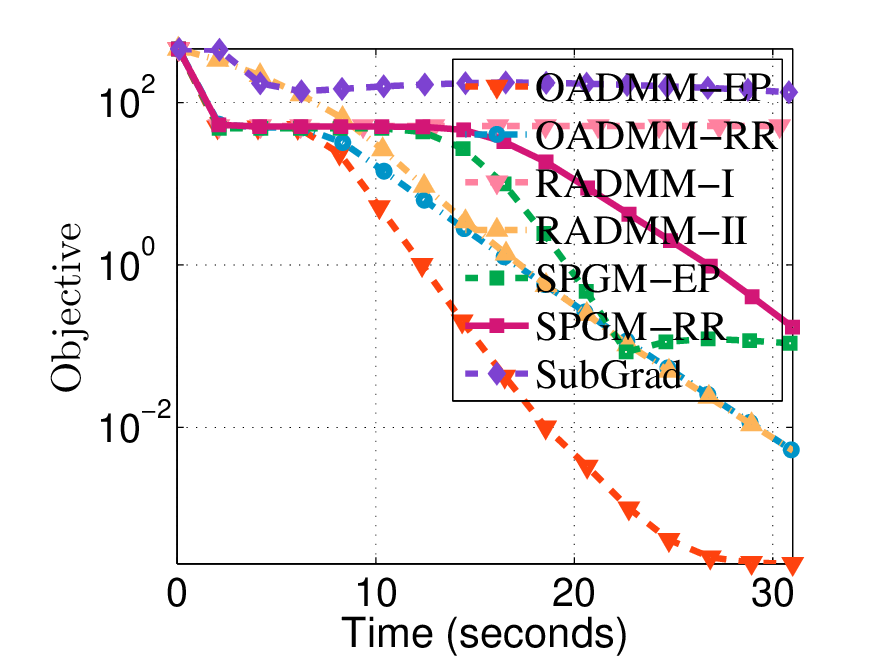}\caption{\scriptsize mnist-1500-500}\label{fig:sub1}\end{subfigure}
\begin{subfigure}{.24\textwidth}\centering\includegraphics[width=1.12\linewidth]{fig//demo_L0PCA50_12.eps}\caption{\scriptsize mnist-2500-500}\label{fig:sub2}\end{subfigure}
\begin{subfigure}{.24\textwidth}\centering\includegraphics[width=1.12\linewidth]{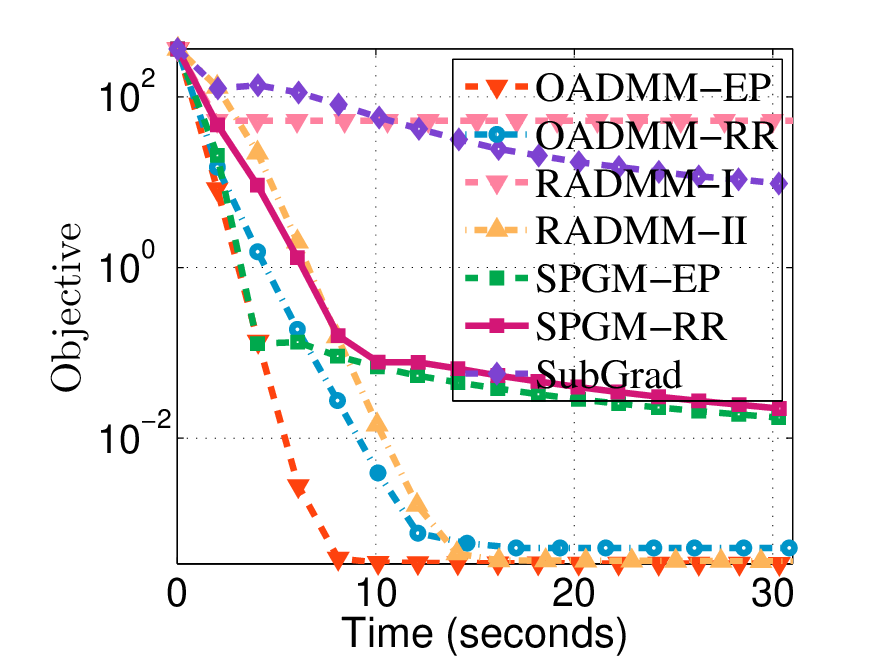}\caption{\scriptsize TDT2-1500-500}\label{fig:sub3}\end{subfigure}
\begin{subfigure}{.24\textwidth}\centering\includegraphics[width=1.12\linewidth]{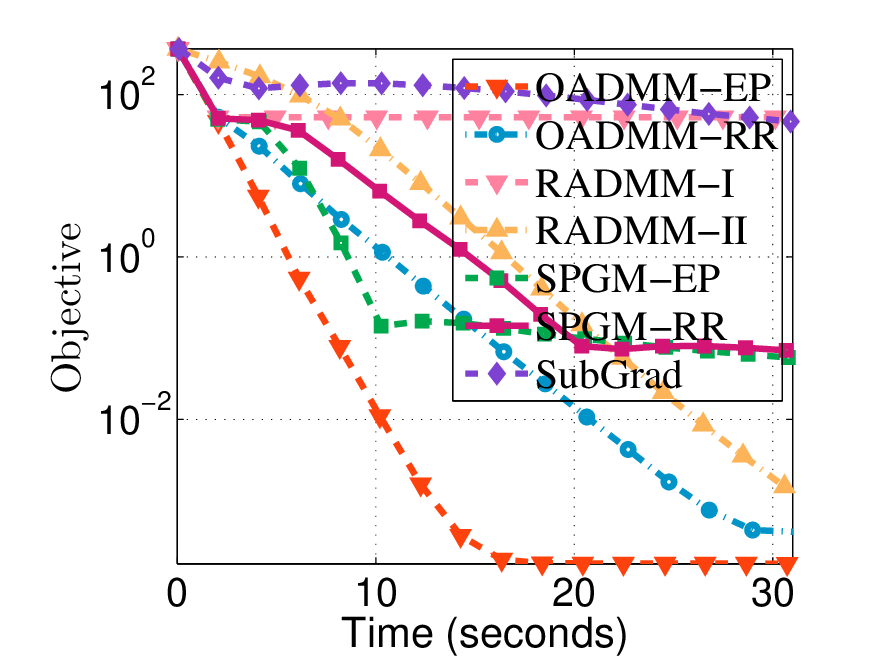}\caption{\scriptsize TDT2-2500-500}\label{fig:sub4}\end{subfigure}

\begin{subfigure}{.24\textwidth}\centering\includegraphics[width=1.12\linewidth]{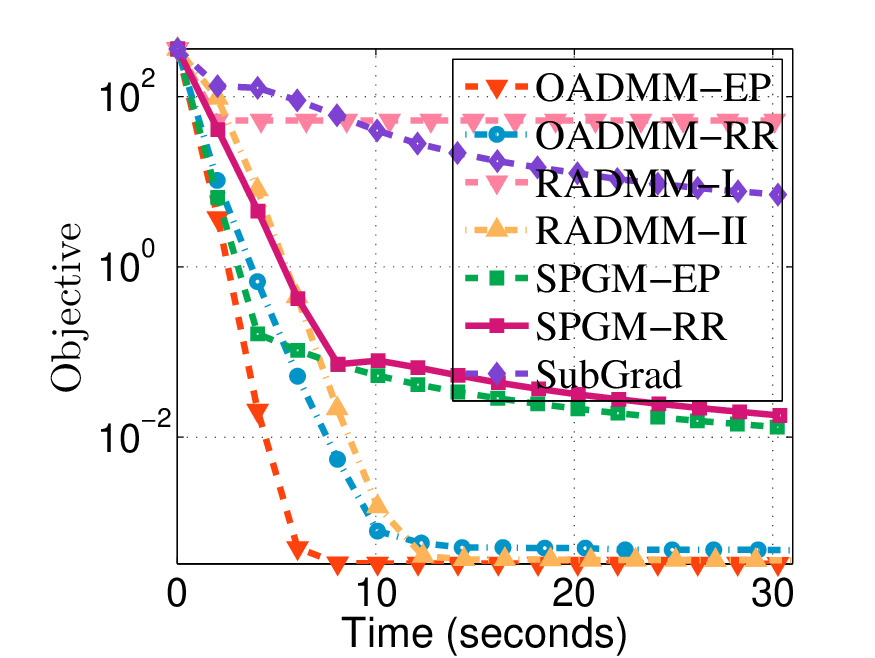}\caption{\scriptsize sector-1500-500}\label{fig:sub1}\end{subfigure}
\begin{subfigure}{.24\textwidth}\centering\includegraphics[width=1.12\linewidth]{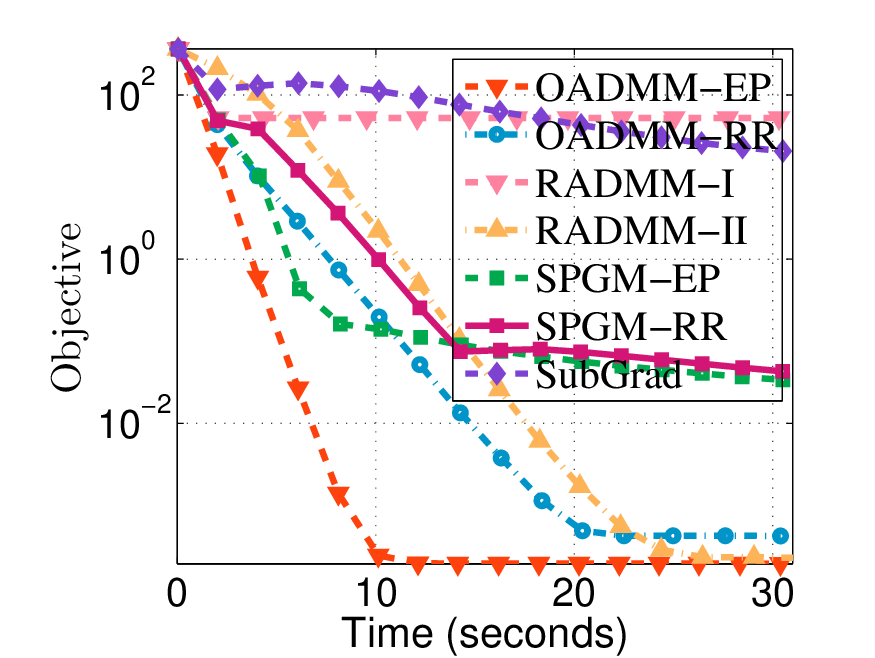}\caption{\scriptsize sector-2500-500}\label{fig:sub2}\end{subfigure}
\begin{subfigure}{.24\textwidth}\centering\includegraphics[width=1.12\linewidth]{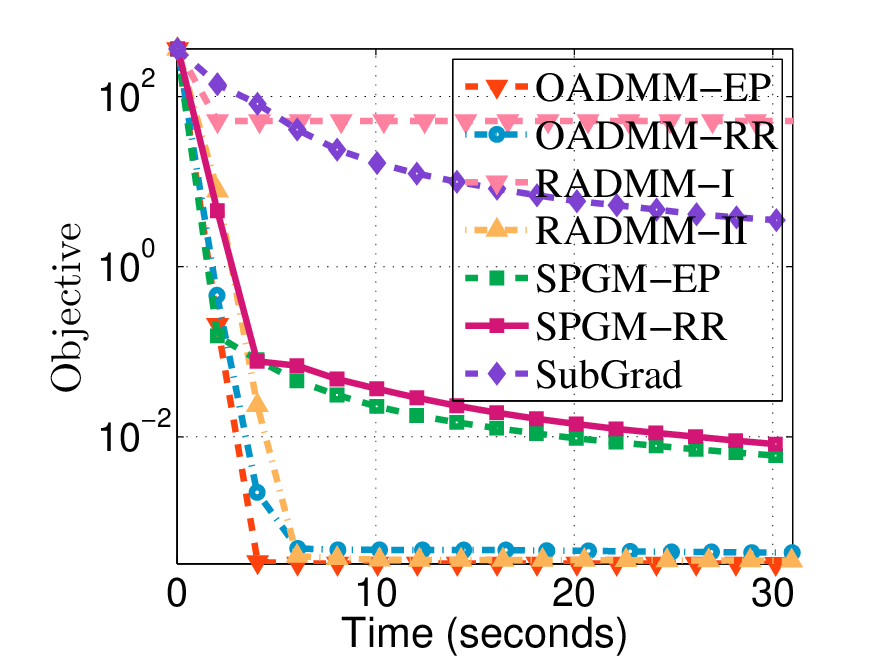}\caption{\scriptsize randn-1500-500}\label{fig:sub3}\end{subfigure}
\begin{subfigure}{.24\textwidth}\centering\includegraphics[width=1.12\linewidth]{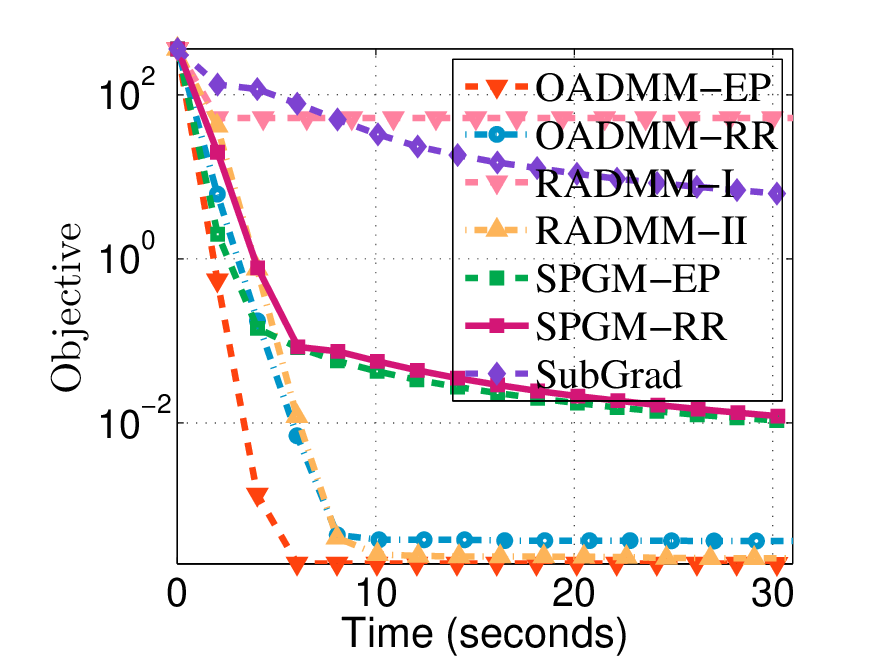}\caption{\scriptsize randn-2500-500}\label{fig:sub4}\end{subfigure}

\caption{The convergence curve of the compared methods with $\dot{\rho}=50$.}

\label{fig:L0PCA:rho:50}

\end{figure}

\begin{figure}[!t]

\centering
\begin{subfigure}{.24\textwidth}\centering\includegraphics[width=1.12\linewidth]{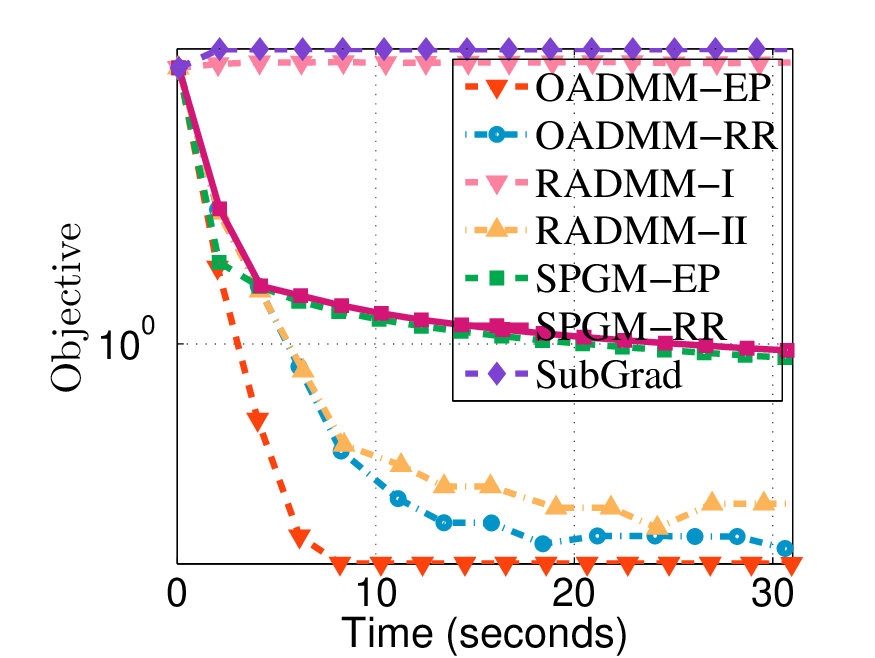}\caption{\scriptsize mnist-1500-500}\label{fig:sub1}\end{subfigure}
\begin{subfigure}{.24\textwidth}\centering\includegraphics[width=1.12\linewidth]{fig//demo_L0PCA500_12.eps}\caption{\scriptsize mnist-2500-500}\label{fig:sub2}\end{subfigure}
\begin{subfigure}{.24\textwidth}\centering\includegraphics[width=1.12\linewidth]{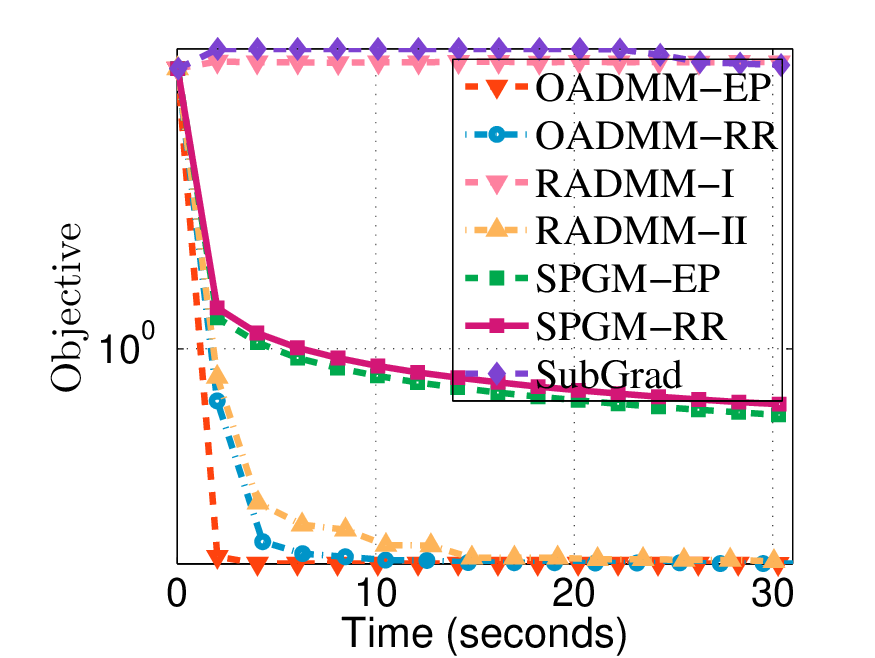}\caption{\scriptsize TDT2-1500-500}\label{fig:sub3}\end{subfigure}
\begin{subfigure}{.24\textwidth}\centering\includegraphics[width=1.12\linewidth]{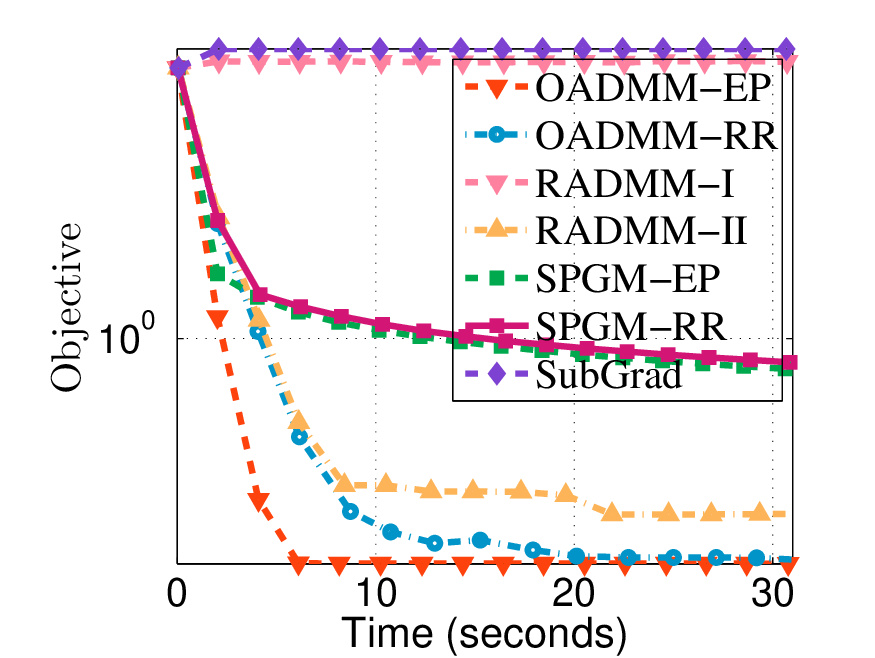}\caption{\scriptsize TDT2-2500-500}\label{fig:sub4}\end{subfigure}

\begin{subfigure}{.24\textwidth}\centering\includegraphics[width=1.12\linewidth]{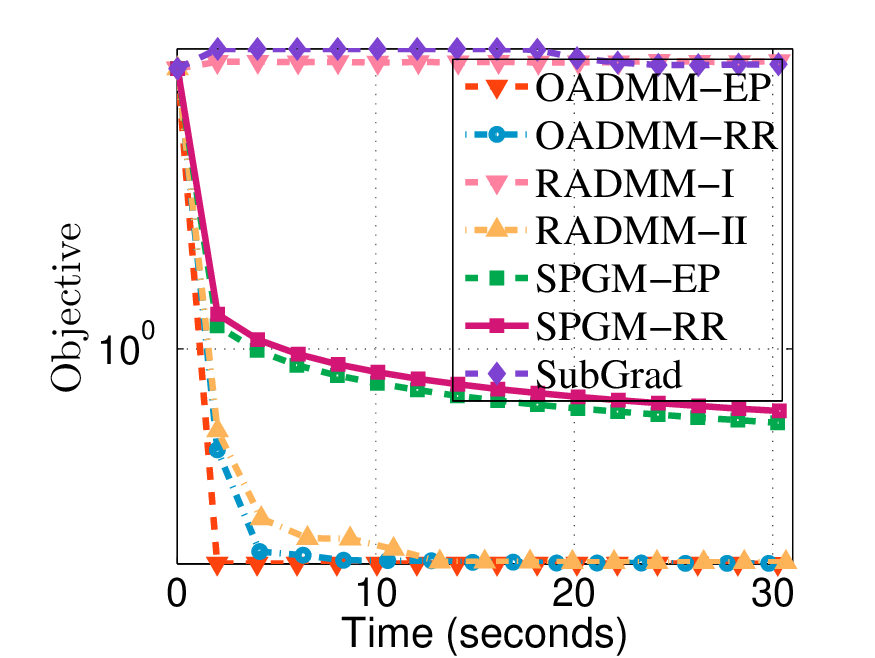}\caption{\scriptsize sector-1500-500}\label{fig:sub1}\end{subfigure}
\begin{subfigure}{.24\textwidth}\centering\includegraphics[width=1.12\linewidth]{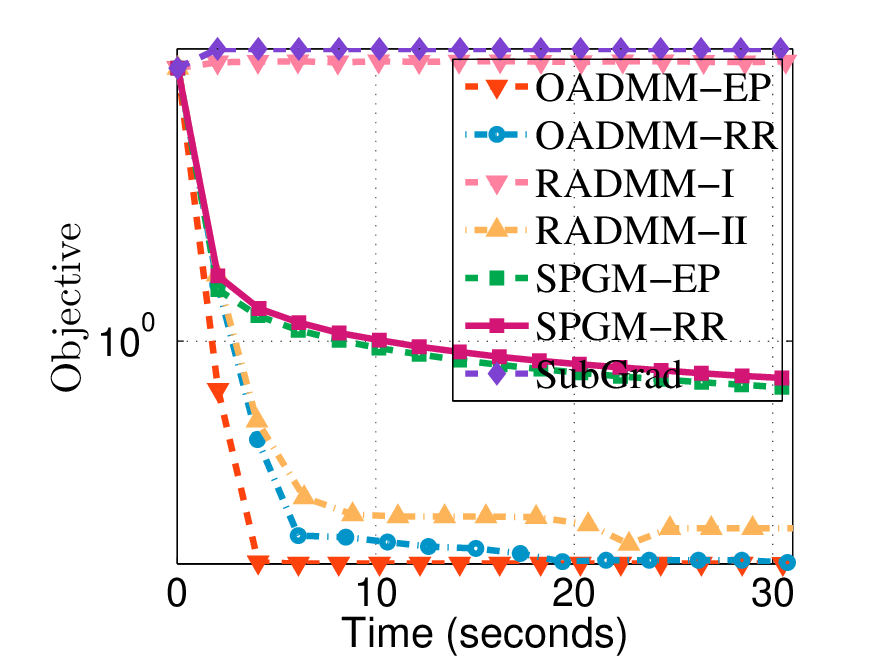}\caption{\scriptsize sector-2500-500}\label{fig:sub2}\end{subfigure}
\begin{subfigure}{.24\textwidth}\centering\includegraphics[width=1.12\linewidth]{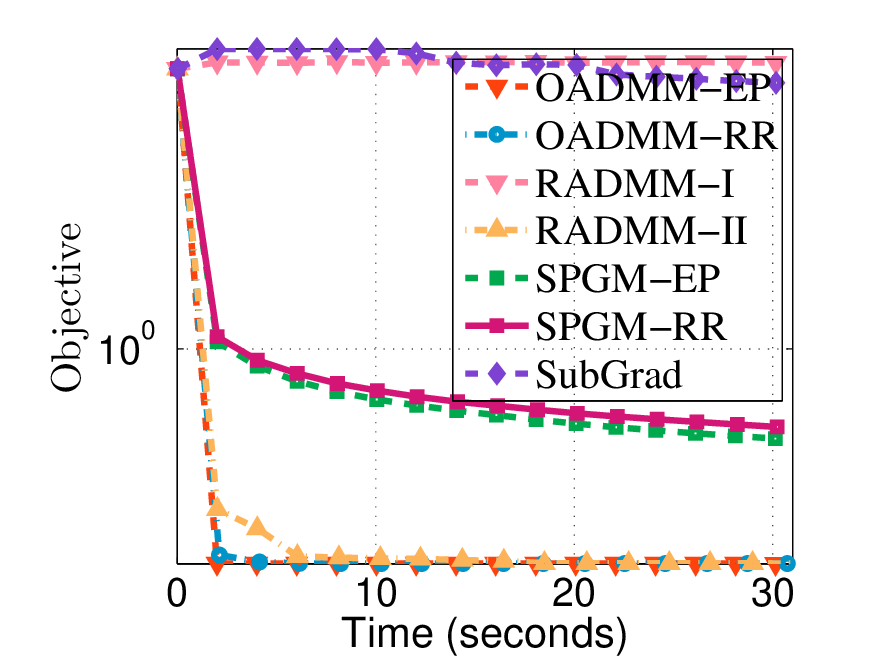}\caption{\scriptsize randn-1500-500}\label{fig:sub3}\end{subfigure}
\begin{subfigure}{.24\textwidth}\centering\includegraphics[width=1.12\linewidth]{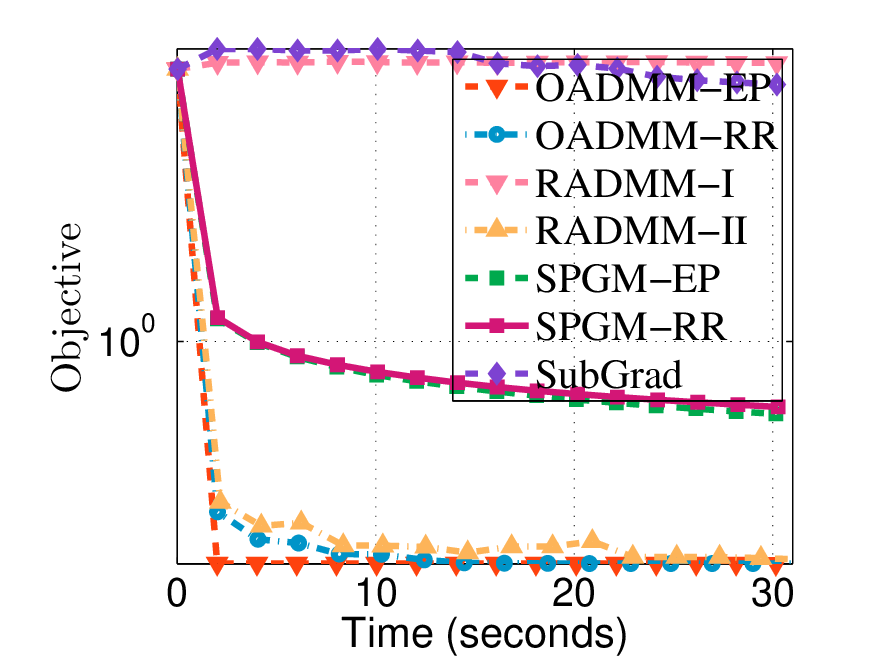}\caption{\scriptsize randn-2500-500}\label{fig:sub4}\end{subfigure}

\caption{The convergence curve of the compared methods with $\dot{\rho}=500$.}

\label{fig:L0PCA:rho:500}

\end{figure}

\normalem
\clearpage
\bibliographystyle{iclr2025_conference}
\bibliography{mybib}

\clearpage
\appendix
{\huge Appendix}

The appendix is organized as follows.

Appendix \ref{app:sect:notation} provides notations, technical preliminaries, and relevant lemmas.

Appendix \ref{app:sect:preli} contains the proofs for Section \ref{sect:preli}.

Appendix \ref{app:sect:iterC} includes the proofs for Section \ref{sect:iterC}.

Appendix \ref{app:sect:KL} encompasses the proofs for Section \ref{sect:strong:KL}.

Appendix \ref{app:exp} presents additional experiments details and results.

\section{Notations, Technical Preliminaries, and Relevant Lemmas} \label{app:sect:notation}
\subsection{Notations}

%In this paper, we denote the Lowercase boldface letters represent vectors, while uppercase letters represent real-valued matrices. The following notations are used throughout this paper.

In this paper, lowercase boldface letters signify vectors, while uppercase letters denote real-valued matrices. The following notations are utilized throughout this paper.

\begin{itemize}[leftmargin=12pt,itemsep=0.2ex]

%Stiefel manifold as $\MM\triangleq\St(n,r)$, which is an embedded submanifold within the Euclidean space $\Rn^{n\times r}$.

\item $[n]$: $\{1,2,...,n\}$

\item $\|\mathbf{x}\|$: Euclidean norm: $\|\mathbf{x}\|=\|\mathbf{x}\|_2 = \sqrt{\la \mathbf{x},\mathbf{x}\ra}$

%\item $\x_i$: the $i$-th element of vector $\x$

%\item $\X_{i,j}$ or $\X_{ij}$ : the ($i^{\text{th}}$, $j^{\text{th}}$) element of matrix $\X$

%\item $\vecc(\X)$ : $\vecc(\X) \in \Rn^{nr\times 1}$, the vector formed by stacking the column vectors of $\X$

%\item $\mat(\mathbf{x})$ : $\mat(\mathbf{x}) \in \Rn^{n\times r}$, Convert $\mathbf{x} \in \Rn^{nr \times 1}$ into a matrix with $\mat(\vecc(\X))=\X$

\item $\X \trans$ : the transpose of the matrix $\X$

%\item $\sign(t)$ :  the signum function, $\sign(t)=1$ if $t\geq0$ and $\sign(t)=-1$ otherwise

%\item $\mathbf{H}$ : $\mathbf{H} \in \Rn^{nr\times nr}$,~the matrix characterizing the smoothness condition of $f(\X)$\\
%\item $\mathcal{H}(\X)$ : $\mathcal{H}(\X) \in \Rn^{n\times r}$,~Equivalent to $\mat(\mathbf{H} (\vec{\X}))\in \Rn^{n \times r}$\\

%\item $\text{det}(\mathbf{D})$ : Determinant of a square matrix $\mathbf{D} \in \Rn^{n\times n}$

%\item $\mathrm{C}_n^k$ : the number of possible combinations choosing $k$ items from $n$ without repetition

\item $\zero_{n,r}$ : A zero matrix of size $n \times r$; the subscript is omitted sometimes

\item $\I_r$  : $\I_r \in \Rn^{r\times r}$, Identity matrix

\item $\MM$: Orthogonality constraint set (a.k.a., Stiefel manifold: $\MM=\{\X \in\Rn^{n\times r}\,|\,\X\trans\X = \mathbf{I}_r\}$.

%\item $\Phi(\mathbf{D})$ : Symmetric operator,  $\Phi(\mathbf{D}) = \frac{1}{2}(\mathbf{D}+\mathbf{D}\trans)$, with $\mathbf{D} \in \Rn^{n\times n}$\\

\item $\X\succeq\zero (\text{or}~\succ \zero)$ : the Matrix $\X$ is symmetric positive semidefinite (or definite)

%\item \Diag$(\mathbf{x})$: Diagonal matrix with $\mathbf{x}$ as the main diagonal entries

%\item \diag$(\X)$: Column vector formed from the main diagonal of $\X$

\item $\tr(\mathbf{A})$ : Sum of the elements on the main diagonal $\mathbf{A}$: $\tr(\mathbf{A})=\sum_i \mathbf{A}_{i,i}$

%\item $\la \X,\mathbf{Y}\ra$ : Euclidean inner product, i.e., $\la \X,\mathbf{Y}\ra =\sum_{ij}{\X_{ij}\mathbf{Y}_{ij}}$

%\item $\X\otimes \mathbf{Y}$ : Kronecker product of $\X$ and $\mathbf{Y}$

%\item $\X \odot \mathbf{Y}$ : Hadamard (a.k.a. entry-wise) product of $\X$ and $\mathbf{Y}$\\
%\item $\|\X\|_{*}$ : Nuclear norm: sum of the singular values of matrix $\X$

\item $\|\X\|$ : Operator/Spectral norm: the largest singular value of $\X$

\item $\|\X\|_{\fro}$ : Frobenius norm: $(\sum_{ij}{\X_{ij}^2})^{1/2}$

\item $\|\X\|_{1}$: Absolute sum of the elements in $\X$ with $\X=\sum_{ij} |\X_{ij}|$

\item $\|\X\|_{[k]}$: $\ell_1$ norm the the $k$ largest (in magnitude) elements of the matrix $\X$

\item $\partial g(\X)$ : (limiting) Euclidean subdifferential of $g(\X)$ at $\X$

%\item $\sgrad F(\X)$ : (limiting) Riemannian subdifferential of $F(\X)$ at $\X$

%\item $\iota_{\Xi}(\X)$ : the indicator function of a set $\Xi$ with $\iota_{\Xi}(\X)=0$ if $\X \in\Xi$ and otherwise $+\infty$

\item $\Proj_{\Xi}(\X')$ : Orthogonal projection of $\X'$ with  $\Proj_{\Xi}(\X') = \arg \arg \min_{\X\in\Xi}\|\X'-\X\|_{\fro}^2$

%\item $\Proj_{\Xi}(\X')$ : Orthogonal projection of $\X'$ with  $\Proj_{\Xi}(\X') = \arg \arg \min_{\X\in\Xi}\|\X'-\X\|_{\fro}^2$

%\item $\mathbb{P}_{\MM }(\y)$ : Nearest orthogonal matrix of $\y$ with  $\mathbb{P}_{\MM }(\y) = \arg \min_{\X\trans \X=\I_r} \|\X-\y\|_{2}^2$

\item  $\text{dist}(\Xi,\Xi')$ : the distance
between two sets with $\text{dist}(\Xi,\Xi') \triangleq \inf_{\X\in\Xi,\X'\in\Xi'}\|\X-\X'\|_{\fro}$

\item  $\|\partial g(\X)\|_{\fro}$: $\|\partial g(\X)\|_{\fro} = \inf_{\mathbf{Y}\in \partial g(\X)}\|\mathbf{Y}\|_{\fro} = \dist(\mathbf{0},\partial g(\X))$.

\item $\ell(\beta^t)$: the smoothness parameter of the function $\mathcal{S}(\X,\y^t,z^t,\beta^t)$ \textit{w.r.t.} $\X$.

%\item $\partial F(\mathbf{\x})$ : classical (limiting) Euclidean subdifferential of $F(\mathbf{\x})$ at $\mathbf{\x}$.

\item $\iota_{\MM}(\mathbf{\x})$ : Indicator function of $\MM$ with $\iota_{\MM}(\mathbf{\x})=0$ if $\mathbf{\x} \in\MM$ and otherwise $+\infty$.
\end{itemize}

We employ the following parameters in Algorithm \ref{alg:main}.

\begin{itemize}[leftmargin=12pt,itemsep=0.2ex]

%\item $\|\mathbf{y}\|_{[k]}$: the $\ell_1$ norm of the $k$ largest (in magnitude) elements of the vector $\mathbf{y}$.

\item $\theta$: proximal parameter

\item $\tau$: correlation coefficient between $\mu^t$ and $\beta^t$, such that $\mu^t\beta^t=\tau$

\item $\sigma$: over-relaxation parameter with $\sigma\in[1,2)$

\item $\alpha$: Nesterov extrapolation parameter with $\alpha \in [0,1)$

 \item $\rho$: search descent parameter with $\rho \in (0,\infty)$

 \item $\gamma$: decay rate parameter in the line search procedure with $\gamma \in (0,1)$

\item $\delta$: sufficient decrease parameter in the line search procedure with $\delta \in (0,\infty)$

\item $p$: exponent parameter used in the penalty update rule with $p\in(0,1)$%: $\beta^t = \beta^0(1+\xi t^p)$

\item $\xi$: growth factor parameter used in the penalty update rule with $\xi\in(0,\infty)$%: $\beta^t = \beta^0(1+\xi t^p)$
\end{itemize}

%We summarize the frequent notations in Table \ref{tab:not}.

%A matrix $\mathbf{W} \in \Rn^{n\times n}$ is called skew-symmetric (or anti-symmetric) if $\mathbf{W}\trans  = - \mathbf{W}$.

%, and directional derivative
%The set of all subgradients of $F(\cdot)$ at $\X$ is called the subdifferential.
%\noi \textbf{$\blacktriangleright$ }
\subsection{Technical Preliminaries}

\textbf{Non-convex Non-smooth Optimization}. Given the potential non-convexity and non-smoothness of the function $F(\cdot)$, we introduce tools from non-smooth analysis \cite{Mordukhovich2006,Rockafellar2009}. The domain of any extended real-valued function $F: \Rn^{n\times r} \rightarrow (-\infty,+\infty]$ is defined as $\text{dom}(F)\triangleq \{\X\in\Rn^{n\times r}: |F(\X)|<+\infty\}$. At $\X\in\text{dom}(F)$, the Fr\'{e}chet subdifferential of $F$ is defined as $\hat{\partial}{F}(\X) \triangleq \{\boldsymbol{\xi}\in\Rn^{n\times r}: \lim_{\mathbf{Z} \rightarrow \X} \inf_{\mathbf{Z}\neq \X} \frac{ {F}(\mathbf{Z}) - {F}(\X) - \la \boldsymbol{\xi},\mathbf{Z}-\X \ra  }{\|\mathbf{Z}-\X\|_{\fro}} \geq 0\}$, while the limiting subdifferential of ${F}(\X)$ at $\X\in\text{dom}({F})$ is denoted as $\partial{F}(\X)\triangleq \{\boldsymbol{\xi}\in \Rn^n: \exists \X^t \rightarrow \X,{F}(\X^t)  \rightarrow {F}(\X), \boldsymbol{\xi}^t \in\hat{\partial}{F}(\X^t) \rightarrow \boldsymbol{\xi},\forall t\}$. The gradient of $F(\cdot)$ at $\X$ in the Euclidean space is denoted as $\nabla{F}(\X)$. The following relations hold among $\hat{\partial}{F}(\X)$, $\partial{F}(\X)$, and $\nabla{F}(\X)$: (\bfit{i}) $\hat{\partial}{F}(\X) \subseteq \partial{F}(\X)$.
(\bfit{ii}) If the function $F(\cdot)$ is convex, $\partial{F}(\X)$ and $\hat{\partial}{F}(\X)$ represent the classical subdifferential for convex functions, i.e., $\partial{F}(\X) = \hat{\partial}{F}(\X) = \{\boldsymbol{\xi}\in\Rn^{n\times r}: F(\mathbf{Z})\geq F(\X)+\la\boldsymbol{\xi},\mathbf{Z}-\X \ra,\forall \mathbf{Z}\in\Rn^{n\times r}\}$. (\bfit{iii}) If the function $F(\cdot)$ is differentiable, then $\hat{\partial}{F}(\X) = \partial{F}(\X) = \{\nabla F(\X)\}$.

\textbf{Optimization with Orthogonality Constraints}. We introduce some prior knowledge of optimization involving orthogonality constraints \cite{Absil2008}. The nearest orthogonality matrix to any arbitrary matrix $\mathbf{Y} \in \Rn^{n\times r}$ is determined as $\mathbb{P}_{\MM }(\mathbf{Y})=\breve{\U}\breve{\mathbf{V}}\trans$, where $\mathbf{Y}= \breve{\U}\Diag(\mathbf{s})\breve{\mathbf{V}}\trans$ represents the singular value decomposition of $\mathbf{Y}$. We use $\mathcal{N}_{\MM}(\X)$ to denote the limiting normal cone to $\MM$ at $\X$, thus defined as $\mathcal{N}_{\MM}(\X)=\partial \iota_{\MM}(\X)=\{\mathbf{Z} \in \Rn^{n\times r}: \la \mathbf{Z},\X\ra\geq \la \mathbf{Z},\mathbf{Y}\ra,\,\forall \mathbf{Y} \in \MM\}$. Moreover, the tangent and normal space to $\MM$ at $\X\in \MM$ are respectively denoted as $\mathrm{T}_{\X}\MM$ and $\mathrm{N}_{\X}\MM$. We have: $\mathrm{T}_{\X}\MM=\{\mathbf{Y}\in\Rn^{n\times r}|\AA_{X}(\mathbf{Y})=\mathbf{0}\}$ and $\mathrm{N}_{\X}\MM=2\X\mathbf{\Lambda}\,|\,\mathbf{\Lambda}=\mathbf{\Lambda}\trans,\mathbf{\Lambda}\in \Rn^{r\times r}\}$, where $\AA_{\X}(\mathbf{Y})\triangleq \X\trans \mathbf{Y} + \mathbf{Y} \trans \X$ for $\mathbf{Y} \in \Rn^{n\times r}$ and $\X \in \MM$.

\textbf{Weakly Convex Functions}. The function $h(\y)$ is weakly convex if there exists a constant $W_h\geq0$ such that $h(\y)+\tfrac{1}{2}W_h\|\y\|_2^2$ is convex; the smallest such $W_h$ is termed the modulus of weak convexity. Weakly convex functions encompass a diverse range, including convex functions, differentiable functions with Lipschitz continuous gradient, and compositions of convex, Lipschitz-continuous functions with $C^1$-smooth mappings having Lipschitz continuous Jacobians \cite{drusvyatskiy2019efficiency}.

%A function $h(\y)$ is $W_h$-weakly convex if and only if: $h(\y) \leq \lambda h(\y_1) + (1-\lambda)h(\y_2) + \tfrac{1}{2}W_h\lambda(1-\lambda)\|\y_1-\y_2\|_{2}^2$ for all $\lambda \in [0,1]$, $\y_1,\y_2 \in \Rn^{m}$, and $\y=\lambda \y_1+(1-\lambda)\y_2$. Any $W_1$-weakly convex function is also $W_2$-weakly convex if $W_1\geq W_2$.

%We denote Minkowski addition and subtraction between two sets by using the symbols `$+$' and `$-$' respectively.

%In the context of any non-convex and non-smooth function $F(\X)$, $\sgrad F(\X)$ symbolizes the limiting Riemannian gradient of $F(\X)$ at $\X$, yielding $\sgrad F(\X) = \mathbb{P}_{\mathrm{T}_{\X}\MM}(\partial F(\X))$. Additionally, we denote $\partial F(\X) \ominus   \X [\partial F(\X)]\trans \X \triangleq \{\E|\E= \G - \X\G\trans \X,\G \in \partial F(\X)\}$, and

\subsection{Relevant Lemmas}

We present a collection of useful lemmas, each of which is independent of context and methodology.

\begin{lemma} \label{lemma:the:alpha}
Assume $\mathbf{a},\mathbf{b}\in \Rn^n$, and $\alpha\geq 0$. We have: $-\|\mathbf{a} - \alpha \mathbf{b}\|_2^2 \leq (\alpha-1) \|\mathbf{a} \|_2^2 -  (\alpha^2-\alpha)  \|\mathbf{b} \|_2^2$.

\begin{proof}
We have: $-\|\mathbf{a} - \alpha \mathbf{b}\|_2^2 = -\|\mathbf{a}\|_2^2- \|\alpha \mathbf{b}\|_2^2 + 2 \alpha \la \mathbf{a}, \mathbf{b} \ra \leq -\|\mathbf{a}\|_2^2- \|\alpha \mathbf{b}\|_2^2 + 2 \alpha\cdot ( \tfrac{1}{2} \|\mathbf{a}\|_2^2 + \tfrac{1}{2}\|\mathbf{b}\|_2^2 ) = (\alpha-1) \|\mathbf{a} \|_2^2 -  (\alpha^2 - \alpha)  \|\mathbf{b} \|_2^2$.
\end{proof}

\end{lemma}

\begin{lemma} \label{lemma:a:b:sigma}
Assume $\a^+ = \varrho \a + \b$, where $\a,\b,\a^+\in\Rn^m$, and $\varrho\in [0,1)$. We have: $\|\a^+\|_2^2 - \tfrac{\varrho}{1-\varrho} (\|\a\|_2^2 - \|\a^+\|_2^2) \leq \tfrac{1}{(1-\varrho)^2}\|\b\|_2^2$.

\begin{proof}
We have: $\|\a^+\|_2^2 =  \|\varrho \a + \b\|_2^2 = \|\varrho \a + (1-\varrho) \cdot \tfrac{\b}{1-\varrho}\|_2^2 \leq \varrho \|\a\|_2^2 + (1-\varrho)\cdot \|\tfrac{\b}{1-\varrho}\|_2^2=\varrho\|\a\|_2^2 + \tfrac{1}{1-\varrho}\|\b\|_2^2$, where the inequality holds due to the convexity of $\|\cdot\|_2^2$. Subtracting $\varrho\|\a^+\|_2^2$ from both sides gives: $(1-\varrho)\|\a^+\|_2^2 \leq \varrho(\|\a\|_2^2-\|\a^+\|_2^2) + \tfrac{1}{1-\varrho}\|\b\|_2^2$. Dividing through by $(1-\varrho)$ yields the resulting inequality.

\end{proof}
\end{lemma}

\begin{lemma} \label{lemma:upper:lower}

Assume $p\in(0,1)$. We have:

\begin{enumerate}[label=\textbf{(\alph*)}, leftmargin=20pt, itemsep=1pt, topsep=1pt, parsep=0pt, partopsep=0pt]

\item $p (t+1)^{p-1} \leq (t+1)^p - t^p \leq p t^{p-1}$ for all integer $t\geq 1$.

\item $p (t+1)^{p-1} \leq (t+1)^p - t^p$ for all integer $t\geq 0$.
\end{enumerate}

\begin{proof}
The proof of this lemma follows from the concavity of $h(x) = x^p$ for all $x\geq 0$ and $p \in (0,1)$, and is omitted for brevity.

\end{proof}

\begin{comment}
We prove that $\frac{t^p-(t-1)^p}{(t+1)^p-t^p} \leq (\tfrac{t-1}{t})^{p-1}$ for all $p\in (0,1)$ and $t\geq 2$. We define $f(x)=x^p$. Given $g(x)$ is concave, for all $x>y\geq 1$, we have: $f(x)-f(y)\leq \la x-y,\nabla f(y)\ra$, leading to $x^p-y^p\leq p y^{p-1}(x-y)$. Applying this inequality yields:
\beq
&&\forall t\geq 2,~t^p - (t-1)^p \leq p(t-1)^{p-1},\nn\\
&&\forall t\geq 1,~(t+1)^p - t^p \geq p t^{p-1}. \nn
\eeq
We derive:
\beq
\ts \frac{t^p-(t-1)^p}{(t+1)^p-t^p} \leq (\tfrac{t-1}{t})^{p-1} \leq \frac{ p(t-1)^{p-1}}{p t^{p-1}} = (\frac{t-1}{t})^{p-1}. \label{eq:cvx:two}
\eeq

\end{comment}

\end{lemma}

\begin{lemma} \label{lemma:g:p:T:p:T}
Assume $p\in(0,1)$. For all $t\geq1$, we have $\tfrac{1}{1-p}(t+1)^{1-p}-\tfrac{1}{1-p} - (1-p) t^{1-p} \geq 0$.

\begin{proof}

We define $f(t) \triangleq \tfrac{1}{1-p}(t+1)^{1-p}-\tfrac{1}{1-p} - (1-p) t^{1-p}$, where $p\in(0,1)$ and $t\geq 1$.

\textbf{Part (a)}. We now show that $(1-p)^{1/p} \leq \tfrac{1}{\exp(1)}$. Recall that it holds: $\lim_{p \rightarrow 0^+} (1+p)^{1/p} = \exp(1)$ and $\lim_{p \rightarrow 0^+} (1-p)^{1/p} = 1/ \exp(1)$. Given the function $h(p)\triangleq (1-p)^{1/p}$ is a decreasing function on $p\in(0,1)$, we have $h(p)\leq \lim_{p \rightarrow 0^+} (1-p)^{1/p} = \tfrac{1}{\exp(1)}$.

\textbf{Part (b)}. We now show that $g(q) = 2^q - 1 - q^2\geq0$ for all $q\in(0,1)$. We have $\nabla g(q)= \log(2) 2^q - 2 q$, and $\nabla^2 g(q)= 2^q (\log(2))^2  - 2  \leq 2 (\log(2))^2  - 2\leq 0$, implying that the function $g(q)$ is concave on $q\in(0,1)$. Noticing $g(0)=g(1)=0$, we conclude that $g(q)\geq0$.

\textbf{Part (c)}. We now show that $f(t)$ is an increasing function. We have: $\nabla f(t) = \ts (t+1)^{-p} - (1-p)^2  t^{-p}  =  \ts(t+1)^{-p} \cdot ( 1 - (1-p)^2  (\tfrac{t+1}{t})^{p} )\overset{\step{1}}{\geq} \ts(t+1)^{-p} \cdot ( 1 - (1-p)^2  2^{p} )\overset{\step{2}}{\geq} \ts (t+1)^{-p}  \cdot ( 1 - (\frac{2}{\exp(1)^{2}})^p   ) \overset{\step{3}}{\geq}  \ts 0$, where step \step{1} uses $\frac{t+1}{t}\leq 2$ for all $t\geq 1$; step \step{2} uses $1-p\leq (\tfrac{1}{\exp(1)})^p$ for all $p\in(0,1)$; step \step{3} uses $\tfrac{2}{\exp(1)^{2}} \thickapprox 0.2707 <1$.

\textbf{Part (d)}. Finally, we have: $\forall t\geq 1,\,f(t) \overset{\step{1}}{\geq} f(1) =  (1-p)^{-1}\cdot \{ 2^{(1-p)}-1 - (1-p)^2\} \overset{\step{2}}{\geq} 0$, where step \step{1} uses the fact that $f(t)$ is an increasing function; step \step{2} uses $2^q - 1 - q^2\geq0$ for all $q=1-p\in(0,1)$.

\end{proof}

\end{lemma}

\begin{lemma} \label{lemma:lp:bounds:2}
Assume $p\in(0,1)$. We have: $(1-p) T^{(1-p)}\leq \sum_{t=1}^T \tfrac{1}{t^p} \leq \tfrac{T^{(1-p)}}{1-p}$.

    \begin{proof}

% \url{https://en.wikipedia.org/wiki/Integral_test_for_convergence}
We define $g(t) \triangleq \tfrac{1}{t^p} $ and $h(t)\triangleq \frac{1}{1-p} t^{(1-p)}$.

Using the integral test for convergence, we obtain: $\int_{1}^{T+1} g(x)dx \leq \sum_{t=1}^T g(t) \leq g(1) + \int_{1}^{T} g(x)dx$.

\textbf{Part (a)}. We now consider the lower bound. We obtain: $\sum_{t=1}^T t^{-p}  \geq  \ts \sum_{t=1}^T \int_{t}^{t+1} x^{-p} dx = \int_{1}^{T+1} x^{-p} dx  \overset{\step{1}}{\geq}   \ts h(T+1) - h(1) =  \tfrac{1}{1-p} (T+1)^{1-p}-\tfrac{1}{1-p} \overset{\step{2}}{\geq}  (1-p) T^{1-p}$, where step \step{1} uses $\nabla h(x)= x^{-p}$; step \step{2} uses Lemma \ref{lemma:g:p:T:p:T}.

\textbf{Part (b)}. We now consider the upper bound. We have: $\sum_{t=1}^T t^{-p} \leq \ts h(1) + \int_{1}^{T} x^{-p} dx  \overset{\step{1}}{=} 1 + h(T) - h(1) =       1 + \tfrac{1}{1-p} (T)^{1-p}-\tfrac{1}{1-p} \overset{}{=}   \tfrac{T^{(1-p)} - p }{1-p} < \tfrac{T^{(1-p)}}{1-p}$, where step \step{1} uses $\nabla h(x)= x^{-p}$.

\end{proof}

\end{lemma}

\begin{lemma} \label{lemma:two:non:negative:sequences}

Assume that $\a^{t} \leq \varrho \a^{t-1} + c$, where $\varrho\in [0,1)$, $c\geq 0$, and $\{\a^i\}_{i=0}^{\infty}$ is a non-negative sequence. We have: $\a^{t}  \leq \a^{0} +  \tfrac{c}{1-\varrho}$ for all $t\geq 0$.

\begin{proof}
Using basic induction, we have the following results:
\beq
& \ts t=1,&\ts \a^{1} \leq \varrho \a^{0} + c \nn\\
& \ts t=2,& \ts \a^{2} \leq \varrho \a^{1} + c \leq  \varrho (\varrho \a^{0} + c) + c = \varrho^2 \a^{0} +   c (1 + \varrho )     \nn\\
& \ts t=3,&\ts \a^{3} \leq \varrho \a^{2} + c\leq  \varrho (\varrho^2 \a^{0} +   ( c + \varrho c)) + c = \varrho^3 \a^{0} +  c ( 1 + \varrho  + \varrho^2 )     \nn\\
& \ts ...& \nn\\
& \ts t=n,& \ts\a^{n} \leq \varrho \a^{n-1} + c\leq  \varrho^n \a^{0} + c \cdot (1 + \varrho + \ldots +  \varrho^{n-1} ).\nn
\eeq
\noi Therefore, we obtain: $\a^{n} \leq\varrho^n \a^{0} + c \cdot (1 + \varrho + \ldots +  \varrho^{n-1} ) \overset{\step{1}}{\leq} \a^0 + \frac{c}{1-\varrho}$, where step \step{1} uses $\rho^n\leq \rho < 1$, and the summation formula of geometric sequences that $ 1 + \varrho^1 + \varrho^2 + ... +  \varrho^{t-1} = \frac{1-\varrho^t}{1-\varrho} < \frac{1}{1-\varrho}$.

\end{proof}

\end{lemma}

\begin{lemma} \label{lemma:p:3:frac}
  
  For all $p\in(0,1)$, it holds that $(p-3) 2^p + (p+3)\leq 0$.
   
   \begin{proof}
     
We define $f(p) = (p-3) 2^p + (p+3)\leq 0$. 
     
We define $g(p)\triangleq f'(p)=2^p \cdot [1+(p-3)\log(2)] +1$.

We define $h(p)\triangleq f''(p)=2^p\log(2) \cdot [2+(p-3)\log(2)]$. 

%Solving the equation $h(p)=0$ gives one unique solution $p^*\triangleq 3 - \tfrac{2}{\log(2)} \thickapprox 0.1146$. 

\textbf{Part (a)}. Clearly, the equation $h(p)=0$ has one unique solution $p^*\triangleq 3 - \tfrac{2}{\log(2)} \thickapprox 0.1146$. Notably, $h(p) \leq 0$ for all $p\in(0,p^*]$, and $h(p)\geq 0$ for all $p\in [p^*,1)$.

\textbf{Part (b)}. We now show that the equation $g(p)=0$ has a unique solution. Noting that $g(p)$ is deceasing on $(0,p^*)$ and increasing on $(p^*,1)$, and observing that $g(p)<0$ for all $(0,p^*)$ and $g(1)>0$, we conclude, by the Mean Value Theorem, that $g(p)=0$ has a unique solution.

\textbf{Part (c)}. Given that $g(p)=0$ has a unique solution, we conclude that $f(p)$ has exactly one critical point on $(0,1)$. Since $f(0)=f(1)=0$, $f'(0)<0$, and $f'(1)>0$, it follows that $f(p)\leq0$ for all $p\in(0,1)$.

   \end{proof}

\end{lemma}

\begin{lemma}\label{lemma:bound:3:p}
Let $\beta^t = \beta^0 (1 + \xi t^p)$, where $t\geq 1$, $\beta^0>0$, $\xi,p\in(0,1)$. For all integer $t\geq 1$, we have: $(\tfrac{1}{\beta^{t-1}}-\tfrac{1}{\beta^{t}}) / (\tfrac{1}{\beta^{t}}- \tfrac{1}{\beta^{t+1}}) \leq \tfrac{3}{p}$.

\begin{proof}

We derive:
\beq \label{eq:Psi:two:case}
\Psi &\triangleq & \ts (\tfrac{1}{\beta^{t-1}}-\tfrac{1}{\beta^{t}}) / (\tfrac{1}{\beta^{t}}- \tfrac{1}{\beta^{t+1}}) \nn\\
&\overset{\step{1}}{=}& \ts (\tfrac{1}{1 + \xi (t-1)^p}-\tfrac{1}{1 + \xi t^p}) / (\tfrac{1}{1 + \xi t^p}- \tfrac{1}{1 + \xi (t+1)^p}) \nn\\
&=& \ts (\frac{\xi t^p - \xi (t-1)^p}{[1+\xi(t-1)^p][1+\xi t^p]}) / (\frac{\xi (t+1)^p - \xi t^p}{[1+\xi(t+1)^p][1+\xi t^p]}) \nn\\
&=& \ts \tfrac{t^p- (t-1)^p}{(t+1)^p - t^p}  \cdot \tfrac{1+\xi (t+1)^p}{1+\xi (t-1)^p},
\eeq
\noi where step \step{1} uses the definition of $\beta^t$.

We now focus on Inequality (\ref{eq:Psi:two:case}). Case (\bfit{i}). When $t=1$, we have:
\beq
\Psi  =  \tfrac{1}{2^p-1}\cdot \tfrac{1+\xi 2^p}{1} = \tfrac{2^p+1}{2^p-1} \overset{\step{1}}{\leq } \tfrac{3}{p}, \nn
\eeq
\noi where step \step{1} uses the fact that $\tfrac{2^p+1}{2^p-1}\leq \tfrac{3}{p}$ for all $p\in(0,1)$, which is implied by Lemma \ref{lemma:p:3:frac}.

Case (\bfit{ii}). When $t\geq 2$, we have:
\beq
\Psi  &=& \ts  \tfrac{t^p- (t-1)^p}{(t+1)^p - t^p}  \cdot \tfrac{1+\xi (t+1)^p}{1+\xi (t-1)^p}\nn\\
&\overset{\step{1}}{\leq}&  \tfrac{ p (t-1)^{p-1}}{ p (t+1)^{p-1}}  \cdot \tfrac{1+\xi (t+1)^p}{1+\xi (t-1)^p} \nn\\
&\overset{\step{2}}{\leq}& \tfrac{ p (t-1)^{p-1}}{ p (t+1)^{p-1}} \cdot \tfrac{(t+1)^p}{(t-1)^p} \nn\\
&\overset{}{=}  & \ts \tfrac{t+1}{t-1}  \leq   \ts \tfrac{3}{p}, \nn
\eeq
\noi where step \step{1} uses Lemma \ref{lemma:upper:lower}; step \step{2} uses $\tfrac{a+b}{c+d}\leq \max(\tfrac{a}{c},\tfrac{b}{d})$ for all $a,b,c,d>0$.

\end{proof}
\end{lemma}

\begin{lemma}\label{lemma:bound:mumumu}
Let $\beta^t = \beta^0 (1 + \xi t^p)$, where $t\geq 0$, $\beta^0>0$, $\xi,p\in(0,1)$. For all $t\geq 1$, we have: $(\tfrac{\beta^{t}}{\beta^{t-1}} - 1)^2 \leq (\tfrac{1}{\beta^t} - \tfrac{1}{\beta^{t+1}} ) \cdot \tfrac{6 \beta^0}{p}$.

\begin{proof}

First, we derive:
\beq
\Psi  &=&  (\tfrac{\beta^{t}}{\beta^{t-1}} - 1)^2 / (\tfrac{1}{\beta^t} - \tfrac{1}{\beta^{t+1}} ) \nn\\
&=& \beta^{t} (\tfrac{\beta^{t}}{\beta^{t-1}} - 1) \cdot\{ (\tfrac{1}{\beta^{t-1}} - \tfrac{1}{\beta^{t}}) / (\tfrac{1}{\beta^t} - \tfrac{1}{\beta^{t+1}} )\}\nn\\
&\overset{\step{1}}{\leq}& \tfrac{3}{p} \beta^{t} (\tfrac{\beta^{t}}{\beta^{t-1}} - 1) \nn\\
&= & \tfrac{3}{p} \xi \beta^0  (1 + \xi t^p) (\tfrac{  t^p - (t-1)^p }{  1 + \xi (t-1)^p } ),  \label{eq:Psi:two:case:2}
\eeq
\noi where step \step{1} uses Lemma \ref{lemma:bound:3:p}. We now focus on Inequality (\ref{eq:Psi:two:case:2}). Case (\bfit{i}). When $t=1$, we have:
\beq
\Psi \leq \tfrac{3}{p}\xi \beta^0 (1+\xi)\overset{\step{1}}{\leq} \tfrac{6}{p}\beta^0, \nn
\eeq
\noi where step \step{1} uses $\xi\in(0,1)$.

Case (\bfit{ii}). When $t\geq 2$, we have:
\beq
\Psi &\leq& \tfrac{3}{p} \xi \beta^0   \tfrac{1 + \xi t^p  }{  1 + \xi (t-1)^p }  \cdot ( t^p - (t-1)^p) \nn\\
&\overset{\step{1}}{\leq}& \tfrac{3}{p} \xi \beta^0   \tfrac{1 + \xi t^p  }{  1 + \xi (t-1)^p }  \cdot ( p (t-1)^{p-1}  ) \nn\\
&\overset{\step{2}}{\leq}& \tfrac{3}{p} \xi \beta^0  \tfrac{ t^p }{ (t-1)^p } \cdot ( p (t-1)^{p-1}  ) \nn\\
& = & 3 \xi \beta^0  \tfrac{ t^p }{ t-1 } \nn\\
&\overset{\step{3}}{\leq}& 3 \xi \beta^0  \tfrac{ t }{ t-1 } \nn\\
&\overset{\step{4}}{\leq}& 6 \beta^0/p, \nn
\eeq
\noi where step \step{1} uses Lemma \ref{lemma:upper:lower}; step \step{2} uses $\tfrac{a+b}{c+d}\leq \max(\tfrac{a}{c},\tfrac{b}{d})$ for all $a,b,c,d>0$; step \step{3} uses $p\in(0,1)$; step \step{4} uses $p,\xi\in(0,1)$, and $\tfrac{t}{t-1}\leq 2$ for all $t\geq 2$.

%Second, we have: $(\tfrac{\beta^{t}}{\beta^{t-1}} - 1)^2 \overset{\step{1}}{=} \ts  (\tfrac{1 + \xi t^p}{1 + \xi (t-1)^p} - 1)^2 = (\tfrac{ \xi t^p  - \xi(t-1)^p}{1 + \xi (t-1)^p} )^2 \overset{\step{2}}{\leq}  \ts  (\tfrac{  t^p  - (t-1)^p}{1 +  (t-1)^p} )^2\overset{\step{3}}{\leq}  \ts (\tfrac{ 1}{t})^2  \overset{\step{4}}{\leq}  \ts \frac{2}{t} - \frac{2}{t+1}$, where step \step{1} uses $\beta^t = \beta^0 (1 + \xi t^p)$; step \step{2} uses $\tfrac{ \xi}{1 + \xi a} < \tfrac{1}{1+a}$ for all $a\geq0$ when $\xi\in(0,1)$; step \step{3} uses Lemma \ref{lemma:bound:mumumu:before}; step \step{4} uses the fact that $\frac{1}{t^2} \leq \tfrac{2}{t}-\tfrac{2}{t+1}$ for all $t\geq 1$.

\end{proof}
\end{lemma}

\begin{lemma} \label{lemma:bound:4}

Assume $\beta^t=\beta^0(1+\xi t^p)$, where $p\in [1/4,1)$, and $t\geq 1$ is an integer. We have: $(\tfrac{1}{\beta^t})^{3} \leq c \sqrt{ (\tfrac{1}{\beta^{t-1}}-\tfrac{1}{\beta^{t}}) \tfrac{1}{\beta^{t-1}} }$, where $c = \tfrac{2}{ (\xi\beta^0)^2}$.

\begin{proof}

For all $t\geq 1$, we have:
\beq
\ts (\tfrac{1}{\beta^t})^{3} / \sqrt{ (\tfrac{1}{\beta^{t-1}}-\tfrac{1}{\beta^{t}}) \tfrac{1}{\beta^{t-1}} }  &\overset{\step{1}}{\leq}& \ts (\tfrac{1}{\beta^t})^{2} / \sqrt{ \tfrac{\beta^{t-1}}{\beta^{t-1}}-\tfrac{\beta^{t-1}}{\beta^{t}}  }  \nn\\
%&\overset{\step{2}}{=}& \ts \tfrac{1}{(\beta^0)^2} \cdot \tfrac{1}{  (1+\xi t^p)^2 } / \sqrt{ 1 - \tfrac{ 1+\xi (t-1)^p }{ 1+\xi t^p }  }  \nn\\
&\overset{\step{2}}{=}& \ts \tfrac{1}{(\beta^0)^2} \cdot \tfrac{1}{  (1+\xi t^p)^{2} }  \cdot \sqrt{ \tfrac{ 1+\xi t^p  }{ \xi t^p-\xi (t-1)^p }  }  \nn\\
&\overset{}{=}& \ts \tfrac{1}{(\beta^0)^2} \cdot \tfrac{1}{  (1+\xi t^p)^{3/2} }  \cdot \sqrt{ \tfrac{ 1  }{ \xi t^p-\xi (t-1)^p }  }  \nn\\
%&\overset{\step{3}}{=}& \ts \tfrac{1}{(\beta^0)^2} \cdot \tfrac{1}{  (1+\xi t^p)^{3/2} }  \cdot \sqrt{ \tfrac{ 1  }{ \xi p t^{p-1} }  }  \nn\\
&\overset{\step{3}}{\leq}& \ts \tfrac{1}{(\beta^0)^2} \cdot \tfrac{1}{  (\xi t^p)^{3/2} }  \cdot \sqrt{ \tfrac{ 1  }{ \xi p t^{p-1} }  }  \nn\\
&\overset{}{=}& \ts \tfrac{1}{(\xi\beta^0)^2} \tfrac{1}{\sqrt{p}} \cdot  \sqrt{  t^{-3p} }    \cdot \sqrt{  t^{1-p}   }  \nn\\
&\overset{}{=}& \ts \tfrac{1}{(\xi\beta^0)^2} \tfrac{1}{\sqrt{p}} \cdot \sqrt{  t^{1-4p}}  \nn\\
&\overset{\step{4}}{\leq}& \ts \tfrac{2}{ (\xi\beta^0)^2}, \nn
\eeq
\noi where step \step{1} uses $\beta^t\leq \beta^{t+1}$; step \step{2} uses $\beta^t=\beta^0(1+\xi t^p)$; step \step{3} uses Lemma \ref{lemma:upper:lower} that $t^p-\xi (t-1)^p \geq p t^{p-1}$ for all integer $t\geq 1$; step \step{4} uses $p\geq 1/4$.

\end{proof}
\end{lemma}

\begin{lemma} \label{lemma:simple}
Assume $\X_{\cc}^{t} = \X^{t}+\alpha (\X^{t}  - \X^{t-1})$, where $\alpha \in[0,1)$, and $\X^t,\X^{t-1}\in\MM$. We have:

(\bfit{a}) $\| \X^{t} - \X^{t}_{\cc}\|_{\fro} \leq \| \X^{t} - \X^{t-1}\|_{\fro}$.

(\bfit{b}) $\| \X^{t+1} - \X^{t}_{\cc}\|_{\fro} \leq \| \X^{t+1} - \X^{t}\|_{\fro}+ \| \X^{t} -  \X^{t-1} \|_{\fro}$.

(\bfit{c}) $\|\AA(\X_{\cc}^t) - \y^t\| \leq \| \AA(\X^{t}) - \y^t\| + \Aup \| \X^{t} - \X^{t-1}\|_{\fro}$.

\begin{proof}

\textbf{Part (a)}. We have: $\| \X^{t} - \X^{t}_{\cc}\|_{\fro} \overset{\step{1}}{=} \alpha \|\X^{t} -  \X^{t-1} \|_{\fro} \overset{\step{2}}{\leq} \| \X^{t} - \X^{t-1}\|_{\fro}$, where step \step{1} uses $\X^t_{\cc} = \X^{t} + \alpha (\X^t - \X^{t-1})$; step \step{2} uses $\alpha\in[0,1)$.

\textbf{Part (b)}. We have: $\| \X^{t+1} - \X^{t}_{\cc}\|_{\fro} \overset{\step{1}}{=} \| \X^{t+1} - \X^{t} - \alpha ( \X^{t} -  \X^{t-1}) \|_{\fro}\overset{\step{2}}{\leq} \| \X^{t+1} - \X^{t}\|_{\fro}+ \| \X^{t} -  \X^{t-1} \|_{\fro}$, where step \step{1} uses $\X^t_{\cc} = \X^{t} + \alpha (\X^t - \X^{t-1})$; step \step{2} uses the triangle inequality and $\alpha\in[0,1)$.

\textbf{Part (c)}. We have: $\|\AA(\X^{t}_{\cc}) - \y^{t}\| \overset{\step{1}}{\leq }\|\AA(\X^{t}) - \y^t\|+\| \AA(\X^{t}) - \AA(\X^{t}_{\cc})\| \leq \|\AA(\X^{t}) - \y^t\|+\Aup \| \X^{t} - \X^{t}_{\cc}\|\overset{\step{2}}{\leq }\| \AA(\X^{t}) -\y^t\|+\Aup \| \X^{t} - \X^{t-1}\|$, where step \step{1} uses the triangle inequality; step \step{2} uses Claim (\bfit{a}) of this lemma.

\end{proof}

\end{lemma}

\begin{lemma} \label{lemma:Proj:Proj}

Let $\mathbf{P}, \tilde{\mathbf{P}}\in \Rn^{n\times r}$, and $\X,\tilde{\X}\in\MM$. We have:
\beq
\| \Proj_{\T_{\X} \MM} (\mathbf{P}) - \Proj_{\T_{\tilde{\X}}\MM} (\tilde{\mathbf{P}})\|_{\fro} \leq 2\| \mathbf{P}-\tilde{\mathbf{P}}\|_{\fro} + 2 \sqrt{r} \|\mathbf{P}\|\|\X-\tilde{\X}\|_{\fro}.\nn
\eeq

\begin{proof}

%Here, $\varrho \triangleq 2 \sqrt{r}$.

%First, we have:
%\beq
%\|\X \mathbf{H}\trans \X - \tilde{\X} \mathbf{H}\trans \tilde{\X}\|_{\fro} &=& \| (\X-\tilde{\X}) \mathbf{H}\trans \X + \tilde{\X}\mathbf{H}\trans (\X-\tilde{\X}) \|_{\fro} \nn\\
%&\overset{\step{1}}{\leq}& \| \X-\tilde{\X}\|_{\fro} \|\mathbf{H}\trans \X\| + \|\tilde{\X}\mathbf{H}\trans\|\| \X-\tilde{\X}\|_{\fro} \nn\\
%&\overset{\step{2}}{\leq}&  2 \sqrt{r} \|\mathbf{H}\| \| \X-\tilde{\X}\|_{\fro}, \nn
%\eeq
%\noi where step \step{1} uses the triangle inequality; step \step{2} uses $\|\mathbf{A}\mathbf{B}\|_{\fro} \leq \|\mathbf{A}\|\cdot \|\mathbf{B}\|_{\fro}$.

First, we obtain:
\beq \label{eq:Proj:AAA}
&&\|\X \mathbf{P}\trans \X - \tilde{\X} \tilde{\mathbf{P}}\trans \tilde{\X}\|_{\fro} \nn\\
&=& \| (\X-\tilde{\X}) \mathbf{P}\trans \X + \tilde{\X}\mathbf{P}\trans (\X-\tilde{\X}) + \tilde{\X} (\mathbf{P} - \tilde{\mathbf{P}})\trans \tilde{\X}  \|_{\fro} \nn\\
&\overset{\step{1}}{\leq}& \| \X-\tilde{\X}\|_{\fro} \|\mathbf{P}\trans \X\| + \|\tilde{\X}\mathbf{P}\trans\|\| \X-\tilde{\X}\|_{\fro} + \|\tilde{\X} (\mathbf{P} - \tilde{\mathbf{P}})\trans \tilde{\X}\|_{\fro}  \nn\\
&\overset{\step{2}}{\leq}&  2 \sqrt{r} \|\mathbf{P}\| \| \X-\tilde{\X}\|_{\fro} + \| \mathbf{P} - \tilde{\mathbf{P}}\|_{\fro},
\eeq
\noi where step \step{1} uses the triangle inequality; step \step{2} uses $\|\mathbf{A}\mathbf{B}\|_{\fro} \leq \|\mathbf{A}\|\cdot \|\mathbf{B}\|_{\fro}$, and $\|\tilde{\X}\|\leq1$.

Second, we have:
\beq\label{eq:Proj:BBB}
&& \|\X \X\trans \mathbf{P} - \tilde{\X} \tilde{\X}\trans \tilde{\mathbf{P}}\|_{\fro}\nn\\
&=& \| (\X-\tilde{\X}) \X\trans\mathbf{P}  + \tilde{\X}(\X-\tilde{\X})\trans \mathbf{P} + \tilde{\X} \tilde{\X}\trans  (\mathbf{P} - \tilde{\mathbf{P}}) \|_{\fro} \nn\\
&\overset{\step{1}}{\leq}& \| \X-\tilde{\X}\|_{\fro} \|\X\trans \mathbf{P}\| + \|\tilde{\X}\|\cdot \| \X-\tilde{\X}\|_{\fro} \cdot \|\mathbf{P}\| + \|\tilde{\X} \tilde{\X}\trans \| \cdot \|\mathbf{P} - \tilde{\mathbf{P}}\|_{\fro}  \nn\\
&\overset{\step{2}}{\leq}&  2 \sqrt{r} \|\mathbf{P}\| \| \X-\tilde{\X}\|_{\fro} + \| \mathbf{P} - \tilde{\mathbf{P}}\|_{\fro},
\eeq
\noi where step \step{1} uses the triangle inequality; step \step{2} uses $\|\mathbf{A}\mathbf{B}\|_{\fro} \leq \|\mathbf{A}\|\cdot \|\mathbf{B}\|_{\fro}$, and $\|\tilde{\X}\|\leq1$.

Finally, we derive:
\beq
&&\| \Proj_{\T_{\X} \MM} (\mathbf{P}) - \Proj_{\T_{\tilde{\X}}\MM} (\tilde{\mathbf{P}})\|_{\fro} \nn\\
&\overset{\step{1}}{=}& \| [\mathbf{P} - \tfrac{1}{2}\X \mathbf{P}\trans \X - \tfrac{1}{2}\X \X\trans \mathbf{P}  ] - [\tilde{\mathbf{P}} - \tfrac{1}{2}\tilde{\X} \tilde{\mathbf{P}}\trans \tilde{\X} - \tfrac{1}{2}\tilde{\X} \tilde{\X}\trans \tilde{\mathbf{P}}  ] \|_{\fro}\nn\\
&\overset{\step{2}}{\leq}& \| \mathbf{P}  -  \tilde{\mathbf{P}} \|_{\fro} + \tfrac{1}{2}\| \X \mathbf{P}\trans \X - \tilde{\X} \tilde{\mathbf{P}} \trans \tilde{\X} \|_{\fro} + \tfrac{1}{2}\| \X \X\trans \mathbf{P} - \tilde{\X} \tilde{\X} \trans \tilde{\mathbf{P}} \|_{\fro} \nn\\
&\overset{\step{3}}{\leq}& \| \mathbf{P}  -  \tilde{\mathbf{P}} \|_{\fro} +2 \sqrt{r} \|\mathbf{P}\| \| \X-\tilde{\X}\|_{\fro} + \| \mathbf{P} - \tilde{\mathbf{P}}\|_{\fro} \nn
\eeq
\noi where step \step{1} uses $\Proj_{\T_{\X}\MM}(\Deltas) = \Deltas - \tfrac{1}{2}\X (\Deltas \trans \X+\X\trans \Deltas)$ for all $\Deltas\in\Rn^{n\times r}$ \cite{absil2008optimization}; step \step{2} uses the triangle inequality; step \step{3} uses Inequalities (\ref{eq:Proj:AAA}) and (\ref{eq:Proj:BBB}).

\end{proof}

\end{lemma}

\begin{lemma} \label{lemma:any:seq}

Assume $(e^{t+1})^2 \leq (p^t - p^{t+1}) \cdot  (e^{t} +w^t)$, where $\{p^t\}_{t=1}^{\infty}$ is a nonnegative decreasing sequence. We have: $\sum_{t=i}^{\infty} e^{t+1} \leq df$.

\begin{proof}

\end{proof}
\end{lemma}

\begin{lemma} \label{lemma:any:three:seq}

Assume $(e^{t+1})^2 \leq e^{t}  (p^t-p^{t+1})$ and $p^t\geq p^{t+1}$, where $\{e^{t},\,p^t\}_{t=0}^{\infty}$ are two nonnegative sequences. For all $i\geq 1$, we have: $\sum_{t=i}^{\infty} e^{t+1} \leq e^i + e^{i-1} + 4 p^{i}$.

\begin{proof}

We define $w_t\triangleq p^t-p^{t+1}$. We let $1\leq i<T$. We let $\alpha>0$ with $1-\sqrt{ \tfrac{\alpha}{2}}>0$.

We obtain the following results:
\beq \label{eq:to:be:tel}
e^{t+1} &\overset{\step{1}}{\leq}& \ts \sqrt{  e^{t}  w_t } \nn\\
 &\overset{\step{2}}{\leq}& \ts \sqrt{    \tfrac{\alpha}{2} (e^{t})^2 + (w_t)^2/(2\alpha) },\nn\\
 &\overset{\step{3}}{\leq}& \ts \sqrt{ \tfrac{\alpha}{2}} \cdot e^{t} + w_t \sqrt{   1 / (2\alpha)},
\eeq
\noi where step \step{1} uses $(e^{t+1})^2 \leq e^{t} (p^t-p^{t+1})$ and $w_t\triangleq p^t-p^{t+1}$; step \step{2} uses the fact that $ab\leq \frac{\alpha}{2} a^2 + \frac{1}{2\alpha}b^2$ for all $\alpha>0$; step \step{3} uses the fact that $\sqrt{a+b}\leq \sqrt{a}+\sqrt{b}$ for all $a,b\geq 0$.

Telescoping Inequality (\ref{eq:to:be:tel}) over $t$ from $i$ to $T$, we obtain:
\begin{align}
\ts \sum_{t=i}^T w_t \sqrt{ {1}/{(2\alpha)}}  \geq&~ \ts \{\sum_{t=i}^T e^{t+1} \} -\sqrt{ \tfrac{\alpha}{2}}  \{ \sum_{t=i}^T e^{t} \}  \nn\\
= &~ \ts e^{T+1}  - \sqrt{ \tfrac{\alpha}{2}}  e^i  + (1-\sqrt{ \tfrac{\alpha}{2}}  )\sum_{t=i}^{T-1} e^{t+1} \nn\\
\overset{\step{1}}{\geq}&~ \ts  - \sqrt{ \tfrac{\alpha}{2}}  e^i + (1-\sqrt{ \tfrac{\alpha}{2}} )\sum_{t=i}^{T-1} e^{t+1}, \nn
\end{align}
\noi where step \step{1} uses $e^{T+1}\geq 0$; step \step{2} uses $1-\sqrt{\frac{\alpha}{2}}  >0$. This results in:
 \beq
 \ts \sum_{t=i}^{T-1} e^{t+1} &\leq&  \ts (1-\sqrt{ \tfrac{\alpha}{2}})^{-1} \cdot \{   \sqrt{ \tfrac{\alpha}{2}}  e^i + \sqrt{   \tfrac{1}{2\alpha}}  \sum_{t=i}^T w_t  \} \nn\\
 &\overset{\step{1}}{=}& \ts  e^i  + 2 \sum_{t=i}^T w_t \nn\\
 &\overset{}{=}& \ts  e^i  + 2 (p^{i} - p^{T+1}) \nn\\
 &\overset{\step{2}}{\leq}& \ts  e^i  + 2 p^i,\nn
 \eeq
 \noi step \step{1} uses the fact that $(1-\sqrt{ \tfrac{\alpha}{2}})^{-1} \cdot  \sqrt{ \tfrac{\alpha}{2}}=1$ and $(1-\sqrt{ \tfrac{\alpha}{2}})^{-1} \cdot \sqrt{   \tfrac{1}{2\alpha}} =2$ when $\alpha=1/2$; step \step{2} uses $p^{T+1}\geq 0$. Letting $T\rightarrow \infty$, we conclude this lemma.

\end{proof}
\end{lemma}

\begin{lemma} \label{lemma:sigma:case:less:than:1:4}

Let $\{d^t\}_{t=1}^{\infty}$ be a nonnegative sequence satisfying $d^t \leq d^{t-1} - d^{t} + c \cdot t^{pu}\cdot (d^{t-1}-d^{t})^u$, where $u>0$, $c>0$, $p\in(0,1)$. Then we have: $d^t \leq  \mathcal{O}(t^{ -\zeta})$, where $\zeta=(1-p)u$.
\begin{proof}

We derive the following results:
\beq \label{eq:two:case:sep}
d^t &\leq& d^{t-1} - d^{t} + c t^{ pu} (d^{t-1}- d^{t})^u \nn\\
&\overset{\step{1}}{\leq}& 2\max(d^{t-1} - d^{t} ,c t^{ pu} (d^{t-1}- d^{t})^u),
\eeq
\noi where step \step{1} uses $a+b\leq 2\max(a,b)$ for all $a,b\geq 0$.

We now consider two cases for Inequality (\ref{eq:two:case:sep}). Case (\bfit{i}). When $d^{t-1} - d^{t}\geq c t^{ u p} (d^{t-1}- d^{t})^u$, we have:
\beq\label{eq:two:case:sep:case:1}
d^t \leq  2 (d^{t-1} - d^{t}) \leq  \tfrac{2}{3} d^{t-1} \leq  (\tfrac{2}{3})^t d^{0}.\nn
\eeq

Case (\bfit{ii}). When $d^{t-1} - d^{t}\leq c t^{pu} (d^{t-1}- d^{t})^u$, we have: $ d^t \leq 2 c t^{pu} (d^{t-1}- d^{t})^u$, leading to:
\beq \label{eq:to:be:sum:case:less:than:1:4}
(2 c)^{-1/u} \tfrac{(d^t)^{1/u}}{  t^{p} } \leq d^{t-1}- d^{t}.
\eeq

Summing over Inequality (\ref{eq:to:be:sum:case:less:than:1:4}) over $t$ from $1$ to $T$ yields:
\beq
0 &\leq&\ts \sum_{t=1}^T \{ d^{t-1}- d^{t} - (2 c)^{-1/u} \tfrac{(d^t)^{1/u}}{  t^{p} } \} \nn\\
&\overset{\step{1}}{\leq}&   d^{0} - \ts (2 c)^{-1/u} \sum_{t=1}^T \{  \tfrac{(d^t)^{1/u}}{  t^{p} } \} \nn\\
&\overset{\step{2}}{\leq}&   d^{0}  - \ts (2 c)^{-1/u} (d^T)^{1/u} \cdot \sum_{t=1}^T  \tfrac{1}{  t^{p} }  \nn\\
&\overset{\step{3}}{\leq}&   d^{0}  - \ts (2 c)^{-1/u} (d^T)^{1/u} \cdot  (1-p) T^{1-p} , \nn
\eeq
\noi where step \step{1} uses $\sum_{t=1}^T d^{t-1} - d^t = d^0-d^T \leq d^0$; step \step{2} uses the fact that $\{d^t\}_{t=1}^T$ is decreasing; step \step{3} uses Lemma \ref{lemma:lp:bounds:2} that $\sum_{t=1}^T \tfrac{1}{t^p}\geq (1-p) T^{(1-p)}$. This further leads to:
\beq
&& (2 c)^{-1/u} (d^T)^{1/u} \cdot  (1-p) T^{1-p} \leq d^0\nn\\
&\Rightarrow&  d^T  \leq  (d^0)^u (2 c) (1-p)^{-u} T^{u(p-1)} = \OO(T^{(p-1)u}).\nn
\eeq

\end{proof}

\end{lemma}

\begin{lemma} \label{lemma:sigma:case:between:1:2:to:1}

Assume that $d^{t}\leq c \cdot t^{pu} \cdot (d^{t-1}-d^{t})^u$, where $c>0$, $u,p\in(0,1)$. Then we have: $d^t \leq  \mathcal{O}(t^{-\zeta})$, where $\zeta = \tfrac{1-p}{1/u-1}$.

\begin{proof}

We define $r \triangleq \tfrac{1}{u}-1>0$, and $g(s) \triangleq  s^{-r-1}$.

From the inequality $d^{t}\leq c \cdot t^{pu} \cdot (d^{t-1}-d^{t})^u$, we obtain:
\beq \label{eq:get:first}
c^{1/u} t^{p} (d^{t-1}-d^{t})   \geq   (d^{t})^{1/u}    \overset{\step{1}}{=}   (d^{t})^{r+1} \overset{\step{2}}{=}   \tfrac{1}{g(d^{t})},
\eeq
\noi where step \step{1} uses $r+1 = \tfrac{1}{u}$; step \step{2} uses the definition of $g(s)$.

We let $\kappa>1$ be any constant and consider two cases for $g(d^{t})/g(d^{t-1})$.

\noi \textbf{Case (1)}. $g(d^{t})\leq \kappa g(d^{t-1})$. We define $f(s) \triangleq - \tfrac{1}{r} \cdot s^{-r}$. We derive:
\beq \label{eq:hhh:case:1}
1 &\overset{\step{1}}{\leq}& \ts c^{1/u} t^p \cdot (d^{t-1} - d^{t})\cdot g( d^{t})\nn\\
&\overset{\step{2}}{\leq}&\ts c^{1/u} t^p \cdot (d^{t-1} - d^{t})\cdot \kappa g( d^{t-1}) \nn\\
&\overset{\step{3}}{\leq}&\ts c^{1/u} t^p \cdot \kappa \int_{d^{t}}^{d^{t-1}} g(s) ds\nn\\
&\overset{\step{4}}{=}& \ts  c^{1/u} t^p \cdot \kappa \cdot ( f(d^{t-1}) - f(d^{t}) )\nn\\
&\overset{\step{5}}{=}& \ts  c^{1/u} t^p \cdot \kappa \cdot \tfrac{1}{r}\cdot ( [d^{t}]^{-r} - [d^{t-1}]^{-r}  ),\nn
\eeq
\noi where step \step{1} uses Inequality (\ref{eq:get:first}); step \step{2} uses $g(d^{t})\leq \kappa g(d^{t-1})$; step \step{3} uses the fact that $g(s)$ is a nonnegative and increasing function that $(a - b ) g(a) \leq \int_{b}^a g(s) ds$ for all $a,b \in [0,\infty)$; step \step{4} uses the fact that $\nabla f(s) = g(s)$; step \step{5} uses the definition of $f(\cdot)$. This leads to:
\beq \label{eq:bound:mu:1}
[d^{t}]^{-r} - [d^{t-1}]^{-r}\geq \tfrac{ r}{c^{1/u} \kappa  t^p }.
\eeq

\noi \textbf{Case (2)}. $g(d^{t})> \kappa g(d^{t-1})$. We have:
\beq\label{eq:dnu:dnu1:case:2}
g(d^{t}) > \kappa g(d^{t-1})   &\overset{\step{1}}{\Rightarrow}&   [d^{t}]^{-(r+1)} > \kappa \cdot [d^{t-1}]^{-(r+1)} \nn \\
&\overset{\step{2}}{\Rightarrow}& ([d^{t}]^{-(r+1) })^{ \tfrac{r}{r+1} } > \kappa^{\tfrac{r}{r+1} }  \cdot ([d^{t-1}]^{- (r+1)} )^{ \tfrac{r}{r+1} }  \nn \\
&\overset{}{\Rightarrow}& [d^{t}]^{-r} > \kappa^{ \tfrac{r}{r+1} } \cdot [d^{t-1}]^{-r} ,
\eeq
\noi where step \step{1} uses the definition of $g(\cdot)$; step \step{2} uses the fact that if $a>b>0$, then $a^{\dot{r}}>b^{\dot{r}}$ for any exponent $\dot{r}\triangleq \tfrac{r}{r+1} \in (0,1)$. We further derive:
\beq\label{eq:bound:mu:2}
\ts [d^{t}]^{-r} - [d^{t-1}]^{-r} &\overset{\step{1}}{\geq}& \ts  (\kappa^{ \tfrac{r}{r+1}}   - 1) \cdot [d^{t-1}]^{-r} \nn\\
&\overset{\step{2}}{\geq}&  \ts    (\kappa^{ \tfrac{r}{r+1} }   - 1) \cdot [d^{0}]^{-r},
\eeq
\noi where step \step{1} uses Inequality (\ref{eq:dnu:dnu1:case:2}); step \step{2} uses $r>0$ and $d^{t-1}\leq d^0$ for all $t$.

In view of Inequalities (\ref{eq:bound:mu:1}) and (\ref{eq:bound:mu:2}), we have:
\beq
[d^{t}]^{-r} - [d^{t-1}]^{-r} &\geq& \ts \min(\tfrac{\kappa^{-1} r}{ c^{1/u} t^p},(\kappa^{ \tfrac{r}{r+1} }   - 1) \cdot [d^{0}]^{-r})\nn\\
&=&\ts \OO(\tfrac{1}{t^p}) . \label{eq:to:be:tel:dd}
\eeq

We now focus on Inequality (\ref{eq:to:be:tel:dd}). Telescoping Inequality (\ref{eq:to:be:tel:dd}) over $t= \{1,2,\ldots,T\}$, we have:
\beq
\ts [d^{T}]^{-r} - [d^0]^{-r}  \geq  \ts    \OO(\sum_{t=1}^{T} \tfrac{1}{t^p})  \overset{\step{1}}{=}   \mathcal{O}( (1-p)T^{1-p} ) = \mathcal{O}(T^{1-p} ), \nn
\eeq
\noi where step \step{1} use Lemma \ref{lemma:lp:bounds:2}. This leads to:
\beq
d^T =  ([d^T]^{-r})^{-1/r} \leq \mathcal{O}(T^{1-p})^{-1/r} = \mathcal{O}(T^{-\zeta}).\nn
\eeq

\end{proof}

\end{lemma}

\begin{lemma}\label{lemma:sigma:case:between:1:4:to:1/2}

Assume that $d^t/d^{t-1} \leq \tfrac{ c t^{q} }{ c t^{q} + 1}$, where $c\geq0$ and $q\in(0,1)$. Then we have: $d^{t}\leq \mathcal{O} ({1}/{\exp(t^\zeta)})$, where $\zeta=1-q$.

\begin{proof}

We define $\gamma^t\triangleq \tfrac{1}{  c t^{q} + 1} \in (0,1)$.

First, we derive the following results:
\beq \label{eq:sum:gamma}
\ts \sum_{t=1}^T \gamma^t &=& \ts \sum_{t=1}^T \tfrac{1}{  c t^{q} + 1} \nn\\
&\overset{\step{1}}{\geq}&\ts \tfrac{1}{1+c}  \sum_{t=1}^T \tfrac{1}{ t^{q} } \nn\\
&\overset{\step{2}}{\geq}&\ts \tfrac{1}{1+c}  (1-q) T^{(1-q)} = \OO(T^{1-q}),
\eeq
\noi where step \step{1} uses $c t^q+1\leq (1+c) t^q$ since $t^q\geq 1$ if $t\geq 1$; step \step{2} uses Lemma \ref{lemma:lp:bounds:2} that $(1-p) T^{(1-p)}\leq \sum_{t=1}^T \tfrac{1}{t^p}$ for all $p\in(0,1)$.

Second, noticing that $\tfrac{d^{t}}{d^{t-1}} = \tfrac{ c t^{q} }{ c t^{q} + 1} = 1 - \gamma^t$, we have:
\beq\label{eq:2j:2j}
\tfrac{d^{T}}{d^{0}} &\leq& (1-\gamma^{1}) (1-\gamma^{2}) (1-\gamma^{3}) \ldots (1-\gamma^{T}).
\eeq
This further leads to:
\beq
\ts d^{T} & = & \ts \exp(\log( d^{T} )) \nn\\
&\overset{\step{1}}{\leq}&  \ts \exp(\log( d^0 \cdot \prod_{t=1}^{T} (1-\gamma^{ t}) )) \nn\\
&\overset{\step{2}}{=}& \ts \exp( \log(d^0) + \sum_{t=1}^{T} \log (1 -\gamma^{ t} ) ) \nn\\
&\overset{\step{3}}{\leq}&\ts \exp( \log(d^0) + \sum_{t=1}^{T} ( -\gamma^{ t} )) \nn\\
&\overset{\step{4}}{\leq}&\ts \exp( \log(d^0)) \times \{\exp ( \sum_{t=1}^{T}  ( \gamma^{ t} )) \}^{-1} \nn\\
&\overset{\step{5}}{\leq}&\ts d^0 \times \{\exp ( \OO(T^{(1-q)})  ) \}^{-1}  =  \mathcal{O}( {1}/{\exp ( T^{1-q})} ) ,\nn
\eeq
\noi where step \step{1} uses Inequality (\ref{eq:2j:2j}); step \step{2} uses $\log(a b)=\log(a)+\log(b)$ for all $a,b>0$; step \step{3} uses $\log(1-x)\leq -x$ for all $x\in(0,1)$, and $1-\gamma^{t}\in(0,1)$ for all $t$; step \step{4} uses $\exp(a + b) = \exp(a) \exp(b)$ for all $a,b>0$; step \step{5} uses Inequality (\ref{eq:sum:gamma}).

%For all $q\in (0,1)$, we have:
%\beq
%\sum_{t=1}^T \frac{1}{ t^q} \geq (1-q) T^{1-q}
%\eeq
%
%$\{d^t\}_{t=0}^{\infty}$ and $\{\dot{\beta}^t\}_{t=0}^{\infty}$ are two nonnegative sequences. Assume that $\{\dot{\beta}^t\}_{t=0}^{\infty}$ is increasing.
%
%$\sum_{i=0}^T ( {1}/{\dot{\beta}^{i}}) \geq \mathcal{O}(T^a)$

%When $T$ is an odd number, employing analogous strategies yields: $ d^{T+1 }  = \mathcal{O}( {1}/{\exp ( T^a)} )$.

\end{proof}

\end{lemma}

\section{Proofs for Section \ref{sect:preli}}\label{app:sect:preli}
\subsection{Proof of Lemma \ref{lemma:mu:continous}}\label{app:lemma:mu:continous}

\begin{proof}

Assume $0<\mu_2<\mu_1<\frac{1}{W_h}$, and fixing $\y\in\Rn^{m}$.

We define $h_{\mu_1}(\y)\triangleq \min_{\v} h(\v) + \frac{1}{2\mu_1}\|\v-\y\|_{2}^2$, and $\prox_{\mu_1}(\y) = \ts\arg \min_{\v}\,h(\v) + \tfrac{1}{2\mu_1} \|\v - \y\|_2^2$.

We define $h_{\mu_2}(\y)\triangleq \min_{\v} h(\v) + \frac{1}{2\mu_2}\|\v-\y\|_{2}^2$, and $\prox_{\mu_2}(\y) = \ts \arg \min_{\v}\,h(\v) + \tfrac{1}{2\mu_2} \|\v - \y\|_2^2$.

\noi By the optimality of $\prox_{\mu_1}(\y)$ and $\prox_{\mu_2}(\y)$, we obtain:
\beq
\ts\y - \prox_{\mu_1}(\y)  &\in& \ts\mu_1 \partial h(\prox_{\mu_1}(\y)) , \label{eq:optimality:mu:1}\\
\ts\y - \prox_{\mu_2}(\y)  &\in& \ts\mu_2  \partial h(\prox_{\mu_2}(\y)) .\label{eq:optimality:mu:2}
\eeq

\noi \textbf{Part (a)}. We now prove that $0 \leq h_{\mu_2}(\y) - h_{\mu_1}(\y)$. For any $\s_1 \in \partial h(\prox_{\mu_1}(\y))$ and $\s_2\in \partial h(\prox_{\mu_2}(\y))$, we have:
\beq
&& h_{\mu_1}(\y) - h_{\mu_2}(\y)\nn\\
&\overset{\step{1}}{=}& \tfrac{1}{2\mu_1} \|\y -  \prox_{\mu_1}(\y) \|_{2}^2 - \tfrac{1}{2\mu_2} \|\y - \prox_{\mu_2}(\y) \|_{2}^2 + h ( \prox_{\mu_1}(\y)) - h(\prox_{\mu_2}(\y))  \nn\\
&\overset{\step{2}}{\leq}& \tfrac{1}{2\mu_1}\| \y -  \prox_{\mu_1}(\y) \|_{2}^2 - \tfrac{1}{2\mu_2}\| \y -  \prox_{\mu_2}(\y)\|_{2}^2 + \la \prox_{\mu_1}(\y)-\prox_{\mu_2}(\y),\s_1\ra+\tfrac{W_h}{2}\|\prox_{\mu_2}(\y) - \prox_{\mu_1}(\y) \|_{2}^2 \nn\\
&\overset{\step{3}}{=}& \tfrac{1}{2\mu_1}\| \mu_1 \s_1 \|_{2}^2 - \tfrac{1}{2\mu_2}\| \mu_2 \s_2 \|_{2}^2 + \la    \mu_2 \s_2 -\mu_1 \s_1 ,\s_1\ra+\tfrac{W_h}{2}\| \mu_1 \s_1 -\mu_2 \s_2 \|_{2}^2 \nn\\
&\overset{\step{4}}{\leq}& \tfrac{1}{2\mu_1}\| \mu_1 \s_1 \|_{2}^2 - \tfrac{1}{2\mu_2}\| \mu_2 \s_2 \|_{2}^2 + \la    \mu_2 \s_2 -\mu_1 \s_1 ,\s_1\ra+\tfrac{1}{2 \mu_1}\| \mu_1 \s_1 -\mu_2 \s_2 \|_{2}^2 \nn\\
&\overset{}{=}& - \tfrac{\mu_2}{2} \|\s_2 \|_{2}^2 \cdot (1 - \tfrac{\mu_2}{\mu_1} )  \nn\\
&\overset{\step{5}}{\leq}&    0, \nn
\eeq
\noi where step \step{1} uses the definition of $h_{\mu_1}(\y)$ and $h_{\mu_2}(\y)$; step \step{2} uses weakly convexity of $h(\cdot)$;  step \step{3} uses the optimality of $\prox_{\mu_1}(\y)$ and $\prox_{\mu_2}(\y)$ in Equations (\ref{eq:optimality:mu:1}) and (\ref{eq:optimality:mu:2}); step \step{4} uses $W_h\leq \tfrac{1}{\mu_1}$; step \step{5} uses $1\geq \tfrac{\mu_2}{\mu_1}$.

\noi \textbf{Part (b)}. We now prove that $h_{\mu_2}(\y) - h_{\mu_1}(\y) \leq \min\{  \tfrac{\mu_1}{2 \mu_2}, 1  \} \cdot (\mu_1-\mu_2) C_h^2$. For any $\s_1 \in \partial h(\prox_{\mu_1}(\y))$ and $\s_2\in \partial h(\prox_{\mu_2}(\y))$, we have:
\beq
&& h_{\mu_2}(\y) - h_{\mu_1}(\y) \nn\\
 &\overset{\step{1}}{=}&  \tfrac{1}{2\mu_2} \|\y -  \prox_{\mu_2}(\y) \|_{2}^2 - \tfrac{1}{2\mu_1} \|\y - \prox_{\mu_1}(\y) \|_{2}^2   + h(\prox_{\mu_2}(\y)) - h(\prox_{\mu_1}(\y)) \nn\\
&\overset{\step{2}}{\leq}&   \tfrac{1}{2\mu_2} \|\y -  \prox_{\mu_2}(\y) \|_{2}^2 - \tfrac{1}{2\mu_1} \|\y - \prox_{\mu_1}(\y) \|_{2}^2 + \la \prox_{\mu_2}(\y) - \prox_{\mu_1}(\y) ,  \s_1\ra + \tfrac{W_h}{2} \| \prox_{\mu_2}(\y) -\prox_{\mu_1}(\y) \|_{2}^2  \nn\\
&\overset{\step{3}}{=}&  \tfrac{\mu_2}{2} \| \s_1 \|_{2}^2 - \tfrac{\mu_1}{2} \|\s_2\|_{2}^2 + \la \mu_1 \s_2- \mu_2 \s_1,  \s_1\ra  + \tfrac{W_h}{2} \| \mu_1 \s_2-\mu_2 \s_1\|_{2}^2    \nn\\
&=&  - \tfrac{\mu_2}{2} \|\s_1\|_{2}^2 - \tfrac{\mu_1}{2} \| \s_2\|_{2}^2  + \mu_1\la \s_1, \s_2\ra  + \tfrac{W_h}{2} \|\mu_1 \s_2-\mu_2 \s_1 \|_{2}^2  \nn\\
&\overset{\step{4}}{\leq }&  \min\{ - \tfrac{\mu_1}{2} \| \s_2\|_{2}^2  + \mu_1\la \s_1, \s_2\ra  + \tfrac{1}{2\mu_2} \|\mu_1 \s_2-\mu_2 \s_1 \|_{2}^2 - \tfrac{\mu_2}{2} \|\s_1\|_{2}^2,\nn\\
&& ~~~~~~~~- \tfrac{\mu_1}{2} \| \s_2\|_{2}^2  + \mu_1\la \s_1, \s_2\ra  + \tfrac{1}{2\mu_1} \|\mu_1 \s_2-\mu_2 \s_1 \|_{2}^2 - \tfrac{\mu_2}{2} \|\s_1\|_{2}^2\} \nn\\
%&\overset{}{=}&  \min\{ - \tfrac{\mu_1}{2} \| \s_2\|_{2}^2   + \tfrac{\mu_1^2}{2\mu_2} \| \s_2\|_{2}^2 , (\mu_1-\mu_2)\la \s_1, \s_2\ra   - \tfrac{\mu_2}{2} \|\s_1\|_{2}^2  + \tfrac{\mu_2^2}{2\mu_1} \| \s_1 \|_{2}^2  \} \nn\\
&\overset{}{=}&  \min\{  (-\mu_2+\mu_1 )\cdot \tfrac{\mu_1}{2 \mu_2} \| \s_2\|_{2}^2, (\mu_1-\mu_2)\la \s_1, \s_2\ra   - \tfrac{\mu_2}{2} \|\s_1\|_{2}^2  + \tfrac{\mu_2^2}{2\mu_1} \| \s_1 \|_{2}^2  \} \nn\\
&\overset{\step{5}}{\leq}&  \min\{  \tfrac{\mu_1}{2 \mu_2} \| \s_2\|_{2}^2\cdot  (\mu_1-\mu_2 ) , (\mu_1-\mu_2)\la \s_1, \s_2\ra     \} \nn\\
&\overset{\step{6}}{\leq}&  \min\{  \tfrac{\mu_1}{2 \mu_2} \cdot  (\mu_1-\mu_2 ) , (\mu_1-\mu_2)   \} \cdot C_h^2\nn\\
&\overset{}{=}&  \min\{  \tfrac{\mu_1}{2 \mu_2}, 1  \} \cdot (\mu_1-\mu_2)\cdot C_h^2,\nn
\eeq
\noi where step \step{1} uses the definition of $h_{\mu_1}(\y)$ and $h_{\mu_2}(\y)$; step \step{2} uses the weakly convexity of $h(\cdot)$; step \step{3} uses the optimality of $\prox_{\mu_2}(\y)$ and $\prox_{\mu_1}(\y)$ in Equations (\ref{eq:optimality:mu:1}) and (\ref{eq:optimality:mu:2}); step \step{4} uses $W_h\leq \frac{1}{\mu_1}$ and $W_h\leq \frac{1}{\mu_2}$; step \step{5} uses $\mu_2\leq\mu_1$; step \step{6} uses $\|\s_1\|\leq C_h$, $\|\s_2\|\leq C_h$, and $\la\s_1,\s_2\ra\leq \|\s_1\|\cdot\|\s_2\|\leq C_h^2$.

\end{proof}

\subsection{Proof of Lemma \ref{lemma:mu:continous}}\label{app:lemma:mu:continous}

\begin{proof}

Assume $0<\mu_2<\mu_1<\frac{1}{W_h}$, and fixing $\y\in\Rn^{m}$.

We define $h_{\mu_1}(\y)\triangleq \min_{\v} h(\v) + \frac{1}{2\mu_1}\|\v-\y\|_{2}^2$, and $\prox_{\mu_1}(\y) = \ts\arg \min_{\v}\,h(\v) + \tfrac{1}{2\mu_1} \|\v - \y\|_2^2$.

We define $h_{\mu_2}(\y)\triangleq \min_{\v} h(\v) + \frac{1}{2\mu_2}\|\v-\y\|_{2}^2$, and $\prox_{\mu_2}(\y) = \ts \arg \min_{\v}\,h(\v) + \tfrac{1}{2\mu_2} \|\v - \y\|_2^2$.

\noi By the optimality of $\prox_{\mu_1}(\y)$ and $\prox_{\mu_2}(\y)$, we obtain:
\beq
\ts\y - \prox_{\mu_1}(\y)  &\in& \ts\mu_1 \partial h(\prox_{\mu_1}(\y)) \label{eq:optimality:mu:1}\\
\ts\y - \prox_{\mu_2}(\y)  &\in& \ts\mu_2  \partial h(\prox_{\mu_2}(\y)) .\label{eq:optimality:mu:2}
\eeq

\noi \textbf{Part (a)}. We now prove that $0 \leq h_{\mu_2}(\y) - h_{\mu_1}(\y)$. For any $\s_1 \in \partial h(\prox_{\mu_1}(\y))$ and $\s_2\in \partial h(\prox_{\mu_2}(\y))$, we have:
\beq
&& h_{\mu_1}(\y) - h_{\mu_2}(\y)\nn\\
&\overset{\step{1}}{=}& \tfrac{1}{2\mu_1} \|\y -  \prox_{\mu_1}(\y) \|_{2}^2 - \tfrac{1}{2\mu_2} \|\y - \prox_{\mu_2}(\y) \|_{2}^2 + h ( \prox_{\mu_1}(\y)) - h(\prox_{\mu_2}(\y))  \nn\\
&\overset{\step{2}}{\leq}& \tfrac{1}{2\mu_1}\| \y -  \prox_{\mu_1}(\y) \|_{2}^2 - \tfrac{1}{2\mu_2}\| \y -  \prox_{\mu_2}(\y)\|_{2}^2 + \la \prox_{\mu_1}(\y)-\prox_{\mu_2}(\y),\s_1\ra+\tfrac{W_h}{2}\|\prox_{\mu_2}(\y) - \prox_{\mu_1}(\y) \|_{2}^2 \nn\\
&\overset{\step{3}}{=}& \tfrac{1}{2\mu_1}\| \mu_1 \s_1 \|_{2}^2 - \tfrac{1}{2\mu_2}\| \mu_2 \s_2 \|_{2}^2 + \la    \mu_2 \s_2 -\mu_1 \s_1 ,\s_1\ra+\tfrac{W_h}{2}\| \mu_1 \s_1 -\mu_2 \s_2 \|_{2}^2 \nn\\
&\overset{\step{4}}{\leq}& \tfrac{1}{2\mu_1}\| \mu_1 \s_1 \|_{2}^2 - \tfrac{1}{2\mu_2}\| \mu_2 \s_2 \|_{2}^2 + \la    \mu_2 \s_2 -\mu_1 \s_1 ,\s_1\ra+\tfrac{1}{2 \mu_1}\| \mu_1 \s_1 -\mu_2 \s_2 \|_{2}^2 \nn\\
&\overset{}{=}& - \tfrac{\mu_2}{2} \|\s_2 \|_{2}^2 \cdot (1 - \tfrac{\mu_2}{\mu_1} )  \nn\\
&\overset{\step{5}}{\leq}&    0, \nn
\eeq
\noi where step \step{1} uses the definition of $h_{\mu_1}(\y)$ and $h_{\mu_2}(\y)$; step \step{2} uses weakly convexity of $h(\cdot)$;  step \step{3} uses the optimality of $\prox_{\mu_1}(\y)$ and $\prox_{\mu_2}(\y)$ in Equations (\ref{eq:optimality:mu:1}) and (\ref{eq:optimality:mu:2}); step \step{4} uses $W_h\leq \tfrac{1}{\mu_1}$; step \step{5} uses $1\geq \tfrac{\mu_2}{\mu_1}$.

\noi \textbf{Part (b)}. We now prove that $h_{\mu_2}(\y) - h_{\mu_1}(\y) \leq \min\{  \tfrac{\mu_1}{2 \mu_2}, 1  \} \cdot (\mu_1-\mu_2) C_h^2$. For any $\s_1 \in \partial h(\prox_{\mu_1}(\y))$ and $\s_2\in \partial h(\prox_{\mu_2}(\y))$, we have:
\beq
&& h_{\mu_2}(\y) - h_{\mu_1}(\y) \nn\\
 &\overset{\step{1}}{=}&  \tfrac{1}{2\mu_2} \|\y -  \prox_{\mu_2}(\y) \|_{2}^2 - \tfrac{1}{2\mu_1} \|\y - \prox_{\mu_1}(\y) \|_{2}^2   + h(\prox_{\mu_2}(\y)) - h(\prox_{\mu_1}(\y)) \nn\\
&\overset{\step{2}}{\leq}&   \tfrac{1}{2\mu_2} \|\y -  \prox_{\mu_2}(\y) \|_{2}^2 - \tfrac{1}{2\mu_1} \|\y - \prox_{\mu_1}(\y) \|_{2}^2 + \la \prox_{\mu_2}(\y) - \prox_{\mu_1}(\y) ,  \s_1\ra + \tfrac{W_h}{2} \| \prox_{\mu_2}(\y) -\prox_{\mu_1}(\y) \|_{2}^2  \nn\\
&\overset{\step{3}}{=}&  \tfrac{\mu_2}{2} \| \s_1 \|_{2}^2 - \tfrac{\mu_1}{2} \|\s_2\|_{2}^2 + \la \mu_1 \s_2- \mu_2 \s_1,  \s_1\ra  + \tfrac{W_h}{2} \| \mu_1 \s_2-\mu_2 \s_1\|_{2}^2    \nn\\
&=&  - \tfrac{\mu_2}{2} \|\s_1\|_{2}^2 - \tfrac{\mu_1}{2} \| \s_2\|_{2}^2  + \mu_1\la \s_1, \s_2\ra  + \tfrac{W_h}{2} \|\mu_1 \s_2-\mu_2 \s_1 \|_{2}^2  \nn\\
&\overset{\step{4}}{\leq }&  \min\{ - \tfrac{\mu_1}{2} \| \s_2\|_{2}^2  + \mu_1\la \s_1, \s_2\ra  + \tfrac{1}{2\mu_2} \|\mu_1 \s_2-\mu_2 \s_1 \|_{2}^2 - \tfrac{\mu_2}{2} \|\s_1\|_{2}^2,\nn\\
&& ~~~~~~~~- \tfrac{\mu_1}{2} \| \s_2\|_{2}^2  + \mu_1\la \s_1, \s_2\ra  + \tfrac{1}{2\mu_1} \|\mu_1 \s_2-\mu_2 \s_1 \|_{2}^2 - \tfrac{\mu_2}{2} \|\s_1\|_{2}^2\} \nn\\
%&\overset{}{=}&  \min\{ - \tfrac{\mu_1}{2} \| \s_2\|_{2}^2   + \tfrac{\mu_1^2}{2\mu_2} \| \s_2\|_{2}^2 , (\mu_1-\mu_2)\la \s_1, \s_2\ra   - \tfrac{\mu_2}{2} \|\s_1\|_{2}^2  + \tfrac{\mu_2^2}{2\mu_1} \| \s_1 \|_{2}^2  \} \nn\\
&\overset{}{=}&  \min\{  (-\mu_2+\mu_1 )\cdot \tfrac{\mu_1}{2 \mu_2} \| \s_2\|_{2}^2, (\mu_1-\mu_2)\la \s_1, \s_2\ra   - \tfrac{\mu_2}{2} \|\s_1\|_{2}^2  + \tfrac{\mu_2^2}{2\mu_1} \| \s_1 \|_{2}^2  \} \nn\\
&\overset{\step{5}}{\leq}&  \min\{  \tfrac{\mu_1}{2 \mu_2} \| \s_2\|_{2}^2\cdot  (\mu_1-\mu_2 ) , (\mu_1-\mu_2)\la \s_1, \s_2\ra     \} \nn\\
&\overset{\step{6}}{\leq}&  \min\{  \tfrac{\mu_1}{2 \mu_2} \cdot  (\mu_1-\mu_2 ) , (\mu_1-\mu_2)   \} \cdot C_h^2\nn\\
&\overset{}{=}&  \min\{  \tfrac{\mu_1}{2 \mu_2}, 1  \} \cdot (\mu_1-\mu_2)\cdot C_h^2,\nn
\eeq
\noi where step \step{1} uses the definition of $h_{\mu_1}(\y)$ and $h_{\mu_2}(\y)$; step \step{2} uses the weakly convexity of $h(\cdot)$; step \step{3} uses the optimality of $\prox_{\mu_2}(\y)$ and $\prox_{\mu_1}(\y)$ in Equations (\ref{eq:optimality:mu:1}) and (\ref{eq:optimality:mu:2}); step \step{4} uses $W_h\leq \frac{1}{\mu_1}$ and $W_h\leq \frac{1}{\mu_2}$; step \step{5} uses $\mu_2\leq\mu_1$; step \step{6} uses $\|\s_1\|\leq C_h$, $\|\s_2\|\leq C_h$, and $\la\s_1,\s_2\ra\leq \|\s_1\|\cdot\|\s_2\|\leq C_h^2$.

\end{proof}

\subsection{Proof of Lemma \ref{lemma:lip:mu}}\label{app:lemma:lip:mu}

\begin{proof}

Assume $0<\mu_2<\mu_1 \leq \tfrac{1}{2 W_h}$, and fixing $\y\in\Rn^{m}$.

Using the result in Lemma \ref{lemma:well:knwon:3}, we establish that the gradient of $h_{\mu}(\y)$ \textit{w.r.t} $\y$ can be computed as:
\beq
\nabla h_{\mu}(\y) = \mu^{-1}(\y-\prox_{\mu}(\y)).\nn
\eeq
\noi The gradient of the mapping $\nabla h_{\mu}(\y)$ \textit{w.r.t.} the variable $1/\mu$ can be computed as: $\nabla_{1/\mu} \left(\nabla h_{\mu}(\y)\right) =\y-\prox_{\mu}(\y)$. We further obtain:
\beq
\|\nabla_{1/\mu} \left(\nabla h_{\mu}(\y)\right)\| = \|\y-\prox_{\mu}(\y)\|  \overset{\step{1}}{=} \mu \| \partial h(\prox_{\mu}(\y)) \| \leq \mu C_h.\nn
\eeq
Here, step \step{1} uses the optimality of $\prox_{\mu}(\y)$ that: $\zero \in \partial h( \prox_{\mu}(\y) ) + \tfrac{1}{\mu} ( \prox_{\mu}(\y) - \y)$. Therefore, for all $\mu\in(0,\tfrac{1}{2 W_h}]$, we have:
\beq
\frac{\|\nabla h_{\mu}(\y) - \nabla h_{\mu'}(\y) \|_{2}}{|1/\mu-1/\mu'|} \leq \mu C_h.\nn
\eeq
\noi Letting $\mu=\mu_1$ and $\mu'=\mu_2$, we have: $\|\nabla h_{\mu_1}(\y) - \nabla h_{\mu_2}(\y) \|_{2} \leq |1 - \mu_1 /\mu_2| C_h = (\mu_1 /\mu_2-1)C_h$.

    \end{proof}

\subsection{Proof of Lemma \ref{lemma:smoothing:problem:prox}}\label{app:lemma:smoothing:problem:prox}

\begin{proof}

%The proof of this lemma parallels that of Lemma 1 in \cite{li2022riemannian}, though its results cannot be directly applied in this scenario since $W_h\neq 0$.

We consider the following optimization problem:
\beq\label{eq:subprob:y:lemma:0}
\bar{\y} = \arg \min_{\y} h_{\mu}(\y)  + \tfrac{\beta}{2}\| \y - \mathbf{b}\|_{2}^2.
\eeq
\noi Given $h_{\mu}(\y)$ being $(\mu^{-1})$-weakly convex and $\beta>\mu^{-1}$, Problem (\ref{eq:subprob:y:lemma:0}) becomes strongly convex and has a unique optimal solution, which leads to the following equivalent problem:
\beq
\ts (\bar{\y},\breve{\y})= \arg \min_{\y,\y'} h(\y') + \tfrac{1}{2\mu} \|\y-\y'\|_{2}^2 + \tfrac{ \beta}{2}\|\y-\mathbf{b} \|_{2}^2, \label{eq:subprob:y:lemma}\nn
\eeq
\noi
We have the following first-order optimality conditions for $(\bar{\y},\breve{\y})$:
\beq
\tfrac{1}{\mu} (\bar{\y} -\breve{\y}) &=& \beta (\mathbf{b}-\bar{\y} )  \label{eq:gUY:0}\\
\tfrac{1}{\mu} (\bar{\y}-\breve{\y})  &\in&  \partial h(\breve{\y}). \label{eq:gUY}
\eeq

\textbf{Part (a)}. We have the following results:
\beq \label{eq:inclusion:hyb}
\zero &\overset{\step{1}}{\in}& \partial h(\breve{\y}) + \tfrac{1}{\mu} ( \breve{\y} - \bar{\y}) \nn\\
&\overset{\step{2}}{=}&  \partial h(\breve{\y}) + \tfrac{1}{\mu} (\breve{\y} -  \tfrac{1}{ 1/\mu+\beta} ( \tfrac{1}{\mu} \breve{\y}  + \beta \mathbf{b} ) ) \nn\\
%& = &  \partial h(\bar{\v}) +   \tfrac{1}{\mu} \bar{\v} -  \frac{1}{1+\mu\beta} \frac{1}{\mu}\bar{\v}  - \frac{\beta \mathbf{b}}{1+\mu\beta}   ) \nn\\
%& = &  \partial h(\bar{\v}) +   \tfrac{1}{\mu} \bar{\v} (1 - \frac{1}{1+\mu\beta} )  - \frac{\beta \mathbf{b}}{1+\mu\beta}   ) \nn\\
%& = &  \partial h(\bar{\v}) +   \tfrac{1}{\mu} \bar{\v} ( \frac{\mu\beta}{1+\mu\beta} )  - \frac{\beta \mathbf{b}}{1+\mu\beta}   ) \nn\\
&\overset{}{=}&  \partial h(\breve{\y}) +\tfrac{\beta}{1+\mu\beta} (\breve{\y} - \mathbf{b}),
\eeq
\noi where step \step{1} uses Equality (\ref{eq:gUY}); step \step{2} uses Equality (\ref{eq:gUY:0}) that $\bar{\y} = \tfrac{1}{ 1/\mu +\beta} ( \tfrac{1}{\mu} \breve{\y}  + \beta \mathbf{b} )$. The inclusion in (\ref{eq:inclusion:hyb}) implies that:
\beq
\breve{\y} = \arg \min_{\v} h(\breve{\y}) + \tfrac{1}{2 } \cdot \tfrac{\beta}{1+\mu\beta} \|\breve{\y}-\mathbf{b}\|_{2}^2.\nn
\eeq

\textbf{Part (b)}. Combining Equalities (\ref{eq:gUY:0}) and (\ref{eq:gUY}), we have: $\beta (\mathbf{b}-\bar{\y}) \in \partial h(\breve{\y})$.

\textbf{Part (c)}. In view of Equation (\ref{eq:gUY}), we have: $\bar{\y}-\breve{\y} = \mu \partial h(\breve{\y})$, leading to: $\|\breve{\y} - \bar{\y}\| \leq \mu C_h$.

\end{proof}

\subsection{Proofs for Lemma \ref{lemma:bound:PXX:XX}} \label{app:lemma:bound:PXX:XX}

\begin{proof}

We let $\Deltas \in \Rn^{n\times r}$ and $\X\in\MM$. We define $\mathbf{U}\triangleq \Deltas\trans \X \in \Rn^{r\times r}$.

We derive the following results:
\beq
&&\|\Proj_{\T_{\X}\MM}(\Deltas)\|_{\fro}^2 - \|\Deltas\|_{\fro}^2\nn\\
&\overset{\step{1}}{=} & \|  \Deltas - \tfrac{1}{2}\X (\Deltas\trans \X + \X \trans  \Deltas) \|_{\fro}^2 - \|\Deltas\|_{\fro}^2\nn\\
&\overset{}{=} &\tfrac{1}{4}\|\X (\Deltas\trans \X + \X \trans  \Deltas)\|_{\fro}^2 - \la \Deltas, \X (\Deltas\trans \X + \X \trans  \Deltas)\ra     \nn\\
&\overset{\step{2}}{=} &\tfrac{1}{4}\|\Deltas\trans \X + \X \trans  \Deltas\|_{\fro}^2 - \la \Deltas, \X (\Deltas\trans \X + \X \trans  \Deltas)\ra    \nn\\
&\overset{\step{3}}{=} & \tfrac{1}{4}\|\mathbf{U} + \mathbf{U} \trans \|_{\fro}^2 - \la \mathbf{U} + \mathbf{U}\trans,\mathbf{U} \ra   \nn\\
&\overset{\step{4}}{=} &\tfrac{1}{4}\|\mathbf{U} + \mathbf{U} \trans \|_{\fro}^2 -  \la \mathbf{U} + \mathbf{U} \trans, \mathbf{U} + \mathbf{U}\trans \ra \cdot \tfrac{1}{2} \nn\\
&\overset{}{=} & - \tfrac{1}{4}\|\mathbf{U} + \mathbf{U} \trans \|_{\fro}^2 \leq 0, \nn
\eeq
\noi where step \step{1} uses $\Proj_{\T_{\X}\MM}(\Deltas) = \Deltas - \tfrac{1}{2}\X (\Deltas \trans \X+\X\trans \Deltas)$ for all $\Deltas \in \Rn^{n\times r}$ \cite{absil2008optimization}; step \step{2} uses the fact that $\|\X \mathbf{P}\|_{\fro}^2 = \tr(\mathbf{P}\X\trans \X \mathbf{P}\trans) = \|\mathbf{P}\|_{\fro}^2$ for all $\X\in\MM$; step \step{3} uses the definition of $\mathbf{U}\triangleq \Deltas\trans \X$; step \step{4} uses the symmetric properties of the matrix $(\mathbf{U} + \mathbf{U}\trans)$.

\end{proof}

\subsection{Proof of Lemma \ref{lemma:GD:bound}}
\label{app:lemma:GD:bound}

\begin{proof}

We let $\rho>0$, $\G\in\Rn^{n\times r}$, and $\X\in\MM$.

We define $\U \triangleq \G\trans\X$, and $\GGG_{\rho} \triangleq \G-\rho\X\G\trans\X - (1-\rho)\X\X\trans\G$.

First, we have the following equalities:
\beq \label{eq:GD}
\la \G,\GGG_{\rho}\ra & = &\la \G,\G-\rho\X\G\trans\X - (1-\rho)\X\X\trans\G\ra \nn\\
&\overset{}{=}& \la \G,\G \ra - \rho \tr(\G\trans \X\G\trans \X ) - (1-\rho) \tr(\G\trans \X\X\trans \G)   \nn\\
&\overset{\step{1}}{=}& \la \G,\G \ra - \rho \tr(\U \U) - (1-\rho) \tr(\U\U\trans)  ,
\eeq
\noi where step \step{1} uses $\U \triangleq \G\trans\X$.

Second, we derive the following equalities:
\beq\label{eq:DD}
\|\GGG_{\rho}\|_{\fro}^2 & = &\la \rho\X\G\trans\X + (1-\rho)\X\X\trans\G -\G,\rho\X\G\trans\X + (1-\rho)\X\X\trans\G -\G\ra \nn\\
%& = & \la \rho\X\G\trans\X ,\rho\X\G\trans\X \ra + \la \rho\X\G\trans\X , (1-\rho)\X\X\trans\G \ra + \la \rho\X\G\trans\X , -\G\ra  \nn\\
%&   & + \la (1-\rho)\X\X\trans\G,\rho\X\G\trans\X  \ra + \la (1-\rho)\X\X\trans\G, (1-\rho)\X\X\trans\G \ra + \la (1-\rho)\X\X\trans\G, -\G\ra  \nn\\
%&   &+\la-\G,\rho\X\G\trans\X \ra +\la-\G,(1-\rho)\X\X\trans\G\ra +\la \G, \G\ra \nn\\
%&\overset{}{=}& \rho^2 \tr( \X\trans \G \G\trans\X ) +\rho (1-\rho) \tr(\X\trans \G \X\trans\G) - \rho \tr( \X\trans \G \X\trans \G)  \nn\\
%&   & + (1-\rho) \rho \tr(\G\trans \X \G\trans\X ) + (1-\rho)^2 \tr(\G\trans \X \X\trans\G \ra - (1-\rho) \tr( \G\trans \X\X\trans \G)  \nn\\
%&   & - \rho \tr(\G\trans \X\G\trans\X)  - (1-\rho) \tr(\G\trans \X\X\trans\G) + \la \G,\G \ra \nn\\
&\overset{\step{1}}{=}& \rho^2 \tr(\U\trans \U ) +\rho (1-\rho) \tr(\U\trans \U\trans ) - \rho \tr( \U\trans \U\trans)  \nn\\
&   & + (1-\rho) \rho \tr( \U \U  ) + (1-\rho)^2 \tr( \U \U\trans ) - (1-\rho) \tr(  \U \U\trans  )  \nn\\
&   & - \rho \tr( \U  \U )  - (1-\rho) \tr( \U  \U\trans  ) + \la \G,\G \ra \nn\\
%&\overset{\step{2}}{=}& \{ \rho^2 +  (1-\rho)^2 - (1-\rho) - (1-\rho) \} \cdot \tr(\U\trans \U ) \nn\\
%&& + \{\rho (1-\rho) -\rho + (1-\rho) \rho - \rho \}\cdot \tr(\U \U )  + \la \G,\G \ra \nn\\
&\overset{\step{2}}{=}&  ( 2\rho^2 -  1 )  \cdot \tr(\U\trans \U ) - 2\rho^2 \cdot \tr(\U \U )  + \la \G,\G \ra,
\eeq
\noi where step \step{1} uses $\U \triangleq \G\trans\X$ and $\X\trans \X=\I_r$; step \step{2} uses $\tr(\U\trans \U\trans)=\tr(\U\U)$.

Third, we have:
\beq \label{eq:GU}
 \tr(\G\trans \G) - \tr(\U\trans \U) \overset{\step{1}}{=}  \la  \G\G\trans,\I_n - \X\X\trans \ra \overset{\step{2}}{\geq} 0,
\eeq
\noi where step \step{1} uses $\U \triangleq \G\trans\X$; step \step{2} uses the fact that the matrix $(\I_n-\X\X\trans)$ only contains eigenvalues that are $0$ or $1$.

\textbf{Part (a-i)}. We now prove that $\max(1,2\rho) \la \G,\GGG_{\rho}\ra \geq \|\GGG_{\rho}\|_{\fro}^2$. We discuss two cases. Case (\bfit{i}): $\rho \in (0,\frac{1}{2}]$. We have:
$$\|\GGG_{\rho}\|_{\fro}^2 - \la \G,\GGG_{\rho}\ra \overset{\step{1}}{=}  (2\rho^2-\rho ) \cdot (   \tr(\U\U\trans)-\tr(\U \U))   \overset{\step{2}}{\leq } 0,$$
\noi where step \step{1} uses Inequalities (\ref{eq:GD}) and (\ref{eq:DD}); step \step{2} uses $2\rho^2-\rho \leq 0$ for all $\rho \in (0,\frac{1}{2}]$, and $\tr(\U \U) \leq \tr(\U\U\trans)$ for all $\U\in \Rn^{r\times r}$. \\
\noi Case (\bfit{ii}): $\rho \in [\frac{1}{2},\infty)$. We have: $$\ts \|\GGG_{\rho}\|_{\fro}^2 -  2 \rho \la  \G ,\GGG_{\rho}\ra\overset{\step{1}}{=} ( 2 \rho- 1) (\tr(\U\U\trans) - \la \G,\G \ra)  \overset{\step{2}}{\leq} 0,$$
\noi where step \step{1} uses Inequalities (\ref{eq:GD}) and (\ref{eq:DD}); step \step{2} uses $2 \rho- 1\geq 0$ for all $\rho \in [\frac{1}{2},\infty)$, and Inequality(\ref{eq:GU}). Therefore, we conclude that: $\max(1,2\rho) \la \G,\GGG_{\rho}\ra \geq \|\GGG_{\rho}\|_{\fro}^2$.

\textbf{Part (a-ii)}. We now prove that $\|\GGG_{\rho}\|_{\fro}^2\geq \min(1,\rho^2)\|\GGG_{1}\|_{\fro}^2$. We consider two cases. Case (\bfit{i}): $\rho\in(0,1]$. We have:
$$\ts \rho^2 \|\GGG_{1}\|_{\fro}^2 -  \|\GGG_{\rho}\|^2_{\fro}\overset{\step{1}}{=}     \ts (1- \rho^2 ) (\tr(\U\trans \U ) -\la \G,\G \ra )\overset{\step{2}}{\leq} 0,$$ %\rho^2 \{ \tr(\U\trans \U ) - 2 \tr(\U \U )  + \la \G,\G \ra \} - \{ ( 2\rho^2 -  1 )  \cdot \tr(\U\trans \U ) - 2\rho^2 \cdot \tr(\U \U )  + \la \G,\G \ra\}
\noi where step \step{1} uses Inequalities (\ref{eq:GD}) and (\ref{eq:DD}); step \step{2} uses $1-\rho^2\geq0$, and Inequality (\ref{eq:GU}). \\
Case (\bfit{ii}): $\rho\in(1,\infty)$. We have:
$$\|\GGG_{1}\|_{\fro}^2 -  \|\GGG_{\rho}\|^2_{\fro}\overset{\step{1}}{=}  (2 - 2\rho^2) (\tr(\U\trans \U ) - \tr(\U \U )) \overset{}{\leq} 0,$$ %\{ \tr(\U\trans \U ) - 2 \tr(\U \U )  + \la \G,\G \ra \}  - \{ ( 2\rho^2 -  1 )  \cdot \tr(\U\trans \U ) - 2\rho^2 \cdot \tr(\U \U )  + \la \G,\G \ra\} \overset{}{=}
\noi where step \step{1} uses Inequality (\ref{eq:DD}); step \step{2} uses $4 \rho^2 - 1\leq 0$ for all $\rho\in(0,\frac{1}{2}]$, and the fact that $\tr(\U \U) - \tr(\U\U\trans) \leq 0$ for all $\U\in \Rn^{r\times r}$. Therefore, we conclude that: $\min(1,\rho^2) \|\GGG_{1}\|_{\fro}^2 \leq \|\GGG_{\rho}\|_{\fro}^2$.

\textbf{Part (b-i)}. We now prove that $\|\GGG_{\rho}\|_{\fro} \geq \min(1,2 \rho) \|\GGG_{1/2}\|_{\fro}$. We consider two cases. Case (\bfit{i}): $\rho\in(0,\frac{1}{2}]$. We have:
$$(2\rho)^2 \|\GGG_{1/2}\|_{\fro}^2 - \|\GGG_{\rho}\|_{\fro}^2 \overset{\step{1}}{=}  ( 4 \rho^2 - 1) \cdot ( \tr(\G\trans \G) - \tr(\U\trans \U)) \overset{\step{2}}{\leq} 0,$$ %(2\rho)^2 ( -\tfrac{1}{2}\tr(\U\trans \U ) - \tfrac{1}{2}   \tr(\U \U )  + \la \G,\G \ra  ) -     ( 2\rho^2 -  1) \cdot \tr(\U\trans \U ) + 2\rho^2   \tr(\U \U )  - \la \G,\G \ra\overset{}{=}
\noi where step \step{1} uses Inequality (\ref{eq:DD}); step \step{2} uses $4 \rho^2 - 1\leq 0$ for all $\rho\in(0,\frac{1}{2}]$, and Inequality (\ref{eq:GU}).\\
Case (\bfit{ii}): $\rho\in(\frac{1}{2},\infty)$. We have:
$$\|\GGG_{1/2}\|_{\fro}^2 - \|\GGG_{\rho}\|_{\fro}^2\overset{\step{1}}{=}    (2\rho^2 - \tfrac{1}{2}) \cdot ( \tr(\U \U ) -  \tr(\U\trans \U )) \overset{\step{2}}{\leq}0,$$ %-\tfrac{1}{2}\tr(\U\trans \U ) - \tfrac{1}{2} \cdot \tr(\U \U )  + \la \G,\G \ra - ( 2\rho^2 -  1) \cdot \tr(\U\trans \U ) + 2\rho^2 \cdot \tr(\U \U )  - \la \G,\G \ra \overset{}{=}
\noi where step \step{1} uses Inequalities (\ref{eq:GD}) and (\ref{eq:DD}); step \step{2} uses $2\rho^2 - \tfrac{1}{2}\geq 0$ for all $\rho\in(\frac{1}{2},\infty)$, and the fact that $\tr(\U \U) - \tr(\U\U\trans) \leq 0$ for all $\U\in \Rn^{r\times r}$. Therefore, we conclude that $\|\GGG_{\rho}\|_{\fro} \geq \min(1,2 \rho) \|\GGG_{1/2}\|_{\fro}$.

\textbf{Part (b-ii)}. We now prove that $\|\GGG_{\rho}\|_{\fro} \leq \max(1,2 \rho) \|\GGG_{1/2}\|_{\fro}$. We consider two cases. Case (\bfit{i}): $\rho\in(0,\frac{1}{2}]$. We have:
$$\|\GGG_{1/2}\|_{\fro}^2 - \|\GGG_{\rho}\|_{\fro}^2\overset{\step{1}}{=} (2\rho^2 - \tfrac{1}{2}) \cdot ( \tr(\U \U ) -  \tr(\U\trans \U )) \overset{\step{2}}{\geq}0,
$$
\noi where step \step{1} uses Inequality (\ref{eq:DD}); step \step{2} uses $2\rho^2 - \tfrac{1}{2} \leq 0$ for all $\rho\in(0,\frac{1}{2}]$, and the fact that $\tr(\U \U) - \tr(\U\U\trans) \leq 0$ for all $\U\in \Rn^{r\times r}$. \\
Case (\bfit{ii}): $\rho\in(\frac{1}{2},\infty)$. We have:
$$(2\rho)^2 \|\GGG_{1/2}\|_{\fro}^2 - \|\GGG_{\rho}\|_{\fro}^2\overset{\step{1}}{=}   ( 4 \rho^2 - 1) \cdot ( \tr(\G\trans \G) - \tr(\U\trans \U))\overset{\step{2}}{\geq} 0,$$
\noi where step \step{1} uses Inequalities (\ref{eq:GD}) and (\ref{eq:DD}); step \step{2} uses $ 4 \rho^2 - 1\geq 0$ for all $\rho\in(\frac{1}{2},\infty)$, and Inequality (\ref{eq:GU}). Therefore, we conclude that: $\|\GGG_{\rho}\|_{\fro} \geq \min(1,2 \rho) \|\GGG_{1/2}\|_{\fro}$.

\end{proof}

\subsection{Proof of Lemma \ref{lemma:subgradient:R:bound}}
\label{app:lemma:subgradient:R:bound}

\begin{proof}

Recall that the following first-order optimality conditions are equivalent for all $\X\in\Rn^{n\times r}$:
\beq
\left(\mathbf{0}\in \partial \iota_{\MM}(\X) + \nabla f(\X) \right) \Leftrightarrow \left(\mathbf{0} \in \Proj_{\T_{\X}\MM}(\nabla f(\X))\right) .\label{eq:opt:cond:equivalent}
\eeq
\noi Therefore, we derive the following results:
\beq\label{eq:HHHH}
\dist(\mathbf{0},   \partial \iota_{\MM}(\X)+ \nabla f(\X)) & =& \inf_{\mathbf{R} \in \nabla f(\X) + \partial \iota_{\MM}(\X) } \|  \mathbf{R}\|_{\fro}\nn\\
&\overset{\step{1}}{=}&  \inf_{\mathbf{R} \in \Proj_{\T_{\X}\MM}(\nabla f(\X))} \|  \mathbf{R}\|_{\fro} \nn\\
&\overset{}{=}&  \| \Proj_{\T_{\X}\MM}(\nabla f(\X)) \|_{\fro} \nn\\
&\overset{\step{2}}{=}&  \| \nabla f(\X)  - \tfrac{1}{2} \X (\X\trans \nabla f(\X)  + \nabla f(\X) \trans \X)\|_{\fro} \nn\\
&\overset{}{=}&  \| (\I-\tfrac{1}{2}\X\X\trans) (\nabla f(\X)  - \X \nabla f(\X) \trans \X)\|_{\fro} \nn\\
&\overset{\step{3}}{\leq}&  \|\nabla f(\X)  - \X \nabla f(\X) \trans \X\|_{\fro}, \nn
\eeq
\noi where step \step{1} uses Formulation (\ref{eq:opt:cond:equivalent}); step \step{2} uses $\Proj_{\T_{\X}\MM}(\Deltas) = \Deltas - \tfrac{1}{2}\X (\Deltas \trans \X+\X\trans \Deltas)$ for all $\Deltas\in\Rn^{n\times r}$ \cite{absil2008optimization}; step \step{3} uses the norm inequality $\|\mathbf{A}\mathbf{B}\|_{\fro}\leq \|\mathbf{A}\|\|\mathbf{B}\|_{\fro}$, and fact that the matrix $\I-\frac{1}{2}\X\X\trans$ only contains eigenvalues that are $\frac{1}{2}$ or $1$.

\end{proof}

\section{Proofs for Section \ref{sect:iterC}}\label{app:sect:iterC}
\subsection{Proof of Lemma \ref{lemma:bounding:dual}}\label{app:lemma:bounding:dual}

\begin{proof}

We define $\mu^t=\tau/\beta^t$.

We define $L(\X,\y,\z,\beta) \triangleq f(\X) - g(\X) + h_{\tau/\beta}(\y) + \la \z,\AA(\X) - \y \ra + \frac{\beta}{2}\|\AA(\X) - \y\|_{2}^2$.

\textbf{Part (a-i)}. Using the first-order optimality condition of $\y^{t+1} \in \arg \min_{\y} L(\X^{t+1},\y,\z^t,\beta^t)$ in Algorithm \ref{alg:main}, for all $t\geq 0$, we have:
\beq
\zero  &=&  \nabla h_{\mu^t} (\y^{t+1})   + \beta^t  (\y^{t+1} - \y^t) + \nabla_\y \SS (\X^{t+1},\y^t,\z^t,\beta^t)
 \nn\\
&\overset{\step{1}}{=}&   \nabla h_{\mu^t} (\y^{t+1})  + \beta^t  (\y^{t+1} - \y^t)- \z^t +  \beta^t (\y^t-\AA(\X^{t+1}))   \nn\\
&=& \nabla h_{\mu^t} (\y^{t+1}) - \z^t +  \beta^t (\y^{t+1}-\AA(\X^{t+1}))  \nn\\
&\overset{\step{2}}{=}&  \nabla h_{\mu^t} (\y^{t+1}) - \z^t +   \tfrac{1}{\sigma} (\z^t - \z^{t+1})  ,\label{eq:Z:hy}
\eeq
\noi where step \step{1} uses $\nabla_{\y}\SS (\X^{t+1},\y,\z^t,\beta^t) = - \z^t +  \beta^t (\y -\AA(\X^{t+1}))$; step \step{2} uses $\z^{t+1} = \z^t + \sigma \beta^t (\AA (\X^{t+1}) - \y^{t+1})$.

\noi \textbf{Part (a-ii)}. We obtain:
\beq
\partial h(\breve{\y}^{t+1}) -\z^t &\overset{\step{1}}{\ni}&  \beta^t (\mathbf{b}- \y^{t+1} ) -\z^t  \nn\\
&\overset{\step{2}}{=}& \beta^t \y^t -  \nabla_{\y} \mathcal{S}^t(\X^{t+1},\y^t,\z^t,\beta^t)  - \beta^t \y^{t+1}  -\z^t\nn\\
&\overset{\step{3}}{=}& \beta^t \y^t  - \beta^t (\y^t - \AA(\X^{t+1}))  - \beta^t \y^{t+1}   \nn\\
&\overset{}{=}& \beta^t ( \AA(\X^{t+1})- \y^{t+1} )  \nn\\
&\overset{\step{4}}{=}& \tfrac{1}{\sigma}(\z^{t+1}-\z^t), \nn
\eeq
\noi where step \step{1} uses the result in Lemma \ref{lemma:smoothing:problem:prox} that $\beta^t (\mathbf{b}- \y^{t+1} ) \in \partial h(\breve{\y}^{t+1})$; step \step{2} uses $\mathbf{b} \triangleq \y^t -  \nabla_{\y} \mathcal{S}^t(\X^{t+1},\y^t,\z^t,\beta^t) /\beta^t$, as shown in Algorithm \ref{alg:main}; step \step{3} uses $\nabla_{\y} \mathcal{S}^t(\X^{t+1},\y,\z^t,\beta^t) = -\z^t  + \beta^t (\y - \AA(\X^{t+1}))$; step \step{4} uses $\z^{t+1} -  \z^t = \sigma \beta^t (\AA (\X^{t+1}) - \y^{t+1})$.

\textbf{Part (b)}. First, we derive:
\beq \label{eq:grad:h:yy}
&&\|\nabla h_{\mu^{t-1}}(\y^{t})-\nabla h_{\mu^{t}}(\y^{t+1})\| \nn\\
%&=& \|\nabla h_{\mu^{t}}(\y^{t+1}) - \nabla h_{\mu^{t}}(\y^{t}) + \nabla h_{\mu^{t}}(\y^{t}) - \nabla h_{\mu^{t-1}}(\y^{t})\|\nn\\
&\overset{\step{1}}{\leq}& \|  \nabla h_{\mu^{t-1}}(\y^{t}) - \nabla h_{\mu^{t}}(\y^{t}) \|+ \|\nabla h_{\mu^{t}}(\y^{t}) - \nabla h_{\mu^{t}}(\y^{t+1})  \|    \nn\\
&\overset{\step{2}}{\leq}& \|\nabla h_{\mu^{t}}(\y^{t}) - \nabla h_{\mu^{t-1}}(\y^{t})\|+ \tfrac{1}{\mu^t} \|\y^{t+1}-\y^{t}\|     \nn\\
&\overset{\step{3}}{\leq}&  C_h  (\tfrac{\mu^{t-1}}{\mu^{t}} - 1)+ \tfrac{1}{\mu^t} \|\y^{t+1}-\y^{t}\| \nn\\
&\overset{\step{4}}{=}&  C_h  (\tfrac{\beta^{t}}{\beta^{t-1}} - 1)+ \tfrac{\beta^t}{\tau} \|\y^{t+1}-\y^{t}\|,
\eeq
\noi where step \step{1} uses $\|\a-\b\|\leq \|\a-\mathbf{c}\|+\|\mathbf{c}-\b\|$; step \step{2} uses the fact that the function $h_{\mu^t}(\y)$ is $\tfrac{1}{\mu^t}$-smooth \textit{w.r.t.} $\y$ that: $\|\nabla h_{\mu^t}(\y^{t+1}) - \nabla h_{\mu^{t}}(\y^{t}) \| \leq  \tfrac{1}{\mu^t} \|\y^{t+1}-\y^{t}\|$; step \step{3} uses the fact that $\| \nabla h_{\mu^{t}}(\y^{t}) - \nabla h_{\mu^{t-1}}(\y^{t})    \|  \leq  ({\mu^{t-1}}/{\mu^{t}} - 1) C_h$ which holds due to Lemma \ref{lemma:lip:mu}; step \step{4} uses $\mu^t = \tfrac{\tau}{\beta^t}$.

Second, we have from Equality (\ref{eq:Z:hy}):
\beq \label{eq:Z:2}
&&\forall t\geq 0,~\zero \in \sigma \nabla h_{\mu^{t}}(\y^{t+1}) - \sigma\z^{t} + (\z^{t} - \z^{t+1}),\nn\\
&&\forall t\geq 1,~\zero \in \sigma \nabla h_{\mu^{t-1}}(\y^{t}) - \sigma \z^{t-1} + (\z^{t-1} - \z^t).\nn
\eeq
\noi Combining these two equalities yields:
\beq
\forall t\geq 1,\, \z^{t+1}-\z^t  = (\sigma-1) (\z^{t-1}-\z^{t}) + \sigma (\nabla h_{\mu^{t}}(\y^{t+1}) - \nabla h_{\mu^{t-1}}(\y^{t}).\nn
\eeq
\noi Applying Lemma \ref{lemma:a:b:sigma} with $\a^+=\z^{t+1}-\z^t$, $\a=\z^{t-1}-\z^{t}$, $\b=\sigma\{\nabla h_{\mu^{t}}(\y^{t+1}) - \nabla h_{\mu^{t-1}}(\y^{t})\}$, and $\varrho=\sigma-1 \in[0,1)$, we have:
\beq\label{eq:zzz}
&&\|\z^{t+1}-\z^t\|_2^2 - \tfrac{\varrho}{1-\varrho} (\|\z^{t-1}-\z^{t}\|_2^2 - \|\z^{t+1}-\z^t\|_2^2) \nn\\
&\leq& \tfrac{\sigma^2}{(1-\varrho)^2} \|\nabla h_{\mu^{t}}(\y^{t+1}) - \nabla h_{\mu^{t-1}}(\y^{t})\|_2^2\nn\\
&\overset{\step{1}}{\leq}&  \tfrac{2\sigma^2}{(1-\varrho)^2} \cdot  \tfrac{(\beta^t)^2}{\tau^2} \|\y^{t+1}-\y^{t}\|_2^2 + \tfrac{2\sigma^2}{(1-\varrho)^2} \cdot C^2_h  ( \tfrac{\beta^{t}}{\beta^{t-1}} - 1)^2  \nn\\
&\overset{\step{2}}{\leq}&  \underbrace{\tfrac{2\sigma^2}{(1-\varrho)^2}  \tfrac{1}{\tau^2}}_{\triangleq \dot{\sigma}} \cdot (\beta^t)^2 \|\y^{t+1}-\y^{t}\|_2^2 + \underbrace{\tfrac{2\sigma^2}{(1-\varrho)^2} \cdot C^2_h\cdot  \tfrac{6}{p} }_{\triangleq \ddot{\sigma}} \cdot ( \tfrac{\beta^0}{\beta^t} - \tfrac{\beta^0}{\beta^{t+1}}),\nn
\eeq
\noi where step \step{1} uses Inequality (\ref{eq:grad:h:yy}), and the inequality $(a+b)^2\leq 2a^2+2b^2$ for all $a,b\in \Rn$; step \step{2} uses Lemma \ref{lemma:bound:mumumu} that $(\tfrac{\beta^{t}}{\beta^{t-1}} - 1)^2 \leq \frac{6\beta^0}{p}(\tfrac{1}{\beta^t}-\tfrac{1}{\beta^{t+1}})$ for all $t\geq 1$;

%We now bound the term $\sum_{t=1}^T \|\z^{t+1}-\z^t\|_{2}^2$. We have:
%\beq
%\ts&& \ts \sum_{t=1}^T \|\z^{t+1}-\z^t\|_{2}^2\nn\\
%&\leq& \ts \sum_{t=1}^T \{  \tfrac{\sigma-1}{\sigma} ( \|\z^{t}-\z^{t-1}\|_{2}^2 - \|\z^{t+1}-\z^t\|_{2}^2) + \ts c_1 \|\beta^t(\y^{t+1}-\y^{t})\|_{2}^2  + \tfrac{K_{\mu}}{\omega}  (\tfrac{\mu^{t}}{\mu^{t+1}} - 1)^2 \}\nn\\
%&\overset{\step{1}}{\leq}& \ts \tfrac{\sigma-1}{\sigma} \|\z^{1}-\z^0\|_{2}^2 +  c_1  \sum_{t=1}^T  \|\beta^t(\y^{t+1}-\y^{t})\|_{2}^2   + \tfrac{2K_{\mu}}{\omega} ,  \nn
%\eeq
%\noi where step \step{1} uses Lemma \ref{lemma:bound:mumumu}.

\end{proof}

\subsection{Proof of Lemma \ref{lemma:smooth:LL}}\label{app:lemma:smooth:LL}

\begin{proof}

%We define $\mathcal{S}(\X,\y^t,\z^t,\beta^t) \triangleq f(\X)+ \la \z^t,\AA(\X) - \y^t\ra + \tfrac{\beta^t}{2}\|\AA(\X) - \y^t\|_{2}^2$.

\textbf{Part (a)}. We have:
\beq
\beta^{t+1} - \beta^t \cdot (1+\xi) \overset{\step{1}}{=}  \beta^0 \xi (t+1)^p -  \beta^0 \xi t^p    - \beta^t \xi \overset{\step{2}}{\leq}   \beta^0 \xi     - \beta^t \xi  \overset{\step{3}}{\leq}    0, \nn
\eeq
\noi where step \step{1} uses $\beta^t = \beta^0 (1 + \xi t^p)$; step \step{2} uses $(t+1)^p - t^p\leq 1$ for all $p\in(0,1)$; step \step{3} uses $\beta^0\leq \beta^t$ and $\xi>0$.

\textbf{Part (b)}. It holds with $\elldown = \Aup^2$ and $\ellup = \Aup^2 + L_f/\beta^0$.

\end{proof}
\subsection{Proof of Lemma \ref{lemma:bound:solution}}
\label{app:lemma:bound:solution}

\begin{proof}
We define $\overline{\rm{X}} \triangleq \sqrt{r}$, $\overline{\rm{z}}\triangleq \| \z^{0} \|+ \tfrac{ \sigma C_h }{ 2 - \sigma} $, $\overline{\rm{y}} \triangleq \Aup \sqrt{r} + \tfrac{2 \overline{\rm{z}}}{\beta^0}$, where $\sigma \in[1,2)$.

We let $\underline{\rm{\Theta}}\triangleq F(\bar{\X}) - \mu^0 C_h^2 - C_h (\Aup \sqrt{r} + \overline{\rm{y}}) - \tfrac{\overline{\rm{z}}^2}{2\beta^0}$, where $\bar{\X}$ is the optimal solution of Problem (\ref{eq:main}).

\textbf{Part (a)}. Given $\X^{t+1} \in \MM$, we have: $\|\X^t\|_{\fro}\leq \overline{\rm{X}}\triangleq \sqrt{r}$.

\textbf{Part (b)}. We show that $\|\z^t\| \leq \overline{\rm{z}}$. For all $t\geq 0$, we have:
\beq
\| \z^{t+1}\| &\overset{\step{1}}{\leq}& \| (\sigma-1) \z^{t} \|+ \| (\sigma-1)\z^{t} + \z^{t+1} \|  \nn\\
&\overset{\step{2}}{=}&  (\sigma-1) \| \z^{t} \|  + \|\sigma\partial h(\breve{\y}^{t+1}) \| \nn\\
&\overset{\step{3}}{=}& (\sigma-1) \| \z^{t} \| + \sigma C_h,\nn
\eeq
\noi step \step{1} uses the triangle inequality; step \step{2} uses $\z^{t+1} + (\sigma-1) \z^t \in \sigma\partial h(\breve{\y}^{t+1})$, as shown in Lemma \ref{lemma:bounding:dual}(\bfit{a}); step \step{3} uses $C_h$-Lipschitz continuity of $h(\y)$. Applying Lemma \ref{lemma:two:non:negative:sequences} with $\a_{t}=\|\z^{t+1}\|$, $c=\sigma C_h$, and $\varrho=\sigma-1 \in [0,1)$, we have:
\beq
\forall t\geq 0,\,\|\z^{t+1} \| \leq \|\z^0\| +  \tfrac{c}{1-\varrho} = \|\z^0\| +  \tfrac{\sigma C_h}{2-\sigma} \triangleq \overline{\rm{z}}.\nn
\eeq

\textbf{Part (c)}. We show that $\|\y^t\| \leq \overline{\rm{y}}$. For all $t\geq0$, we have:
\beq
\|\y^{t+1}\| &=& \|\AA(\X^{t+1}) - \tfrac{\z^{t+1}-\z^t}{\sigma \beta^t} \| \nn\\
&\overset{\step{1}}{\leq}& \| \AA(\X^{t+1})\| + \tfrac{1}{\beta^0} \|\z^{t+1}-\z^t\|  \nn\\
%&\overset{\step{2}}{\leq}& \Aup \|\X^{t+1}\|_{\fro} + \tfrac{1}{\beta^0} \cdot 2 \overline{\rm{z}}\nn\\
&\overset{\step{2}}{\leq}& \Aup \sqrt{r} + \tfrac{1}{\beta^0} \cdot 2 \overline{\rm{z}}
\triangleq \overline{\rm{y}} , \nn
\eeq
\noi where step \step{1} uses the triangle inequality, $\sigma\geq 1$, and $\tfrac{1}{\beta^t}\leq \tfrac{1}{\beta^0}$; step \step{2} uses $\| \AA(\X)\|_{\fro} \leq \Aup \|\X\|_{\fro} \leq \Aup \sqrt{r}$, and $\|\z^t\|\leq \overline{\rm{z}}$.

\textbf{Part (d)}. We show that $\Theta^t \geq \underline{\rm{\Theta}}$. For all $t\geq 1$, we have:
\begin{align}
 \Theta^t \triangleq&~ L(\X^{t},\y^t,\z^t,\beta^{t},\mu^{t-1})+ \mu^{t-1} C_h^2+ \TTT^t + \DDD^{t}  + \PPP^{t}  \nn\\
\overset{\step{1}}{\geq}&~ f(\X^t) - g(\X^t) + h_{\mu^{t-1}}(\y^t) + \la \z^t,\AA(\X^t) - \y^t \ra + \tfrac{\beta^t}{2}\|\AA(\X^t) - \y^t\|_{2}^2 \nn\\
 =  &~ f(\X^t) - g(\X^t) + h_{\mu^{t-1}}(\y^t) +  \tfrac{\beta^t}{2}\|\AA(\X^t) - \y^t + \z^t /\beta^t\|_{2}^2 - \tfrac{\beta^t}{2} \|\z^t /\beta^t\|_{2}^2\nn\\
\overset{\step{2}}{\geq}&~ f(\X^t) - g(\X^t) +   h_{\mu^{t-1}}(\AA(\X^{t})) -  C_h \| \AA(\X^{t}) -\y^t\|  - \tfrac{1}{2\beta^t} \|\z^t  \|_{2}^2\nn\\
\overset{\step{3}}{\geq}&~ f(\X^t) - g(\X^t) +  h(\AA(\X^{t})) - \mu^{t-1} C_h^2 -  C_h (\| \AA(\X^{t}) \| +\|\y^t\| ) - \tfrac{1}{2\beta^t} \|\z^t  \|_{2}^2\nn\\
\overset{\step{4}}{\geq}& ~F(\bar{\X}) - \mu^0 C_h^2 - C_h (\Aup \sqrt{r} + \overline{\rm{y}}) - \tfrac{\overline{\rm{z}}^2}{2\beta^0}\triangleq  \underline{\rm{\Theta}}, \nn
\end{align}
\noi where step \step{1} uses the definition of $L(\X,\y,\z,\beta)$ and the positivity of $\{\mu^{t},\TTT^t,\DDD^t,\PPP^t\}$; step \step{2} uses the $L_h$-Lipschitz continuity of $h_{\mu^{t-1}}(\y)$, ensuring $h_{\mu^{t-1}}(\y^t)\geq h_{\mu^{t-1}}(\y) - C_h \|\y^t -\y \|$, with the specific choice of $\y = \AA(\X^{t})$; step \step{3} uses $h(\y)-h_{\mu}(\y)\leq \mu C_h^2$, which has been shown in Lemma \ref{lemma:mu:continous}; step \step{4} uses $\mu^t\leq \mu^0$, $\beta^t\geq \beta^0$, $\| \AA(\X)\|\leq \Aup \|\X\|_{\fro}\leq \Aup \sqrt{r}$ for all $\X\in\MM$; $\|\y^t\|\leq \overline{\rm{y}}$, and $\|\z^t\|\leq \overline{\rm{z}}$.

 \end{proof}

\subsection{Proof of Lemma \ref{lemma:dec:non:x}}
 \label{app:lemma:dec:non:x}

\begin{proof}

We define $L(\X,\y,\z,\beta) \triangleq f(\X) - g(\X) + h_{\tau/\beta}(\y) + \la \z,\AA(\X) - \y \ra + \frac{\beta}{2}\|\AA(\X) - \y\|_{2}^2$.

We define $\mu^t\triangleq \tau/\beta^t$. We define $\DDD^t\triangleq \tfrac{\sigma-1}{2-\sigma} \tfrac{2}{\beta^{t-1}}\|\z^{t}-\z^{t-1}\|_2^2 $.

\textbf{Part (a)}. We focus on the sufficient decrease for variable $\{\beta\}$. We have:
\begin{align} \label{eq:dec:betabeta}
& ~ \varepsilon_{\beta} \beta^t \BB_{t+1}^2 + L(\X^{t+1},\y^{t+1},\z^{t+1},\beta^{t+1})  - L(\X^{t+1},\y^{t+1},\z^{t+1},\beta^{t}) \nn\\
\overset{\step{1}}{=}& ~  \varepsilon_{\beta} \beta^t (\tfrac{1}{\beta^t} - \tfrac{1}{\beta^{t+1}})\cdot\tfrac{1}{\beta^t} +  h_{\mu^{t+1}}(\y^{t+1}) -h_{\mu^t}(\y^{t+1}) + \tfrac{\beta^{t+1}-\beta^{t}}{2}\|\AA(\X^{t+1}) - \y^{t+1}\|_{2}^2   \nn\\
\overset{\step{2}}{\leq}& ~ \varepsilon_{\beta}  ( \tfrac{1}{\beta^{t}} - \tfrac{1}{\beta^{t+1}}) +  \tau C_h^2 (\tfrac{1}{\beta^{t}}-\tfrac{1}{\beta^{t+1}}) + \tfrac{\beta^{t+1}-\beta^{t}}{2 (\sigma \beta^t)^2} \| \z^{t+1}-\z^t\|_{2}^2 \nn\\
\overset{\step{3}}{\leq}& ~ \varepsilon_{\beta}  ( \tfrac{1}{\beta^{t}} - \tfrac{1}{\beta^{t+1}}) +  \tau C_h^2 (\tfrac{1}{\beta^{t}}-\tfrac{1}{\beta^{t+1}}) + \tfrac{(1+\xi)\beta^{t}-\beta^{t}}{ \sigma^2 \beta^t}  \tfrac{1}{2\beta^t}\| \z^{t+1}-\z^t\|_{2}^2 \nn\\
\overset{\step{4}}{\leq}& ~ (\varepsilon_{\beta} + \tau C_h^2) (\tfrac{1}{\beta^{t}}-\tfrac{1}{\beta^{t+1}}) + \tfrac{\xi}{ \sigma^2 }  \tfrac{1}{2\beta^t}\| \z^{t+1}-\z^t\|_{2}^2,
\end{align}
\noi where step \step{1} uses the definition of $L(\cdot,\cdot,\cdot,\cdot)$; step \step{2} uses Lemma \ref{lemma:mu:continous}, and the fact that $\AA(\X^{t+1}) - \y^{t+1} = \tfrac{1}{\sigma \beta^t} (\z^{t+1}-\z^{t})$; step \step{3} uses $\beta^{t+1}\leq (1+\xi)\beta^t$; step \step{4} uses $\beta^t( \tfrac{1}{\beta^{t}} - \tfrac{1}{\beta^{t+1}})\leq1$.

\textbf{Part (b)}. We focus on the sufficient decrease for variable $\{\z\}$. We have:
\begin{align}\label{eq:dec:zz}
&~ \varepsilon_z \beta^t \|\AA(\X^{t+1}) - \y^{t+1} \|_{2}^2 + L(\X^{t+1},\y^{t+1},\z^{t+1},\beta^{t}) - L(\X^{t+1},\y^{t+1},\z^{t},\beta^t)     \nn \\
\overset{\step{1}}{=}&~ \varepsilon_z \beta^t \|\AA(\X^{t+1}) - \y^{t+1} \|_{2}^2 + \la \AA(\X^{t+1} )- \y^{t+1},\z^{t+1}-\z^t \ra  \nn\\
\overset{\step{2}}{=}&~   (\tfrac{\varepsilon_z}{\sigma^2}+\tfrac{1}{\sigma}   )  \tfrac{1}{\beta^t} \| \z^{t+1} - \z^t \|_2^2,
\end{align}
\noi where step \step{1} uses the definition of $L(\X,\y,\z,\beta)$; step \step{2} uses $\z^{t+1} - \z^t =   \sigma \beta^t (\AA (\X^{t+1}) - \y^{t+1})$.

\textbf{Part (c)}. We focus on the sufficient decrease for variable $\{\y\}$. We have:
\begin{align} \label{eq:dec:yy}
&~  L(\X^{t+1},\y^{t+1},\z^{t},\beta^{t})  - L(\X^{t+1},\y^{t},\z^{t},\beta^{t})\nn\\
 \overset{}{=}& ~ h_{\mu^t}(\y^{t+1}) -  h_{\mu^t}(\y^{t}) + \la\y^t-\y^{t+1}, \z^t\ra + \tfrac{\beta^t}{2}\|\y^{t+1}-\AA(\X^{t+1})\|_{2}^2 - \tfrac{\beta^t}{2}\|\y^{t}-\AA(\X^{t+1})\|_{2}^2 \nn\\
 \overset{\step{1}}{=}& ~  h_{\mu^t}(\y^{t+1}) -  h_{\mu^t}(\y^{t}) + \la\y^t-\y^{t+1}, \z^t + \beta^t (\AA(\X^{t+1}) - \y^{t+1})\ra  - \tfrac{\beta^t}{2}\|\y^{t+1}-\y^t\|_{2}^2 \nn\\
 \overset{\step{2}}{=}&~ h_{\mu^t}(\y^{t+1}) -  h_{\mu^t}(\y^{t})  - \tfrac{\beta^t}{2}\|\y^{t+1}-\y^t\|_{2}^2+\la\y^t-\y^{t+1}, \z^t + \tfrac{1}{\sigma} (\z^{t+1}-\z^t) \ra \nn\\
 \overset{\step{3}}{=}&~h_{\mu^t}(\y^{t+1}) -  h_{\mu^t}(\y^{t}) - \tfrac{\beta^t}{2}\|\y^{t+1}-\y^t\|_{2}^2+ \la\y^t-\y^{t+1}, \nabla h_{\mu^t} (\y^{t+1}) \ra \nn\\
 \overset{\step{4}}{\leq}&~ \tfrac{1}{2\mu^t} \|\y^{t+1}-\y^t\|_{2}^2 - \tfrac{\beta^t}{2}\|\y^{t+1}-\y^t\|_{2}^2 \nn\\
 \overset{\step{5}}{=}&~  (\tfrac{1}{\tau} -1)\tfrac{\beta^t}{2}\|\y^{t+1}-\y^t\|_{2}^2,
\end{align}
\noi where step \step{1} uses the Pythagoras Relation that $\tfrac{1}{2}\|\y^+-\a\|_{2}^2 - \tfrac{1}{2}\|\y-\a\|_{2}^2    =  - \tfrac{1}{2}\|\y^+-\y\|_{2}^2 + \la \y- \y^+, \a-\y^+\ra $ for all $\y,\y^+,\a\in \Rn^{m}$; step \step{2} uses $\z^{t+1}-\z^t = \sigma \beta^t (\AA(\X^{t+1}) - \y^{t+1})$; step \step{3} uses $\nabla h_{\mu^t} (\y^{t+1})=\z^t + \tfrac{1}{\sigma} (\z^{t+1}-\z^t)$, as shown in Lemma \ref{lemma:bounding:dual}(\bfit{a}); step \step{4} uses the fact that the function $h_{\mu^{t}}(\y)$ is $(1/\mu^t)$-weakly convex \textit{w.r.t} $\y$; step \step{5} uses $\mu^t\beta^t=\tau$.

\textbf{Part (d)}. We focus on the sufficient decrease for variable $\{\X\}$. We have:
\beq  \label{eq:dec:XX}
L(\X^{t+1},\y^{t},\z^{t},\beta^{t})  - L(\X^{t},\y^{t},\z^{t},\beta^{t}) = \mathfrak{X}.
\eeq

Adding Inequalities (\ref{eq:dec:betabeta}), (\ref{eq:dec:zz}), (\ref{eq:dec:yy}), and (\ref{eq:dec:XX}) together, we have:
\begin{align}
&~ \ts \varepsilon_{\beta} \beta^t \BB_{t+1}^2 + \varepsilon_y \beta^t \|\y^{t+1}-\y^t \|_{2}^2 + \varepsilon_z \beta^t \|\AA(\X^{t+1}) - \y^{t+1} \|_{2}^2 \nn\\
&~ + L(\X^{t+1},\y^{t+1},\z^{t+1},\beta^{t+1}) - L(\X^{t},\y^t,\z^t,\beta^t)  \nn\\
\overset{}{\leq}&~ (\varepsilon_{\beta} + \tau C_h^2) (\tfrac{1}{\beta^{t}}-\tfrac{1}{\beta^{t+1}}) + (\tfrac{1}{2\tau} -\tfrac{1}{2}) \beta^t \|\y^{t+1}-\y^t\|_{2}^2 + ( \tfrac{\xi}{2\sigma^2} + \tfrac{\varepsilon_z}{\sigma^2}+\tfrac{1}{\sigma}) \cdot \tfrac{1}{\beta^t}\|\z^{t+1}-\z^t\|_{2}^2\nn\\
\overset{\step{1}}{\leq}&~ (\varepsilon_{\beta} + \tau C_h^2) (\tfrac{1}{\beta^{t}}-\tfrac{1}{\beta^{t+1}})+ (\tfrac{1}{2\tau} -\tfrac{1}{2})\beta^t \|\y^{t+1}-\y^t\|_{2}^2 + ( \tfrac{1}{2\sigma} + \tfrac{1}{4\sigma}+\tfrac{1}{\sigma}) \cdot \tfrac{1}{\beta^t}\|\z^{t+1}-\z^t\|_{2}^2 \nn\\
\overset{\step{2}}{\leq}&~ (\varepsilon_{\beta} + \tau C_h^2) (\tfrac{1}{\beta^{t}}-\tfrac{1}{\beta^{t+1}}) + (\tfrac{1}{2\tau} -\tfrac{1}{2})\beta^t \|\y^{t+1}-\y^t\|_{2}^2 + \tfrac{2}{\beta^t}\|\z^{t+1}-\z^t\|_{2}^2 \nn\\
\overset{\step{3}}{\leq}&~ (\varepsilon_{\beta} + \tau C_h^2) (\tfrac{1}{\beta^{t}}-\tfrac{1}{\beta^{t+1}}) + (\tfrac{1}{2\tau} -\tfrac{1}{2})\beta^t \|\y^{t+1}-\y^t\|_{2}^2 \nn\\
&~ + \tfrac{2}{\beta^t}\{  \tfrac{\sigma-1}{2-\sigma} (\|\z^{t}-\z^{t-1}\|_2^2 - \|\z^{t+1}-\z^t\|_2^2) + \dot{\sigma}(\beta^t)^2 \|\y^{t+1}-\y^{t}\|_2^2 + \ddot{\sigma} (\tfrac{\beta^0}{\beta^t}-\tfrac{\beta^0}{\beta^{t+1}}) \} \nn\\
\overset{\step{4}}{\leq}&~ \underbrace{ (\varepsilon_{\beta} + \tau C_h^2 + \tfrac{2}{\sigma}\ddot{\sigma} )}_{\triangleq c} (\tfrac{1}{\beta^{t}}-\tfrac{1}{\beta^{t+1}}) +    \underbrace{\tfrac{\sigma-1}{2-\sigma} ( \tfrac{2}{\beta^{t-1}}\|\z^{t}-\z^{t-1}\|_2^2 - \tfrac{2}{\beta^t}\|\z^{t+1}-\z^t\|_2^2)}_{= \DDD^t - \DDD^{t+1}} \nn\\
&~ + (\tfrac{1}{2\tau}  +  2\dot{\sigma} - \tfrac{1}{2} )\beta^t\|\y^{t+1}-\y^t\|_{2}^2  ,\nn
\end{align}
\noi where step \step{1} uses $\varepsilon_z=\tfrac{1}{4}$, $\xi\leq 1$, $\sigma\geq 1$; step \step{2} uses $\sigma\geq 1$; step \step{3} uses the upper bound for $\|\z^{t+1}-\z^t\|_2^2$ as shown in Lemma \ref{lemma:bounding:dual}(\bfit{b}); step \step{4} uses $\frac{1}{\beta^t}\leq \frac{1}{\beta^{t-1}}$. This further leads to:
\begin{align}
&~ \varepsilon_{\beta} \beta^t \BB_{t+1}^2 + \varepsilon_y \beta^t \YY_{t+1}^2 + \varepsilon_z \beta^t \ZZ_{t+1}^2 + \DDD^{t+1} - \DDD^t - \tfrac{c}{\beta^{t}}+\tfrac{c}{\beta^{t+1}}  \nn\\
&~ + L(\X^{t+1},\y^{t+1},\z^{t+1},\beta^{t+1}) - L(\X^{t},\y^t,\z^t,\beta^t) \nn\\
\overset{}{\leq}&~\beta^t\|\y^{t+1}-\y^t\|_{2}^2 \cdot ( \tfrac{1}{2\tau} + 2\dot{\sigma} - \tfrac{1}{2} )\nn\\
\overset{\step{1}}{=}&~\beta^t\|\y^{t+1}-\y^t\|_{2}^2 \cdot ( \tfrac{1}{2\tau} + 2 \tfrac{2\sigma^2}{(2-\sigma)^2}  \tfrac{1}{\tau^2} - \tfrac{1}{2} )\nn\\
\overset{\step{2}}{\leq}&~\beta^t\|\y^{t+1}-\y^t\|_{2}^2 \cdot ( \ts \tfrac{2-\sigma}{8}  + \tfrac{4\sigma}{16} - \tfrac{1}{2} )\nn\\
\overset{}{=}&~-\beta^t\|\y^{t+1}-\y^t\|_{2}^2 \cdot \underbrace{(\tfrac{2-\sigma}{8})}_{\triangleq \varepsilon_y},\nn
\end{align}
\noi where step \step{1} uses the definition of $\dot{\sigma}\triangleq \tfrac{2\sigma^2}{(2-\sigma)^2}  \tfrac{1}{\tau^2}$, as shown in Lemma \ref{lemma:bounding:dual}; step \step{2} uses $\tau\geq \tfrac{4}{2-\sigma}$.

%Using $\z^{t+1}-\z^t=\beta^t\sigma(\AA(\X^{t+1})-\y^{t+1})$, we have: $\DDD^t\triangleq \omega\dot{\sigma}\sigma^2 \beta^{t-1}\|\AA(\X^t)-\y^t\|_2^2=\tfrac{\omega\dot{\sigma}}{\beta^{t-1}}\|\z^{t}-\z^{t-1}\|_2^2$.

\end{proof}

\subsection{Proof of Lemma \ref{lemma:dec:x:proj}}
 \label{app:lemma:dec:x:proj}

\begin{proof}

We define $\mathcal{S}(\X,\y^t,\z^t,\beta^t) \triangleq f(\X)+ \la \z^t,\AA(\X) - \y^t\ra + \tfrac{\beta^t}{2}\|\AA(\X) - \y^t\|_{2}^2$.

We let $\G^t \in \nabla_{\X} \mathcal{S}(\X_{\cc}^t,\y^t,\z^t,\beta^t) - \partial g(\X^t)$.

We define $\PPP^{t} \triangleq \tfrac{1}{2}(\alpha+ \theta\alpha   ) \ell(\beta^t)\| \X^{t}  - \X^{t-1}\|_{\fro}^2$.

We define $\varepsilon_x'\triangleq  (\theta-1 - \alpha  -  \theta \alpha  ) - (1+\xi) (\alpha+ \theta\alpha)>0$, and $\varepsilon_x \triangleq \tfrac{1}{2} \varepsilon_x'\elldown>0$.

First, using the optimality condition of $\X^{t+1}\in \MM$, we have:
\beq
\la \X^{t+1} - \X^t, \G^t  \ra + \tfrac{\theta \ell(\beta^t) }{2} \|\X^{t+1} - \X_{\cc}^t\|_{\fro}^2\leq \la \X^t - \X^t, \G^t \ra + \tfrac{\theta \ell(\beta^t) }{2} \|\X^t - \X_{\cc}^t\|_{\fro}^2. \label{eq:add:1}
\eeq
Second, we have:
\begin{align}\label{eq:add:5}
&~L(\X^{t+1},\y^{t},\z^{t},\beta^t) - L(\X^{t},\y^{t},\z^{t},\beta^t) \nn\\
=&~ \ts  \mathcal{S}(\X^{t+1},\y^t,\z^t,\beta^t) - \mathcal{S}(\X^{t},\y^t,\z^t,\beta^t) + g(\X^t) - g(\X^{t+1})\nn\\
\overset{\step{1}}{\leq}&~ \ts  \tfrac{\ell(\beta^t) }{2}\|\X^{t+1} - \X^t\|_{\fro}^2+ \la \X^{t+1}-\X^t, \nabla_{\X} \mathcal{S}(\X^t,\y^t,\z^t,\beta^t)\ra    + \la \X^{t}-\X^{t+1}, \partial g(\X^t)\ra,
\end{align}
\noi where step \step{1} uses the $\ell(\beta^t)$-smoothness of $\mathcal{S}(\X,\y^t,\z^t,\beta^t)$ and convexity of $g(\X)$.

Third, we derive:
\beq \label{eq:add:4}
&&\la \X^{t+1}-\X^t, \nabla_{\X} \mathcal{S}(\X^t,\y^t,\z^t,\beta^t) - \nabla_{\X} \mathcal{S}( \X_{\cc}^t ,\y^t,\z^t,\beta^t)  \ra\nn\\
&\overset{\step{1}}{\leq}& \|\X^{t+1}-\X^t\|_{\fro}\cdot \| \nabla_{\X} \mathcal{S}(\X^t,\y^t,\z^t,\beta^t) - \nabla_{\X} \mathcal{S}( \X_{\cc}^t ,\y^t,\z^t,\beta^t)  \|_{\fro}\nn\\
&\overset{\step{2}}{\leq}& \|\X^{t+1}-\X^t\|_{\fro}\cdot \ell(\beta^t) \| \X^t - \X_{\cc}^t \|_{\fro}\nn\\
&\overset{\step{3}}{\leq}& \alpha \ell(\beta^t)  \|\X^{t+1}-\X^t\|_{\fro}\cdot  \| \X^t - \X^{t-1} \|_{\fro}\nn\\
&\overset{\step{4}}{\leq}& \tfrac{\alpha \ell(\beta^t)}{2}  \|\X^{t+1}-\X^t\|_{\fro}^2 + \tfrac{\alpha \ell(\beta^t)}{2} \| \X^t - \X^{t-1} \|^2_{\fro},
\eeq
\noi where step \step{1} uses the norm inequality; step \step{2} uses the $\ell(\beta^t)$-smoothness of $\mathcal{S}(\X,\y^t,\z^t,\beta^t)$; step \step{3} uses $\X_{\cc}^{t} = \X^{t}+\alpha (\X^{t}  - \X^{t-1})$; step \step{4} uses $a b \leq \frac{1}{2}a^2+\frac{1}{2}b^2$ for all $a\in\Rn$ and $b\in\Rn$.

Summing Inequalities (\ref{eq:add:1}),(\ref{eq:add:4}), and (\ref{eq:add:5}), we obtain:
\beq
&&L(\X^{t+1},\y^{t},\z^{t},\beta^t) - L(\X^{t},\y^{t},\z^{t},\beta^t) \nn\\
&\leq & \ts  \tfrac{\ell(\beta^t)}{2}  \{  (1+\alpha) \|\X^{t+1}-\X^t\|^2_{\fro} + \alpha \| \X^t - \X^{t-1} \|_{\fro} + \theta \|\X^t - \X_{\cc}^t\|_{\fro}^2  -  \theta \|\X^{t+1} - \X_{\cc}^t\|_{\fro}^2 \} \nn\\
&\overset{\step{1}}{=}& \ts \tfrac{\ell(\beta^t)}{2}  \{ (1+\alpha)  \|\X^{t+1}-\X^t\|^2_{\fro} + ( \alpha    + \theta  \alpha^2 ) \| \X^t - \X^{t-1} \|_{\fro}^2    -  \theta \|\X^{t+1} - \X^t - \alpha (\X^t-\X^{t-1})\|_{\fro}^2 \} \nn\\
&\overset{\step{2}}{\leq}& \ts \tfrac{\ell(\beta^t)}{2}  \{ (1+\alpha)  \|\X^{t+1}-\X^t\|^2_{\fro} +  (\alpha + \theta \alpha^2 )\| \X^t - \X^{t-1} \|_{\fro}^2  \nn\\
&&   +    \theta (\alpha-1) \|\X^{t+1} - \X^t\|_{\fro}^2 - \theta \alpha (\alpha-1)  \|\X^t-\X^{t-1}\|_{\fro}^2 \}    \nn\\
&\overset{}{=}& \ts \underbrace{\tfrac{1}{2}(\alpha+ \theta\alpha   ) \ell(\beta^t)\| \X^{t}  - \X^{t-1}\|_{\fro}^2}_{\triangleq \PPP^{t}}  + \tfrac{\ell(\beta^t)}{2} \cdot \|\X^{t+1}-\X^t\|^2_{\fro} \cdot \{ 1+\alpha  +  \theta \alpha - \theta \}    \nn\\
&\overset{}{=}& \ts \PPP^{t} -\PPP^{t+1} + \tfrac{1}{2} \cdot \|\X^{t+1}-\X^t\|^2_{\fro} \cdot \{ \ell(\beta^t) (1+\alpha  +  \theta \alpha - \theta) + \ell(\beta^{t+1}) (\alpha+ \theta\alpha) \}    \nn\\
&\overset{\step{3}}{\leq}& \ts \PPP^{t} - \PPP^{t+1} + \tfrac{1}{2} \cdot \|\X^{t+1}-\X^t\|^2_{\fro} \cdot \ell(\beta^t) \{ \underbrace{ (1+\alpha  +  \theta \alpha - \theta) + (1+\xi) (\alpha+ \theta\alpha)}_{\triangleq -\varepsilon_x' } \}    \nn\\
&\overset{\step{4}}{\leq}& \ts \PPP^{t}-\PPP^{t+1} - \tfrac{1}{2} \cdot \|\X^{t+1}-\X^t\|^2_{\fro} \cdot \varepsilon_x' \cdot \beta^t \elldown      \nn\\
&\overset{\step{5}}{=}& \ts \PPP^{t} - \PPP^{t+1} - \varepsilon_x \cdot \beta^t \|\X^{t+1}-\X^t\|^2_{\fro}  , \nn
\eeq
\noi where step \step{1} uses $\X_{\cc}^{t} = \X^{t}+\alpha (\X^{t}  - \X^{t-1})$; step \step{2} uses Lemma \ref{lemma:the:alpha} with $\mathbf{a}=\X^{t+1} - \X^t$, and $\mathbf{b}=\X^{t} - \X^{t-1}$; step \step{3} uses the fact that $\ell(\beta^{t+1}) \leq (1+\xi)\ell(\beta^t)$, which is implied by $\beta^{t+1}\leq (1+\xi)\beta^t$; step \step{4} uses Lemma \ref{lemma:smooth:LL} that $\beta^t\elldown\leq \ell(\beta^t) \leq \beta^t\ellup$; step \step{5} uses $\varepsilon_x \triangleq \tfrac{1}{2} \varepsilon_x' \elldown >0 $.

\end{proof}

\subsection{Proof of Lemma \ref{lemma:bound:Gamma:Pj}}
\label{app:lemma:bound:Gamma:Pj}

\begin{proof}

We define: $\Theta^t \triangleq L(\X^{t},\y^t,\z^t,\beta^t) + c/\beta^t + \PPP^{t} + \DDD^{t}$.

%We define $\tilde{e}_t\triangleq \|\y^{t}-\y^{t-1}\|^2 + \|\AA(\X^{t})-\y^{t}\|^2 + \|\X^{t}-\X^{t-1}\|_{\fro}^2$.

\textbf{Part (a)}. Using Lemma \ref{lemma:dec:non:x}, we have:
\beq\nn
&&L(\X^{t+1},\y^{t+1},\z^{t+1},\beta^{t+1}) - L(\X^{t},\y^t,\z^t,\beta^t) \nn\\
&\leq& c/\beta^t - c/\beta^{t+1} + \DDD^t -\DDD^{t+1} - \varepsilon_y \beta^t \YY_{t+1}^2 - \varepsilon_z \beta^t \ZZ_{t+1}^2 - \varepsilon_{\beta} \beta^t \BB_{t+1}^2 + \mathfrak{X}.
\eeq

Using Lemma \ref{lemma:dec:x:proj}, we have:
\beq\nn
\mathfrak{X} \leq \PPP^{t} - \PPP^{t+1} - \varepsilon_x \beta^t \XX_{t+1}^2.
\eeq
\noi Adding these two inequalities together and using the definition of $\Theta^t$, we have:
\beq
\Theta^t -\Theta^{t+1}   &\geq & \varepsilon_{\beta} \beta^t \BB_{t+1}^2 + \varepsilon_y\beta^t \YY_{t+1}^2         + \varepsilon_x \beta^t  \XX_{t+1}^2 + \varepsilon_z \beta^t \ZZ_{t+1}^2\nn\\
&\geq & \min(\varepsilon_y,\varepsilon_x,\varepsilon_z,\varepsilon_{\beta} )\cdot \beta^t \cdot (\BB_{t+1}^2+\XX_{t+1}^2+\YY_{t+1}^2+\ZZ_{t+1}^2). \nn
\eeq

\textbf{Part (b)}. Telescoping this inequality over $t$ from 1 to $T$, we have:
\begin{align}
\ts \sum_{t=1}^T \beta^t  (\BB_{t+1}^2+\XX_{t+1}^2+\YY_{t+1}^2+\ZZ_{t+1}^2) &~\leq  \ts \ts \tfrac{1}{\min(\varepsilon_y,\varepsilon_x,\varepsilon_z )} \cdot \sum_{t=1}^T (\Theta^t -\Theta^{t+1} )   \nn\\
&~= \ts \tfrac{1}{\min(\varepsilon_y,\varepsilon_x,\varepsilon_z )} \cdot (\Theta^1 -\Theta^{T+1} )   \nn\\
&~\overset{\step{1}}{\leq}  \ts \tfrac{1}{\min(\varepsilon_y,\varepsilon_x,\varepsilon_z )} \cdot (\Theta^1 - \underline{\rm{\Theta}} ) , \label{eq:OP:final:1}
\end{align}
\noi where step \step{1} uses $\Theta^t\geq \underline{\rm{\Theta}}$. Furthermore, we have:
\beq
&& \ts \sum_{t=1}^T \beta^t(\BB_{t+1}^2+\XX_{t+1}^2+\YY_{t+1}^2+\ZZ_{t+1}^2)\nn\\
&= &\ts \sum_{t=1}^T\tfrac{1}{\beta^t} (\beta^t)^2 (\BB_{t+1}^2+\XX_{t+1}^2+\YY_{t+1}^2+\ZZ_{t+1}^2)\nn\\
& \geq&  \ts \tfrac{1}{\beta^T} \sum_{t=1}^T (\beta^t)^2 (\BB_{t+1}^2+\XX_{t+1}^2+\YY_{t+1}^2+\ZZ_{t+1}^2)  \nn\\
& \overset{\step{1}}{\geq} &  \ts \tfrac{1}{4T\beta^T} (\sum_{t=1}^T \beta^t [(\BB_{t+1}+\XX_{t+1}+\YY_{t+1}+\ZZ_{t+1})])^2,\label{eq:OP:final:2}
\eeq
\noi where step \step{1} uses $\sum_{i=1}^n \x_i^2 \geq \tfrac{1}{n} (\sum_{i=1}^n |\x_i|)^2$ for all $\x\in\Rn^n$.

Combining Inequalities (\ref{eq:OP:final:1}) and (\ref{eq:OP:final:2}), we have: $$\ts \sum_{t=1}^T \beta^t [(\BB_{t+1}+\XX_{t+1}+\YY_{t+1}+\ZZ_{t+1})] \leq \ts  \{ \tfrac{\Theta^1 - \underline{\rm{\Theta}}}{\min(\varepsilon_y,\varepsilon_x,\varepsilon_z,\varepsilon_{\beta} )} \cdot 4 T \beta^T\}^{1/2} =  \OO(T^{(1+p)/2}).$$

\end{proof}

\subsection{Proof of Theorem \ref{theorem:OADMM:Pj}}
\label{app:theorem:OADMM:Pj}
\begin{proof}

We define $e^t \triangleq \BB_{t} + \ZZ_{t} + \YY_{t} + \XX_{t}$.

We define $\Crit(\X,\y,\z) \triangleq\|\AA(\X)- \y\| + \| \partial h(\y) - \z\| + \| \Proj_{\T_{\X} \MM} ( \nabla f(\X) - \partial g(\X)  + \AA \trans (\z) )\|_{\fro}$.

We define $\dot{\G} \triangleq   \nabla f(\X^{t}) - \partial g(\X^{t})  + \AA\trans(\z^{t})$.

We define $\ddot{\G} \triangleq   \nabla  f(\X_{\cc}^t)- \partial g(\X^{t})  +  \AA\trans (\z^t + \beta^t \AA(\X_{\cc}^t) - \y^t )    + \theta \ell(\beta^t)  (\X^{t+1} - \X_{\cc}^t)$.

We first derive the following inequalities:
\begin{align}\label{eq:fHGZ}
&~\|\ddot{\G} - \dot{\G}\|_{\fro}\nn\\
\overset{\step{1}}{=}&~\| \nabla f(\X^{t}) - \nabla  f(\X_{\cc}^t) -  \beta^t \AA\trans (\AA(\X_{\cc}^t) - \y^t) - \theta \ell(\beta^t)  (\X^{t+1} - \X_{\cc}^t) \|_{\fro}\nn\\
\overset{\step{2}}{\leq}&~L_f \| \X^{t} - \X_{\cc}^t\|_{\fro}  +  \beta^t \Aup \| \AA(\X_{\cc}^t) - \y^t\| + \theta \ell(\beta^t)  \| \X^{t+1} - \X_{\cc}^t \|_{\fro}\nn\\
\overset{\step{3}}{\leq}&~ L_f \| \X^{t} -\X^{t-1} \|_{\fro}  +  \beta^t \Aup \{ \| \AA(\X^{t}) - \y^t\|  + \Aup \| \X^{t} - \X^{t-1}\|_{\fro}\} \nn\\
&~ + \theta \ell(\beta^t) ( \| \X^{t+1} - \X^{t}\|_{\fro} + \| \X^{t} -  \X^{t-1} \|_{\fro}) \nn\\
\overset{\step{4}}{\leq}&~ (L_f + \beta^t \Aup^2 + \theta \ell(\beta^t) ) \| \X^{t} -\X^{t-1} \|_{\fro}  +  \beta^t \Aup \|\AA(\X^t) - \y^t\| + \theta \ell(\beta^t)  \| \X^{t+1} - \X^{t}\|_{\fro}\nn\\
\overset{\step{5}}{=}&~ \OO(\beta^{t-1} e^{t})+\OO(\beta^{t} e^{t+1}) ,
\end{align}
\noi where step \step{1} uses the definitions of $\{\ddot{\G},\dot{\G}\}$; step \step{2} uses the triangle inequality; step \step{3} uses the fact that $f(\X)$ is $L_f$-smooth, $\| \X^{t} - \X^{t}_{\cc}\|_{\fro} \leq \| \X^{t} - \X^{t-1}\|_{\fro}$, $\| \X^{t+1} - \X^{t}_{\cc}\|_{\fro} \leq \| \X^{t+1} - \X^{t}\|_{\fro}+ \| \X^{t} -  \X^{t-1} \|_{\fro}$, and $\|\AA(\X_{\cc}^t) - \y^t\| \leq \| \AA(\X^{t}) - \y^t\|_{\fro}  + \Aup \| \X^{t} - \X^{t-1}\|_{\fro}$, as shown in Lemma \ref{lemma:simple}.

We derive the following inequalities:
\beq \label{eq:C:point:XXX}
&& \| \Proj_{\T_{\X^{t}} \MM}  ( \dot{\G} ) \|_{\fro} \nn\\
&\overset{\step{1}}{=}& \| \Proj_{\T_{\X^{t}} \MM} ( \dot{\G} ) + \Proj_{\T_{\X^{t+1}} \MM} (\ddot{\G} ) \|_{\fro}\nn\\
&\overset{\step{2}}{\leq }&  2 \| \dot{\G} - \ddot{\G} \|_{\fro} + 2 \sqrt{r} \| \dot{\G} \|\cdot \|\X^{t+1}-\X^{t}\|_{\fro}\nn\\
&\overset{\step{3}}{\leq}& \OO(\beta^{t-1} e^{t})+\OO(\beta^{t} e^{t+1}) +  2 \sqrt{r} (C_f+C_g + \Aup \overline{\rm{z}} ) \|\X^{t+1}-\X^{t}\|_{\fro} \nn\\
&=& \OO(\beta^{t-1} e^{t})+\OO(\beta^{t} e^{t+1}), \nn
\eeq
\noi where step \step{1} uses the optimality of $\X^{t+1}$ that:
\beq
\zero  =   \Proj_{\T_{\X^{t+1}} \MM} (\ddot{\G} ); \nn
\eeq
\noi step \step{2} uses the result of Lemma \ref{lemma:Proj:Proj} by applying
\beq
\X=\X^t, ~\tilde{\X}=\X^{t+1},~\mathbf{P} = \dot{\G},~\text{and}~\tilde{\mathbf{P}} = \ddot{\G};\nn
\eeq
\noi step \step{3} uses Inequality (\ref{eq:fHGZ}), and the fact that $\|\dot{\G}\| = \| \nabla f(\X^{t}) - \partial g(\X^{t})  + \AA\trans(\z^{t}) )\| \leq \| \nabla f(\X^{t}) - \partial g(\X^{t})  + \AA\trans (\z^{t}) \|_{\fro} \leq C_f+C_g + \Aup \overline{\rm{z}}$.

Finally, we derive:
\beq \label{eq:CP:00:KK}
& & \ts\tfrac{1}{T}\sum_{t=1}^T \Crit(\X^{t},\breve{\y}^{t},\z^{t}) \nn\\
&\overset{\step{1}}{=}& \ts \tfrac{1}{T}\sum_{t=1}^T \{ \|\AA(\X^{t})- \breve{\y}^{t}\|  + \| \partial h(\breve{\y}^{t}) - \z^{t}\| + \| \Proj_{\T_{\X^{t}} \MM}  ( \dot{\G} )\|_{\fro}\} \nn\\
&\overset{\step{2}}{\leq}& \ts \tfrac{1}{T}\sum_{t=1}^T\{ \|\AA(\X^{t})- \y^{t}\| + \|\breve{\y}^{t} - \y^{t}\| +\|  (1- \tfrac{1}{\sigma})(\z^{t}-\z^{t-1} ) \| + \| \Proj_{\T_{\X^{t}} \MM}   ( \dot{\G} )  \|_{\fro}\} \nn\\
&\overset{\step{3}}{=}& \ts \tfrac{1}{T}\sum_{t=1}^T\{ \OO(\beta^{t-1} e^{t})+\OO(\beta^{t} e^{t+1})\}  + \tfrac{1}{T} \sum_{t=1}^T \|\breve{\y}^{t} - \y^{t} \|\nn\\
&\overset{\step{4}}{=}& \ts \tfrac{1}{T}\sum_{t=1}^T\{ \OO(\beta^t e^{t+1}) + \OO(\beta^{t-1} e^t) \} + \tfrac{1}{T}\OO(\sum_{t=1}^T \tfrac{1}{t^p}) \nn\\
&\overset{\step{5}}{=}& \ts \OO(T^{(p-1)/2}) + \OO(T^{1-p-1}) \nn\\
&\overset{\step{6}}{=}& \OO(T^{-1/3}),\nn
\eeq
\noi where step \step{1} uses the definition of $\Crit(\X,\y,\z)$; step \step{2} uses $\z^{t+1} - \partial h(\breve{\y}^{t+1})  \ni  (1- \tfrac{1}{\sigma})(\z^{t+1}-\z^{t} )$, as shown in Lemma \ref{lemma:bounding:dual}; step \step{3} uses $\|\z^t-\z^{t-1}\| = \|\sigma \beta^{t-1} (\AA(\X^{t})-\y^t)\| \leq 2\beta^t \|\AA(\X^{t})-\y^t\|=\OO(\beta^{t-1} e^t)$; step \step{4} uses Lemma \ref{lemma:smoothing:problem:prox}(\bfit{c}) that $\|\bar{\y}^t- \breve{\y}^t\|\leq \mu^t C_h = \OO(\tfrac{1}{t^p})$; step \step{5} uses Lemma \ref{lemma:lp:bounds:2} that $\sum_{t=1}^T \tfrac{1}{t^p}\leq \OO(T^{1-p})$, and Lemma \ref{lemma:bound:Gamma:Pj}(\bfit{b}) that $\tfrac{1}{T}\sum_{t=1}^T \beta^t e^{t+1} \leq \OO(T^{(p-1)/2})$; step \step{6} uses the choice $p=1/3$ and Lemma \ref{lemma:bound:Gamma:Pj}(\bfit{b}).

\end{proof}

\subsection{Proof of Lemma \ref{lemma:dec:x}}
\label{app:lemma:dec:x}

\begin{proof}

%Noticing $\alpha=0$, we have $\X_{\cc}^t=\X^t$ for all $t$.

We define $\mathcal{S}(\X,\y^t,\z^t,\beta^t) \triangleq f(\X)+ \la \z^t,\AA(\X) - \y^t\ra + \tfrac{\beta^t}{2}\|\AA(\X) - \y^t\|_{2}^2$.

We let $\G^t \in \nabla_{\X} \mathcal{S}(\X^t,\y^t,\z^t,\beta^t) - \partial g(\X^t)$. We define $\eta^t \triangleq \tfrac{b^t \gamma^{j}}{\beta^t}  \in (0,\infty)$.

\textbf{Part (a)}. Initially, we show that $\|\G^t\|_{\fro}$ is always bounded for $t$ with $\X \in\MM$. We have:
\beq
\| \G^t \|_{\fro} & =& \| \nabla f(\X^t)- \partial g(\X^t)  + \AA\trans [\z^t + \beta^t (\AA(\X^t) - \y^t ) ]  \|_{\fro}  \nn\\
&\overset{\step{1}}{=}& \| \nabla f(\X^t) - \partial g(\X^t) +  \AA\trans [\z^t +  \tfrac{\beta^t}{\sigma \beta^{t-1}} ( \z^t - \z^{t-1}) ]  \|_{\fro}  \nn\\
&\overset{\step{2}}{\leq}& \| \nabla f(\X^t) \|_{\fro} + \| \partial g(\X^t)\|_{\fro}  + \Aup \cdot \{ \|\z^t\| +  \tfrac{\beta^t}{\sigma \beta^{t-1}} (\| \z^t\| + \| \z^{t-1}\|) \}    \nn\\
&\overset{\step{3}}{\leq}& C_f + C_g + \Aup \cdot ( \overline{\rm{z}} +  2(1+\xi) \overline{\rm{z}}  )  \triangleq  \overline{g}, \nn
\eeq
\noi where step \step{1} uses $\z^{t+1} = \z^t + \sigma \beta^t (\AA (\X^{t+1}) - \y^{t+1})$; step \step{2} uses the triangle inequality; step \step{3} uses $\| \nabla f(\X^t) \|_{\fro}\leq C_f$, $\| \nabla g(\X^t) \|_{\fro}\leq C_g$, $\|\nabla \AA(\X^t)\|_{\fro}\leq \|\nabla \AA(\X^t)\|\leq \Aup$, $\|\z^t\|\leq \overline{\rm{z}}$, $\tfrac{1}{\sigma}\leq 1$, $\beta^t\leq \beta^{t-1} (1+\xi)$; step \step{4} uses $\xi\leq1$.

We derive the following inequalities:
\begin{align}\label{eq:condition:satisfy}
&~ L(\X^{t+1},\y^t,\z^t,\beta^t) - L(\X^t,\y^t,\z^t,\beta^t) = \dot{\mathcal{L}}(\X^{t+1}) - \dot{\mathcal{L}}(\X^{t})\nn\\
\overset{\step{1}}{=}&~ \{ \mathcal{S}^t(\X^{t+1},\y^t,\z^t,\beta^t)   - g(\X^{t+1})\} -\{ \mathcal{S}^t(\X^{t},\y^t,\z^t,\beta^t) - g(\X^{t}) \} \nn\\
\overset{\step{2}}{\leq}&~    \tfrac{1}{2} \ell(\beta^t) \|\X^{t+1} - \X^t\|_{\fro}^2+\la \G^t, \X^{t+1} - \X^t  \ra\nn\\
\overset{\step{3}}{=}&~ \tfrac{1}{2} \ell(\beta^t) \|\Retr_{\X^{t}}(-\eta^t \GGG^t_{\rho}) - \X^t\|_{\fro}^2+ \la \G^t,\Retr_{\X^{t}}(-\eta^t \GGG^t_{\rho}) - \X^t + \eta^t \GGG^t_{\rho} \ra - \eta^t \la \G^t, \GGG^t_{\rho} \ra \nn\\
\overset{\step{4}}{\leq}&~  \tfrac{1}{2}\ell(\beta^t) \|\Retr_{\X^{t}}(-\eta^t \GGG^t_{\rho}) - \X^t\|_{\fro}^2+  \overline{g} \|\Retr_{\X^{t}}(-\eta^t \GGG^t_{\rho}) - \X^t + \eta^t \GGG^t_{\rho}\|_{\fro} - \tfrac{\eta^t}{\max(1,2\rho)} \|\GGG^t_{\rho}\|_{\fro}^2  \nn\\
\overset{\step{5}}{\leq}&~\tfrac{1}{2}\ell(\beta^t)  \dot{k} \| \eta^t \GGG^t_{\rho} \|_{\fro}^2+  \tfrac{1}{2}\overline{g} \ddot{k} \| \eta^t \GGG^t_{\rho}\|_{\fro}^2 - \tfrac{\eta^t}{\max(1,2\rho)} \|\GGG^t_{\rho}\|_{\fro}^2   \nn\\
\overset{\step{6}}{=}&~   \eta^t  \|\GGG^t_{\rho}\|_{\fro}^2 \cdot \{  \tfrac{1}{2} \ell(\beta^t) \dot{k} \tfrac{b^t \gamma^{j} }{\beta^t} + \tfrac{1}{2}\overline{g} \ddot{k} \tfrac{b^t \gamma^{j}}{\beta^t}   - \tfrac{1}{\max(1,2\rho)}  \}  \nn\\
\overset{\step{7}}{\leq}&~   \eta^t  \| \GGG^t_{\rho}\|_{\fro}^2 \cdot \{  (  \tfrac{\overline{b}}{2} \dot{k} \ellup +  \tfrac{\overline{b}}{2 \beta^0} \ddot{k} \overline{g} ) \gamma^{j}  - \tfrac{1}{\max(1,2\rho)}  \}  \nn\\
\overset{\step{8}}{\leq}&~   \eta^t  \| \GGG^t_{\rho}\|_{\fro}^2 \cdot \{  - \delta  \} ,
\end{align}
\noi where step \step{1} uses the definitions of $L(\X,\y,\z,\beta)$; step \step{2} uses the fact that the function $g(\X)$ is convex and the function $\mathcal{S}(\X,\y^t,\z^t,\beta^t)$ is $\ell(\beta^t) $-smooth \textit{w.r.t.} $\X$; step \step{3} uses $\X^{t+1} = \Retr_{\X^{t}}( -\eta^t \GGG^t_{\rho})$; step \step{4} uses the Cauchy-Schwarz Inequality, $\|\G^t\|_{\fro}\leq \overline{g}$, and Lemma \ref{lemma:GD:bound}(\bfit{a}) that $\la \G^t,\GGG^t_{\rho}\ra \geq \tfrac{1}{\max(1,2\rho)}\|\GGG^t_{\rho}\|_{\fro}^2$; step \step{5} uses Lemma \ref{lemma:M12} with $\Deltas \triangleq -\eta^t \GGG^t_{\rho}$ given that $\X^t \in \MM$ and $\Deltas \in \T_{\X^t}\MM$; step \step{6} uses $\eta^t \triangleq \tfrac{b^t \gamma^{j}}{\beta^t}$; step \step{7} uses $\ell(\beta^t)\leq \beta^t \ellup$, $\beta^0\leq \beta^t$, and $b^t\leq \overline{b}$; step \step{8} uses the fact that $\gamma^{j}$ is sufficiently small such that:
\beq
\gamma^{j} \leq \frac{2 (\tfrac{1}{\max(1,2\rho)} -\delta)}{ \ellup    \dot{k} \overline{b} + \overline{g} \ddot{k} \overline{b}/\beta^0   } \triangleq \overline{\gamma}.\label{eq:gamma:gamma}
\eeq

Given Inequality (\ref{eq:condition:satisfy}) coincides with the condition of the line search procedure, we complete the proof.

\textbf{Part (b)}. We derive the following inequalities:
\beq
&& L(\X^{t+1},\y^t,\z^t,\beta^t) - L(\X^t,\y^t,\z^t,\beta^t) \nn\\
&\overset{\step{1}}{\leq}&  - \|\GGG^t_{\rho}\|_{\fro}^2 \delta {\eta}^t  \nn\\
&\overset{\step{2}}{\leq}& - \|\GGG^t_{1/2}\|^2_{\fro} \delta {\eta}^t  \cdot \min(1,2 \rho)^2  \nn\\
&\overset{\step{3}}{=}& - \tfrac{1}{\beta^t} \|\GGG^t_{1/2}\|^2_{\fro} \cdot \delta b^t \gamma^{j-1} \gamma \cdot \min(1,2 \rho)^2 \nn\\
&\overset{\step{4}}{\leq}& -  \tfrac{1}{\beta^t} \|\GGG^t_{1/2}\|^2_{\fro}\cdot \underbrace{\delta \underline{b}  \overline{\gamma} \gamma  \cdot \min(1,2 \rho)^2}_{\triangleq \varepsilon_x }, \nn
\eeq
\noi where step \step{1} uses Inequality (\ref{eq:condition:satisfy}); step \step{2} uses Lemma \ref{lemma:GD:bound}(\bfit{b}) that $\|\GGG_{\rho}\|_{\fro} \geq \min(1,2 \rho) \|\GGG_{1/2}\|_{\fro}$; step \step{3} uses the definition ${\eta}^t\triangleq \tfrac{b^t \gamma^{j}}{\beta^t} $; step \step{4} uses $b^t\geq \underline{b}$, and the following inequality:
\beq
\gamma^{j-1} \geq \overline{\gamma}\geq \gamma^{j},  \nn
\eeq
\noi which can be implied by the stopping criteria of the line search procedure.

\end{proof}

 \subsection{Proof of Lemma \ref{lemma:bound:Gamma}}
 \label{app:lemma:bound:Gamma}

\begin{proof}

We define: $\Theta^t \triangleq L(\X^{t},\y^t,\z^t,\beta^t) + c/\beta^t +  \DDD^{t}+ 0 \times \PPP^{t}$,

%We define $\tilde{e}_t\triangleq \|\y^{t}-\y^{t-1}\|^2 + \|\AA(\X^{t})-\y^{t}\|^2 + \|\tfrac{1}{\beta^t}\GGG^t_{1/2}\|_{\fro}^2$.

\textbf{Part (a)}. Using Lemma \ref{lemma:dec:non:x}, we have:
\beq\nn
&&L(\X^{t+1},\y^{t+1},\z^{t+1},\beta^{t+1}) - L(\X^{t},\y^t,\z^t,\beta^t) \nn\\
&\leq& c/\beta^t  - c/\beta^{t+1} + \mathfrak{X} + \DDD^t -\DDD^{t+1} - \varepsilon_y\beta^t\YY_{t+1}^2 - \varepsilon_z \beta^t \ZZ_{t+1}^2 .
\eeq

Using Lemma \ref{lemma:dec:x}, we have:
\beq
\mathfrak{X} \leq 0 \times \PPP^{t} - 0 \times \PPP^{t+1} - \varepsilon_x \beta^t \XX_{t+1}^2.\nn
\eeq

\noi Adding these two inequalities together and using the definition of $\Theta^t$, we have:
\beq
\Theta^t -\Theta^{t+1}   &\geq & \varepsilon_{\beta} \beta^t \BB_{t+1}^2 + \varepsilon_y\beta^t \YY_{t+1}^2         + \varepsilon_x \beta^t  \XX_{t+1}^2 + \varepsilon_z \beta^t \ZZ_{t+1}^2\nn\\
&\geq & \min(\varepsilon_y,\varepsilon_x,\varepsilon_z,\varepsilon_{\beta} )\cdot \beta^t \cdot (\BB_{t+1}^2+\XX_{t+1}^2+\YY_{t+1}^2+\ZZ_{t+1}^2). \nn
\eeq

\textbf{Part (b)}. Using the same strategy as in deriving Lemma \ref{lemma:bound:Gamma:Pj}(\bfit{b}), we finish the proof.

\end{proof}
 \subsection{Proof of Theorem \ref{theorem:OADMM:R}}
 \label{app:theorem:OADMM:R}

\begin{proof}

We define $e^t \triangleq \BB_{t} + \ZZ_{t} + \YY_{t} + \XX_{t}$.

We define $\Crit(\X,\y,\z) \triangleq\|\AA(\X)- \y\| + \| \partial h(\y) - \z\| + \| \Proj_{\T_{\X} \MM} ( \nabla f(\X) - \partial g(\X)  + \AA\trans(\z) )\|_{\fro}$.

We define $\dot{\G}\triangleq \nabla f(\X^t) -  \partial g(\X^t)+ \AA\trans(\z^t)$, and $\ddot{\G} \triangleq \beta^t \AA\trans (\AA(\X^t) - \y^t)$.

We let $\G = \G^t \in \partial_{\X} L(\X^t,\y^t,\z^t,\beta^t)$.

\noi First, we obtain:
\beq \label{eq:D:12:333}
\GGG^t_{1/2} &\overset{\step{1}}{=}&  \G- \tfrac{1}{2}\X^t \G \trans\X^t - \tfrac{1}{2} \X^t[\X^t]\trans \G   \nn\\
 &\overset{\step{2}}{=}& (\dot{\G}-\tfrac{1}{2} \X^t \dot{\G} \trans\X^t - \tfrac{1}{2} \X^t[\X^t]\trans \dot{\G} ) + (\ddot{\G}-\tfrac{1}{2} \X^t \ddot{\G} \trans\X^t - \tfrac{1}{2} \X^t[\X^t]\trans \ddot{\G} )\nn\\
&\overset{\step{3}}{=}& \Proj_{\T_{\X^{t}} \MM} ( \dot{\G} ) + \Proj_{\T_{\X^{t}} \MM} ( \ddot{\G} )\nn
\eeq
\noi where step \step{1} uses the definition $\GGG^t_{\rho} \triangleq  \G- \rho\X^t \G \trans\X^t - (1-\rho)\X^t[\X^t]\trans \G$, as shown in Algorithm \ref{alg:main}; step \step{2} uses $\G \in \dot{\G} + \ddot{\G}$; step \step{3} uses the fact that $\Proj_{\T_{\X}\MM}(\Deltas) = \Deltas - \tfrac{1}{2}\X (\Deltas \trans \X+\X\trans \Deltas)$ for all $\Deltas\in\Rn^{n\times r}$ \cite{absil2008optimization}. This leads to:
\beq \label{eq:Proj:G1}
 \| \Proj_{\T_{\X^{t}} \MM} ( \dot{\G} )\|_{\fro}& =& \|\GGG^t_{1/2}- \Proj_{\T_{\X^{t}} \MM} ( \ddot{\G} ) \|_{\fro}\nn\\
&\overset{\step{1}}{\leq}& \|\GGG^t_{1/2}\|_{\fro} + \| \Proj_{\T_{\X^{t}} \MM} ( \ddot{\G} ) \|_{\fro}\nn\\
&\overset{\step{2}}{\leq}& \|\GGG^t_{1/2}\|_{\fro} + \|\ddot{\G}  \|_{\fro}\nn\\
&\overset{}{\leq}&\|\GGG^t_{1/2}\|_{\fro} +  \beta^t \Aup \|  \AA(\X^t) - \y^t \| \nn\\
&\overset{}{\leq}& \beta^{t} e^{t+1}  + \OO(\beta^{t-1}e^t), \nn
\eeq
\noi where step \step{1} uses the triangle inequality; step \step{2} uses Lemma \ref{lemma:bound:PXX:XX} that $\|\Proj_{\T_{\X}\MM}(\Deltas)\|_{\fro}\leq \|\Deltas\|_{\fro}$ for all $\Deltas \in \Rn^{n\times r}$.

Finally, we derive:
\beq
& & \ts\tfrac{1}{T}\sum_{t=1}^T \Crit(\X^{t},\breve{\y}^{t},\z^{t}) \nn\\
&\overset{\step{1}}{=}& \ts \tfrac{1}{T}\sum_{t=1}^T \{ \|\AA(\X^{t})- \breve{\y}^{t}\|  + \| \partial h(\breve{\y}^{t}) - \z^{t}\| + \| \Proj_{\T_{\X^{t}} \MM}  ( \dot{\G} )\|_{\fro}\} \nn\\
&\overset{\step{2}}{\leq}& \ts \tfrac{1}{T}\sum_{t=1}^T\{ \|\AA(\X^{t})- \y^{t}\| + \|\breve{\y}^{t} - \y^{t}\| +\|  (1- \tfrac{1}{\sigma})(\z^{t}-\z^{t-1} ) \| + \| \Proj_{\T_{\X^{t}} \MM}   ( \dot{\G} )  \|_{\fro}\} \nn\\
&\overset{\step{3}}{=}& \ts \tfrac{1}{T}\sum_{t=1}^T\{ \OO(\beta^t e^{t+1}) + \OO(\beta^{t-1} e^t) \} + \tfrac{1}{T} \sum_{t=1}^T \|\breve{\y}^{t} - \y^{t}\|\nn\\
&\overset{\step{4}}{=}& \ts \tfrac{1}{T}\sum_{t=1}^T\{ \OO(\beta^t e^{t+1}) + \OO(\beta^{t-1} e^t) \} + \tfrac{1}{T}\OO(\sum_{t=1}^T \tfrac{1}{t^p}) \nn\\
&\overset{\step{5}}{=}& \ts \OO(T^{(p-1)/2}) + \OO(T^{1-p-1}) \nn\\
&\overset{\step{6}}{=}& \OO(T^{-1/3}),\nn
\eeq
\noi where step \step{1} uses the definition of $\Crit(\X,\y,\z)$; step \step{2} uses $\z^{t+1} - \partial h(\breve{\y}^{t+1})  \ni  (1- \tfrac{1}{\sigma})(\z^{t+1}-\z^{t} )$, as shown in Lemma \ref{lemma:bounding:dual}; step \step{3} uses $\|\z^t-\z^{t-1}\| = \|\sigma \beta^{t-1} (\AA(\X^{t})-\y^t)\| \leq 2\beta^t \|\AA(\X^{t})-\y^t\|=\OO(\beta^{t-1} e^t)$; step \step{4} uses Lemma \ref{lemma:smoothing:problem:prox}(\bfit{c}) that $\|\bar{\y}^t- \breve{\y}^t\|\leq \mu^t C_h = \OO(\tfrac{1}{t^p})$; step \step{5} uses Lemma \ref{lemma:lp:bounds:2} that $\sum_{t=1}^T \tfrac{1}{t^p}\leq \OO(T^{1-p})$, and Lemma \ref{lemma:bound:Gamma:Pj}(\bfit{b}) that $\tfrac{1}{T}\sum_{t=1}^T \beta^t e^{t+1} \leq \OO(T^{(p-1)/2})$; step \step{6} uses the choice $p=1/3$ and Lemma \ref{lemma:bound:Gamma:Pj}(\bfit{b}).

\end{proof}

\section{Proofs for Section \ref{sect:strong:KL}}\label{app:sect:KL}

\subsection{Proof of Lemma \ref{lemma:subgrad:bound:P}} \label{app:lemma:subgrad:bound:P}

We begin by presenting the following four useful lemmas.

\begin{lemma} \label{lemma:Grad:Theta:IIII:P}

For both {\sf OADMM-EP} and {\sf OADMM-RR}, we have:
\beq
(\d_{\X},\d_{\y},\d_{\z},\d_{\beta}) \in \partial \Theta(\X^{t},\y^{t},\z^{t},\beta^{t}),
\eeq
\noi where $\d_{\X} \triangleq \mathbb{A}^t + \beta^t \AA \trans ( \AA(\X^t) - \y^t )$, $\d_{\y}\triangleq \nabla h_{[\tau/\beta^t]} (\y^t) - \z^t +  \beta^t(\y^t-\AA(\X^t))$, $\d_{\z} \triangleq \AA(\X^t)-\y^t$, $\d_{\beta} \triangleq \tfrac{1}{2}\|\AA(\X^t)-\y^t\|_2^2  + \partial_{\beta} (h_{[\tau/\beta^t]} (\y^t))   - \tfrac{c}{(\beta^t)^2}$, $\mathbb{A}^t \triangleq \partial \iota_{\MM}  (\X^t) + \nabla f(\X^t)  + \AA\trans (\z^t) $.

\begin{proof}

We define the Lyapunov function as: $\Theta(\X,\y,\z,\beta) \triangleq  L(\X,\y,\z,\beta) + c/\beta$. Using this definition, we can promptly derive the conclusion of the lemma.

%These results are based on basic deductions.
\end{proof}
\end{lemma}

\begin{lemma} \label{lemma:bound:term:1}
For {\sf OADMM-EP}, we define $\{\d_{\X},\d_{\y},\d_{\z},\d_{\beta}\}$ as in Lemma \ref{lemma:Grad:Theta:IIII:P}. There exists a constant $K$ such that:
\begin{align}
\ts  \tfrac{1}{\beta^t} \{\|\d_{\X}\|_{\fro} +  \|\d_{\y}\| +  \|\d_{\z} \|+  |\d_{\beta} | \}\leq \ts  K \{\mathcal{B}_t  +  \mathcal{X}^{t} +  \mathcal{Y}^{t} +  \mathcal{Z}^{t} \}.
\end{align}
Here, $\mathcal{B}_t\triangleq\sqrt{ (\tfrac{1}{\beta^{t-1}} - \tfrac{1}{\beta^{t}})\tfrac{1}{\beta^{t-1}}}$, $\mathcal{X}_t\triangleq\|\X^t - \X^{t-1}\|_{\fro}$, $\mathcal{Y}_t\triangleq\|\y^t - \y^{t-1}\|$, and $\mathcal{Z}_t\triangleq\|\AA(\X^{t}) - \y^t\|$.

\begin{proof}

First, we obtain:
\begin{align}\label{eq:beta:A:A:A}
&~ \tfrac{1}{\beta^t} \| \mathbb{A}^t\|_{\fro}  = \|  \partial I_{\MM}  (\X^t) + \nabla f(\X^t) +  \AA\trans(\z^t)\|_{\fro} \nn\\
\overset{\step{1}}{=}&~ \tfrac{1}{\beta^t} \|  \nabla f(\X^t) - \nabla f(\X^{t-1}) - \theta \ell(\beta^{t-1})   (\X^{t} - \X^{t-1}) \nn\\
&~   + \AA\trans (\z^t-\z^{t-1}) - \beta^{t-1} \AA\trans  (\AA(\X^{t-1}) - \y^{t-1} ))  \|_{\fro} \nn\\
\overset{\step{2}}{=}&~ \tfrac{1}{\beta^t} \|  \nabla f(\X^t) - \nabla f(\X^{t-1}) - \theta \ell(\beta^{t-1})   (\X^{t} - \X^{t-1})\|_{\fro} \nn\\
&~ +\tfrac{1}{\beta^t} \| \sigma \beta^{t-1} \AA\trans ( \AA(\X^t) - \y^t ) - \beta^{t-1} \AA\trans  (\AA(\X^{t-1}) - \y^{t-1} ))  \|_{\fro} \nn\\
\overset{\step{3}}{\leq}&~  \tfrac{1}{\beta^t} ( L_f + \theta \ell(\beta^{t-1})  ) \|  \X^t- \X^{t-1} \|_{\fro}  \nn\\
&~ + \beta^{t-1} \| (\sigma-1)  \AA \trans ( \AA(\X^t) - \y^t  )   +    \AA \trans ( \AA(\X^{t-1}-\X^{t}) )  + \AA \trans (\y^{t-1}-\y^{t})  \| \nn\\
\overset{}{=}&~  \OO(\|\X^t - \X^{t-1}\|_{\fro})   +   \OO(\|\y^t - \y^{t-1}\|)   +  \OO(\|\AA(\X^{t}) - \y^t\|),
\end{align}
\noi where step \step{1} uses the optimality of $\X^{t+1}$ for {\sf OADMM-EP} that:
\begin{align}
&~\ts \partial I_{\MM}(\X^{t+1}) \nn\\
\ni&~ \ts - \theta \ell(\beta^t)  (\X^{t+1} - \X^t)     - \nabla_{\X} \mathcal{S}(\X^t,\y^t,\z^t,\beta^t)  \nn \\
=&~- \theta \ell(\beta^t)  (\X^{t+1} - \X^t) - \nabla f(\X^t) - \AA \trans [ \z^t + \beta^t ( \AA(\X^t) - \y^t   ) ];
\end{align}
\noi step \step{2} uses the triangle inequality, the $L_f$-Lipschitz continuity of $\nabla f(\X)$ for all $\X$; the $L_g$-Lipschitz continuity of $\nabla g(\X)$, and the upper bound of $\|\AA(\X_{\cc}^t) - \y^t\|$ as shown in Lemma \ref{lemma:simple}(\bfit{c}); step \step{3} uses the upper bound of $\|  \X^t- \X^{t-1}_{\cc} \|_{\fro}$, and $\z^{t} -\z^{t-1} = \sigma \beta^{t-1} (\AA (\X^{t}) - \y^{t}) $.

\noi \textbf{Part (a)}. We bound the term $\tfrac{1}{\beta^t}\|\d_{\X}\|_{\fro}$. We have:
\begin{align}
\tfrac{1}{\beta^t}\|\d_{\X}\|_{\fro} \overset{\step{1}}{=}&~ \ts \tfrac{1}{\beta^t}\| \mathbb{A}^t + \beta^t \AA\trans( \AA(\X^t) - \y^t)   \|_{\fro}\nn\\
\overset{\step{3}}{\leq}&~ \ts  \OO(\|\X^t - \X^{t-1}\|_{\fro})   +\OO(\|\y^t - \y^{t-1}\|)   +  \OO(\|\AA(\X^{t}) - \y^t\|), \nn
\end{align}
\noi where step \step{1} uses the definition of $\d_{\X}$ in Lemma \ref{lemma:Grad:Theta:IIII:P}; step \step{2} uses the triangle inequality, $\beta^{t-1}\leq \beta^t$, and $\ell(\beta^t)\leq \beta^t \ellup$; step \step{3} uses Inequality (\ref{eq:beta:A:A:A}).

\noi \textbf{Part (b)}. We bound the term $\tfrac{1}{\beta^t}\|\d_{\y}\|$. We have:
\beq
\ts \tfrac{1}{\beta^t}\|\d_\y\|&\overset{\step{1}}{=}& \ts \tfrac{1}{\beta^t} \|\nabla h_{\tau/\beta^t} (\y^t) - \z^t +  \beta^t (\y^t-\AA(\X^t)) \|\nn\\
&\overset{\step{2}}{=}&\ts \tfrac{1}{\beta^t} \|\nabla h_{\tau/\beta^t} (\y^t) - \nabla h_{\tau/\beta^{t-1}} (\y^t) +  \beta^t (\y^t-\AA(\X^t)) \|\nn\\
&\overset{\step{3}}{\leq}&\ts \OO(\|\AA(\X^t) - \y^t\|) + C_h\tfrac{1}{\beta^t} (\tfrac{\beta^t}{\beta^{t-1}}-1)\nn\\
&\overset{}{=}& \ts \OO(\|\AA(\X^t) - \y^t\|) + C_h \sqrt{ \tfrac{1}{\beta^{t-1}} - \tfrac{1}{\beta^{t}} }  \sqrt{\tfrac{1}{\beta^{t-1}}}  \cdot \sqrt{ \tfrac{ \beta^{t-1} }{ \beta^{t-1} } - \tfrac{ \beta^{t-1} }{ \beta^t } }\nn\\
&\overset{\step{4}}{\leq}& \ts \OO(\|\AA(\X^t) - \y^t\|) + C_h \underbrace{\sqrt{ \tfrac{1}{\beta^{t-1}} - \tfrac{1}{\beta^{t}} }  \sqrt{\tfrac{1}{\beta^{t-1}}}}_{\triangleq \mathcal{B}_t}  \cdot \sqrt{1-0} ,
\eeq
\noi where step \step{1} uses the definition of $\d_{\y}$ in Lemma \ref{lemma:Grad:Theta:IIII:P}; step \step{2} uses the fact that $\z^{t} - \tfrac{1}{\sigma} (\z^{t} - \z^{t+1}) =  \nabla h_{\mu^{t}} (\y^{t+1})$, as shown in Lemma \ref{lemma:bounding:dual}; step \step{3} uses $\tfrac{1}{\beta^t}(\z^{t+1} - \z^t) =  \sigma (\AA (\X^{t+1}) - \y^{t+1})$, and $\beta^{t-1}=\OO(\beta^t)$; step \step{4} uses $\mathcal{B}_t\triangleq\sqrt{ (\tfrac{1}{\beta^{t-1}} - \tfrac{1}{\beta^{t}})\tfrac{1}{\beta^{t-1}}}$.

\noi \textbf{Part (d)}. We bound the term $\tfrac{1}{\beta^t}|\d_{\beta}|$. We have:
\begin{align}
\tfrac{1}{\beta^t}|\d_{\beta}| \overset{\step{1}}{=}&~ \tfrac{1}{\beta^t}\cdot \{\tfrac{1}{2}\|\AA(\X^t)-\y^t\|_2^2  + \partial_{\beta} (h_{[\tau/\beta^t]} (\y^t))   - \tfrac{c}{(\beta^t)^2} \} \nn\\
\overset{}{=}&~ \ts \tfrac{1}{\beta^t}\{\OO(\|\AA(\X^t)-\y^t\|)  + \partial_{[\tau/\beta^t]} (h_{[\tau/\beta^t]} (\y^t))\cdot \tfrac{\tau}{(\beta^t)^2}   - \tfrac{c}{(\beta^t)^2} \ts \} \nn\\
\overset{\step{2}}{\leq}&~ \ts \OO(\|\AA(\X^t)-\y^t\|)  +  (\tau C_h^2 +c)\tfrac{1}{(\beta^t)^3} \ts \nn\\
\overset{\step{3}}{\leq}&~ \ts \OO(\|\AA(\X^t)-\y^t\|)  +  (\tau C_h^2 +c) \tfrac{2}{ (\xi\beta^0)^2}\cdot\BB^t,  \ts \nn
\end{align}
\noi where step \step{1} uses the definition of $\d_{\beta}$ in Lemma \ref{lemma:Grad:Theta:IIII:P}; step \step{2} uses Lemma \ref{lemma:mu:continous} that the function $h_{\tau/\beta^t}(\y)$ is $C_h^2$-Lipschitz continuous \textit{w.r.t.} $\mu^t\triangleq \tau/\beta^t$; step \step{3} uses Lemma \ref{lemma:bound:4}.

\noi \textbf{Part (e)}. We bound the term $\tfrac{1}{\beta^t}\|\d_{\z}\|_{\fro}$. We have: $\tfrac{1}{\beta^t}\|\d_{\z}\|\leq \tfrac{1}{\beta^0} \| \AA(\X^t)-\y^t \|$.

\noi \textbf{Part (f)}. Combining the upper bounds for the terms $\{\tfrac{1}{\beta^t}\|\d_{\X}\|_{\fro},\tfrac{1}{\beta^t}\|\d_{\y}\|,\tfrac{1}{\beta^t}|\d_{\beta}|,\tfrac{1}{\beta^t}\|\d_{\z}\|\}$, we finish the proof of this lemma.
\end{proof}

\end{lemma}

\begin{lemma} \label{lemma:bound:term:2}
For {\sf OADMM-RR}, we define $\{\d_{\X},\d_{\y},\d_{\z},\d_{\beta}\}$ as in Lemma \ref{lemma:Grad:Theta:IIII:P}. There exists a constant $K$ such that :
\begin{align}
\ts  \tfrac{1}{\beta^t} \{ \|\d_{\X}\|_{\fro}  +\|\d_{\y}\|  +\|\d_{\z}\| + |\d_{\beta}| \}\leq  \ts  K \{\mathcal{B}_t  +  \mathcal{X}^{t} +  \mathcal{Y}^{t} +  \mathcal{Z}^{t} \},\nn
\end{align}
Here, $\mathcal{B}_t\triangleq\sqrt{ (\tfrac{1}{\beta^t} - \tfrac{1}{\beta^{t+1}})\tfrac{1}{\beta^t}}$, $\mathcal{X}_t\triangleq \|\tfrac{1}{\beta^t}\GGG_{1/2}\|_{\fro}$, $\mathcal{Y}_t\triangleq\|\y^t - \y^{t-1}\|$, and $\mathcal{Z}_t\triangleq\|\AA(\X^{t}) - \y^t\|$.

\begin{proof}
%We define $\mathbb{B}^t \triangleq \partial I_{\MM}  (\X^t) + \nabla f(\X^t) + [\nabla \AA(\X^t)]\trans \z^t  - \nabla g(\X^t)$.

We define $\mathbf{G}^t\triangleq \nabla f(\X^t) - \nabla g(\X^t)  + A\trans (\z^t + \beta^t (\AA(\X^t) - \y^t))$.

We define $\dot{\mathcal{L}}(\X)\triangleq L(\X,\y^t,\z^t,\beta^t)$, we have: $\nabla \dot{\mathcal{L}}(\X^t) = \mathbf{G}^t$.

%First, given $\X^t \in \MM$, we obtain:

\noi \textbf{Part (a)}. We bound the term $\tfrac{1}{\beta^t}\|\d_{\X}\|_{\fro}$. We have:
\beq \label{eq:bound:RR:G:partial:I}
\tfrac{1}{\beta^t}\|\d_{\X}\|_{\fro} &\overset{}{=}& \tfrac{1}{\beta^t}\| \partial I_{\MM}  (\X^t) + \nabla \dot{\mathcal{L}}(\X^t) \|_{\fro} \nn\\ &\overset{\step{1}}{\leq}& \ts \tfrac{1}{\beta^t} \| \nabla \dot{\mathcal{L}}(\X^t) - \X^t [\nabla \dot{\mathcal{L}}(\X^t)]\trans \X^t\|_{\fro}\nn\\
&\overset{\step{2}}{=}&\ts \tfrac{1}{\beta^t} \| \mathbf{G}^t - \X^t [\mathbf{G}^t]\trans \X^t\|_{\fro} = \tfrac{1}{\beta^t}\|\GGG^t_{1}\|_{\fro} \nn\\
&\overset{\step{3}}{\leq}& \ts \tfrac{1}{\beta^t} \max(1,1/\rho) \cdot \| \GGG_{1/2}\|_{\fro} = \OO(\mathcal{X}_t),
\eeq
\noi where step \step{1} uses Lemma \ref{lemma:subgradient:R:bound}; step \step{2} uses the definitions of $\{\mathbf{G}^t,\mathbf{D}_{\rho}^t\}$ as in Algorithm \ref{alg:main}; step \step{3} uses $\|\GGG_{1}\|_{\fro} \leq \max(1,1/\rho) \|\GGG_{\rho}\|_{\fro}$, as shown in Lemma \ref{lemma:GD:bound}(\bfit{b}).

\noi \textbf{Part (b)}. We bound the terms $\tfrac{1}{\beta^t}\|\d_{\y}\|$, $\tfrac{1}{\beta^t}|\d_{\beta}|$, and $\tfrac{1}{\beta^t}\|\d_{\z}\|_{\fro}$. Considering that the same strategies for updating $\{\y^t,\beta^t,\z^t\}$ are employed, their bounds in {\sf OADMM-RR} are identical to those in {\sf OADMM-ER}.

\noi \textbf{Part (c)}. Combining the upper bounds for the terms $\{\tfrac{1}{\beta^t}\|\d_{\X}\|_{\fro},\tfrac{1}{\beta^t}\|\d_{\y}\|,\tfrac{1}{\beta^t}|\d_{\beta}|,\tfrac{1}{\beta^t}\|\d_{\z}\|\}$, we finish the proof of this lemma.

\end{proof}

\end{lemma}

Now, we proceed to prove the main result of this lemma.

\begin{lemma} \label{lemma:subg:bound}

(\rm{Subgradient Bounds}) For both {\sf OADMM-EP} and {\sf OADMM-RR}, there exists a constant $K>0$ such that: $$\dist(\zero,\partial \Theta(\ww^t)) \leq \beta^t K e^{t}.$$
\noi Here, $\dist^2(\zero,\partial \Theta(\ww))\triangleq  \dist^2(\zero,\partial_{\beta}\Theta(\ww))+\dist^2(\zero,\partial_{\X}\Theta(\ww)) + \dist^2(\zero,\partial_{\y}\Theta(\ww))+ \dist^2(\zero,\partial_{\z}\Theta(\ww))$.

\begin{proof}

For both {\sf OADMM-EP} and {\sf OADMM-RR}, we have:
\begin{align}
\dist(\zero,\partial \Theta(\ww^t)) = &~ \sqrt{ \|\d_{\X}\|_{\fro}^2  + \|\d_{\y}\|^2 + \|\d_{\z}\|^2 + |\d_{\beta}|^2 }\nn\\
\overset{\step{1}}{\leq}&~  \|\d_{\X}\|_{\fro}  + \|\d_{\y}\| + |\d_{\beta}|+ \|\d_{\z}\| \nn \\
\overset{\step{2}}{\leq}&~  K \beta^t \{ \XX^t  +   \YY^{t}+   \BB^{t}+   \ZZ^{t} \}    , \nn
\end{align}
\noi where step \step{1} uses $\sqrt{a+b}\leq \sqrt{a}+\sqrt{b}$ for all $a,b\geq 0$; step \step{2} uses Lemmas \ref{lemma:bound:term:1} and \ref{lemma:bound:term:2}.

\end{proof}

\end{lemma}

\subsection{Proof of Theorem \ref{theorem:finite:length:P}} \label{app:theorem:finite:length:P}

\begin{proof}

We define $K'\triangleq 4K/\min(\varepsilon_x,\varepsilon_y,\varepsilon_z,\varepsilon_{\beta})$.

Firstly, using Assumption \ref{assumption:KL}, we have:
\beq \label{eq:KL:ours}
\varphi'( \Theta(\ww^t) - \Theta(\ww^{\infty})  )\cdot \dist(\zero,\partial \Theta(\ww^t) )\geq1 .
\eeq
\noi Secondly, given the desingularization function $\varphi(\cdot)$ is concave, for any $a,b \in \Rn$, we have: $\varphi(b) + (a-b) \varphi'(a)\leq \varphi(a)$. Applying the inequality above with $a= \Theta(\ww^t) - \Theta(\ww^{\infty})$ and $b = \Theta(\ww^{t+1}) - \Theta(\ww^{\infty})$, we have:
\beq \label{eq:concave:bound}
&& \ts ( \Theta(\ww^t) - \Theta(\ww^{t+1}) ) \cdot \varphi'( \Theta(\ww^t) - \Theta(\ww^{\infty}))   \nn\\
&\leq& \ts \underbrace{\varphi( \Theta(\ww^{t}) -  \Theta(\ww^{\infty}))}_{\triangleq\varphi^t} - \underbrace{\varphi( \Theta(\ww^{t+1}) -  \Theta(\ww^{\infty})  )}_{\triangleq\varphi^{t+1}}.
\eeq

\noi Third, we define $\mathcal{X}_t\triangleq\|\X^t - \X^{t-1}\|_{\fro}$, and derive the following inequalities for {\sf OADMM-EP}:
\beq \label{eq:beta:XYZ:phi}
&& \min(\varepsilon_z,\varepsilon_y, \varepsilon_x,\varepsilon_{\beta}) \beta^t \{ \BB^2_{t+1}  + \ZZ^2_{t+1} + \YY^2_{t+1} +\XX^2_{t+1} \} \nn\\
&\overset{\step{1}}{\leq}&    \Theta^t - \Theta^{t+1}      =  \Theta(\ww^t) - \Theta(\ww^{t+1})    \nn\\
&\overset{\step{2}}{\leq}&   \ts (\varphi^t - \varphi^{t+1}) \cdot \tfrac{ 1 }{ \varphi'( \Theta(\ww^t) - \Theta(\ww^{\infty}) ) )  }    \nn \\
&\overset{\step{3}}{\leq}&  \ts  (\varphi^t - \varphi^{t+1}) \cdot
\dist(\zero,\partial \Theta(\ww^t)  )\nn\\
&\overset{\step{4}}{\leq}&  \ts   (\varphi^t - \varphi^{t+1}) \cdot K \beta^{t}  (\BB_{t}  + \ZZ_{t} + \YY_{t} +\XX_{t}) ,
\eeq
\noi where step \step{1} uses Lemma \ref{lemma:bound:Gamma:Pj}; step \step{3} uses Inequality (\ref{eq:concave:bound}); step \step{4} uses Inequality (\ref{eq:KL:ours}); step \step{5} uses Lemma \ref{lemma:subgrad:bound:P}. We further derive the following inequalities:
\beq\label{eq:finite:length:resursive:1}
&& \ts   (\BB_{t+1}  + \ZZ_{t+1} + \YY_{t+1} +\XX_{t+1})^2  \nn\\
&\overset{\step{1}}{\leq} & \ts 4 \cdot \{ \BB_{t+1}^2 + \ZZ_{t+1}^2+ \YY_{t+1}^2+\XX_{t+1}^2\}\nn\\
&\overset{\step{2}}{\leq} & \ts \{ 4 K/\min(\varepsilon_z,\varepsilon_y, \varepsilon_x,\varepsilon_{\beta}) \} \cdot  (\varphi^t - \varphi^{t+1}) \cdot (\BB_{t}  + \ZZ_{t} + \YY_{t} +\XX_{t}),
\eeq
\noi where step \step{1} uses the norm inequality that $(a+b+c+d)^2 \leq 4(a^2+b^2+c^2+d^2)$ for any $a,b,c,d\in\Rn$; step \step{2} uses Inequality (\ref{eq:beta:XYZ:phi}).

\noi Fourth, we define $\XX_{t}\triangleq \|\tfrac{1}{\beta}\GGG^{t}_{1/2}\|_{\fro}$, and derive the following inequalities for {\sf OADMM-RR}:
\beq \label{eq:beta:XYZ:phi:2}
&& \min(\varepsilon_z,\varepsilon_y, \varepsilon_x, \varepsilon_{\beta}) \beta^t \{  \BB^2_{t+1}  + \ZZ^2_{t+1} + \YY^2_{t+1} +\XX^2_{t+1}   \} \nn\\
&\overset{\step{1}}{\leq}&    \Theta^t - \Theta^{t+1}      =  \Theta(\ww^t) - \Theta(\ww^{t+1})    \nn\\
&\overset{\step{2}}{\leq}&   \ts (\varphi^t - \varphi^{t+1}) \cdot \tfrac{ 1 }{ \varphi'( \Theta(\ww^t) - \Theta(\ww^{\infty}) ) )  }    \nn \\
&\overset{\step{3}}{\leq}&  \ts  (\varphi^t - \varphi^{t+1}) \cdot
\dist(\zero,\partial \Theta(\ww^t)  )\nn\\
&\overset{\step{4}}{\leq}&  \ts   (\varphi^t - \varphi^{t+1}) \cdot K \beta^{t} (\varphi^t - \varphi^{t+1}) \cdot (\BB_{t}  + \ZZ_{t} + \YY_{t} +\XX_{t}) ,
\eeq
\noi where step \step{1} uses Lemma \ref{lemma:bound:Gamma}; step \step{2} uses Inequality (\ref{eq:concave:bound}); step \step{3} uses Inequality (\ref{eq:KL:ours}); step \step{4} uses Lemma \ref{lemma:subgrad:bound:P}. Furthermore, using the similar strategies as in deriving Inequality (\ref{eq:finite:length:resursive:1}), we have:
\beq \label{eq:finite:length:resursive:2}
&& \ts   (\BB_{t+1}  + \ZZ_{t+1} + \YY_{t+1} +\XX_{t+1})^2  \nn\\
&\leq& \ts \{ 4 K/\min(\varepsilon_z,\varepsilon_y, \varepsilon_x,\varepsilon_{\beta}) \} \cdot  (\varphi^t - \varphi^{t+1}) \cdot (\BB_{t}  + \ZZ_{t} + \YY_{t} +\XX_{t}).
\eeq

\textbf{Part (a)}. We define $e^{t}\triangleq \BB_t + \XX_t + \YY_t+ \ZZ_t$. Given Inequalities (\ref{eq:finite:length:resursive:1}) and (\ref{eq:finite:length:resursive:2}), we establish the following unified inequality applicable to both {\sf OADMM-EP} and {\sf OADMM-RR}:
\beq\label{eq:finite:length:resursive:together}
(e^{t+1})^2 \leq e^t \cdot \underbrace{ \{ 4 K/\min(\varepsilon_z,\varepsilon_y, \varepsilon_x, \varepsilon_{\beta}) \} }_{\triangleq K'} \cdot (\varphi^t - \varphi^{t+1}).
\eeq
\noi \textbf{Part (b)}. Considering Inequality (\ref{eq:finite:length:resursive:together}) and applying Lemma \ref{lemma:any:three:seq} with $p^t\triangleq K'\varphi^t$, we have:
\beq
\ts \forall t,\, \sum_{i=t}^{\infty} e^{i+1}  \leq e^{t} + 2 K' \varphi^t.\nn
\eeq
\noi Letting $t=1$, we have: $\sum_{i=1}^{\infty} e^{i+1} \leq  \ts e^1 + 2 K' \varphi^1$.

%\noi where step \step{1} uses $\varphi^1 = \varphi( F^1 -  F^{\star})$.
\end{proof}

\subsection{Proof of Lemma \ref{lemma:rate:pre}} \label{app:lemma:rate:pre}

\begin{proof}

 We define $d^t \triangleq \sum_{i=t}^{\infty} e^{i+1}$.

\textbf{Part (a-i)}. For {\sf OADMM-EP}, we have for all $t\geq 1$: $\|\X^{t} - \X^\infty\|_{\fro} \overset{\step{1}}{\leq}  \ts \sum_{i=t}^{\infty} \|\X^i-\X^{i+1}\|_{\fro} \leq \ts  \sum_{i=t}^{\infty}\{ \sqrt{ (\tfrac{1}{\beta^t}-\tfrac{1}{\beta^{t+1}})\tfrac{1}{\beta^t}  } + \|\X^{i+1} - \X^i\|_{\fro} + \|\y^{i+1} - \y^i\| + \|\AA(\X^{i+1}) - \y^{i+1}\| \} =\ts \sum_{i=t}^{\infty} e^{i+1}  \triangleq d^t$, where step \step{1} use the triangle inequality.

\textbf{Part (a-ii)}. For {\sf OADMM-RR}, we have: $\|\X^{t+1} - \X^t\|_{\fro} \overset{\step{1}}{=} \| \Retr_{\X^{t}}(-\eta^t \GGG^t_{\rho}) - \X^t\|_{\fro} \overset{\step{2}}{\leq}\dot{k} \|\eta^t \GGG^t_{\rho} \|_{\fro}\overset{\step{3}}{\leq} \dot{k} \eta^t \max(2\rho,1) \| \GGG^t_{1/2} \|_{\fro}    \overset{\step{4}}{=}\dot{k} \max(2\rho,1)  \tfrac{b^t \gamma^j}{\beta^t}  \| \GGG^t_{1/2} \|_{\fro} \overset{\step{5}}{\leq}\dot{k} \max(2\rho,1) \overline{b} \overline{\gamma} \cdot \|\tfrac{1}{\beta^t} \GGG^t_{1/2} \|_{\fro}  = \OO( \| \tfrac{1}{\beta^t}\GGG^t_{1/2} \|_{\fro} )$, where step \step{1} uses the update rule of $\X^{t+1}$; step \step{2} uses Lemma \ref{lemma:M12}; step \step{3} uses Lemma \ref{lemma:GD:bound}(\bfit{c}); step \step{4} uses the definition of $\eta^t\triangleq \tfrac{b^t \gamma^{j}}{\beta^t}$; step \step{5} uses $b^t\leq \overline{b}$, and the fact that $\gamma^j \leq \overline{\gamma}$. Furthermore, we derive for all $t\geq 1$: $\|\X^{t} - \X^\infty\|_{\fro} \leq  \ts \sum_{i=t}^{\infty} \|\X^i-\X^{i+1}\|_{\fro} \leq  \ts \OO(\sum_{i=t}^{\infty}  \|\tfrac{1 }{\beta^t} \GGG_{1/2}^i\|_{\fro}) \overset{}{\leq}  \OO(\sum_{i=t}^{\infty} e^{i+1})=\OO(d^t)$.

\textbf{Part (b)}. We define $\varphi^t \triangleq \varphi(s^t)$, where $s^t\triangleq \Theta(\ww^{t}) -  \Theta(\ww^{\infty})$. We derive:
\beq
\ts d^{t} &\overset{}{\triangleq} & \ts    \sum_{i=t}^{\infty} e^{i+1}\nn\\
&\overset{\step{1}}{\leq} & \ts e^{t} + 2  K' \varphi^t \nn\\
&\overset{\step{2}}{=} & \ts  e^{t}  + 2  K' \tilde{c}\cdot [ (s^t)^{\tilde{\sigma}} ]^{\frac{1-\tilde{\sigma}}{\tilde{\sigma}}} \nn\\
&\overset{\step{3}}{=} & \ts  e^{t}  + 2  K' \tilde{c}\cdot [ \tilde{c} (1-\tilde{\sigma})\cdot \frac{1}{\varphi'(s^t)} ]^{\frac{1-\tilde{\sigma}}{\tilde{\sigma}}} \nn\\
&\overset{\step{4}}{\leq} & \ts  e^{t} + 2  K'  \tilde{c}\cdot [    \tilde{c} (1-\tilde{\sigma})\cdot  \dist(\zero,\partial \Theta(\ww^{t})) ]^{\frac{1-\tilde{\sigma}}{\tilde{\sigma}}} \nn\\
&\overset{\step{5}}{\leq} & \ts e^{t} + 2  K' \tilde{c}\cdot [   \tilde{c} (1-\tilde{\sigma})\cdot  \beta^t K e^{t} ]^{\frac{1-\tilde{\sigma}}{\tilde{\sigma}}} \nn\\
&\overset{\step{6}}{=} & \ts d^{t-1} - d^{t} +  2  K'  \tilde{c}\cdot \{    \tilde{c} (1-\tilde{\sigma})\cdot  \beta^t K( d^{t-1} - d^{t}) \}^{\frac{1-\tilde{\sigma}}{\tilde{\sigma}}} \nn\\
&\overset{}{=} & \ts d^{t-1} - d^{t} + \underbrace{ 2  K' \tilde{c}\cdot [    \tilde{c} (1-\tilde{\sigma}) K ]^{\frac{1-\tilde{\sigma}}{\tilde{\sigma}}}    }_{\triangleq \nu}\cdot [ \beta^t (d^{t-1} - d^{t}) ]^{\frac{1-\tilde{\sigma}}{\tilde{\sigma}}},\nn
\eeq
\noi where step \step{1} uses $\sum_{i=t}^{\infty} e^{i+1}  \leq e^{t} + e^{t-1} + 4 K' \varphi^t$, as shown in Theorem \ref{theorem:finite:length:P}(\bfit{b}); step \step{2} uses the definitions that $\varphi^t \triangleq \varphi(s^t)$, $s^t\triangleq \Theta(\ww^{t}) -  \Theta(\ww^{\infty})$, and $\varphi(s) = \tilde{c} s^{1-\tilde{\sigma}}$; step \step{3} uses $\varphi'(s) = \tilde{c}(1-\tilde{\sigma}) \cdot [s]^{-\tilde{\sigma}}$, leading to $[s^t]^{\tilde{\sigma}}=\tilde{c}(1-\tilde{\sigma})\cdot \tfrac{1}{\varphi'(s^t)}$; step \step{4} uses Assumption \ref{assumption:KL} that $1 \leq \dist(\zero,\partial \Theta(\ww^t) ) \cdot \varphi'( s^t)$; step \step{5} uses $\dist(\zero,\partial \Theta(\ww^t) )\leq \beta^t K e^t$ for both {\sf OADMM-EP} and {\sf OADMM-RR}, as shown in Lemma \ref{lemma:subgrad:bound:P}; step \step{6} uses the fact that $e^t = d^{t-1} - d^{t}$.

\end{proof}

\subsection{Proof of Theorem \ref{theorem:KL:rate:Exponent:P}} \label{app:theorem:KL:rate:Exponent:P}

\begin{proof}

Using Lemma \ref{lemma:rate:pre}(\bfit{b}), we have:
\beq\label{eq:three:case:conv:ddd}
d^t \leq d^{t-1} - d^{t} + \nu\cdot  [\beta^t (d^{t-1} - d^{t}) ]^{\frac{1-\tilde{\sigma}}{\tilde{\sigma}}}.
\eeq
By the definition of $e^t \triangleq d^{t-1} - d^{t} \triangleq \BB_t + \XX_t + \YY_t+ \ZZ_t$, there exists a universal constant $\overline{\rm{e}}>0$ such that $e^t \leq \overline{\rm{e}}$.

We consider three cases for Inequality (\ref{eq:three:case:conv:ddd}).

\textbf{Part (a)}. $\tilde{\sigma} \in (0,\frac{1}{2}]$. We let $u = \frac{ 1-\tilde{\sigma}}{\tilde{\sigma}} \in [1,\infty)$, and $\zeta= (1-p)u = \tfrac{2}{3}\cdot \frac{1-\tilde{\sigma}}{\tilde{\sigma}}\in [\tfrac{2}{3},\infty]$.

From Inequality (\ref{eq:three:case:conv:ddd}), we obtain:
\beq
d^t& \leq& d^{t-1} - d^{t} + \nu \cdot  [\beta^t (d^{t-1} - d^{t}) ]^{u} \nn\\
&\overset{\step{1}}{ \leq} & d^{t-1} - d^{t} + \nu' \cdot  t^{pu} \cdot [d^{t-1} - d^{t}]^{u} \nn\\
&\overset{\step{2}}{ \leq} &  \OO(t^{-\zeta}), \nn
\eeq
\noi where step \step{1} uses $\beta^t \leq \nu' t^p$, where $\nu'$ is some certain constant; step \step{2} uses Lemma \ref{lemma:sigma:case:less:than:1:4} with $c=\nu'$ and $u>1$.

%\textbf{Part (a)}. $\tilde{\sigma} \in (0,\frac{1}{4}]$. We let $u = \frac{ 1-\tilde{\sigma}}{\tilde{\sigma}} \in [3,\infty)$, and $\zeta= \frac{pu}{u-1} = \tfrac{1}{3}\cdot \frac{1-\tilde{\sigma}}{1-2 \tilde{\sigma}}\in (\frac{1}{3},\tfrac{1}{2}]$.
%
%From Inequality (\ref{eq:three:case:conv:ddd}), we obtain:
%\beq
%d^t& \leq& d^{t-1} - d^{t} + \nu \cdot  [\beta^t (d^{t-1} - d^{t}) ]^{u} \nn\\
%&\overset{\step{1}}{ \leq} & d^{t-1} - d^{t} + \nu \cdot  [\beta^t]^{u} \cdot [d^{t-1}  ]^{u} \nn\\
%&\overset{\step{2}}{ \leq} & d^{t-1} - d^{t} + \nu' \cdot  t^{pu} \cdot [d^{t-1}  ]^{u} \nn\\
%&\overset{\step{3}}{ \leq} &  \OO(t^{-\zeta}), \nn
%\eeq
%\noi where step \step{1} uses $u\geq 3$ and $-d^t\leq 0$; step \step{2} uses $\beta^t \leq \nu' t^p$, where $\nu'$ is some certain constant; step \step{3} uses Lemma \ref{lemma:sigma:case:less:than:1:4} with $c=\nu'$ and $u>1$.

\textbf{Part (b)}. $\tilde{\sigma} \in (\frac{1}{2},1)$. We let $u = \frac{1-\tilde{\sigma}}{\tilde{\sigma}} \in (0,1)$, and $\zeta = \tfrac{1-p}{1/u-1} = \frac{2}{3} \cdot \frac{1-\tilde{\sigma}}{2\tilde{\sigma}-1} \in (0,\infty)$.

We have from Inequality (\ref{eq:three:case:conv:ddd}):
\beq \label{eq:three:case:conv:ddd:BBBBB}
d^t &\leq&  \nu [\beta^t]^{u} \cdot (d^{t-1} - d^{t}) ^{u} + d^{t-1} - d^{t}    \nn\\
&\overset{\step{1}}{\leq} &  \nu [\beta^t]^{u} \cdot (d^{t-1} - d^{t}) ^{u}+  (d^{t-1} - d^{t} )^{u} \cdot \overline{\rm{e}}^{ 1 - u }  \nn\\
&\overset{\step{2}}{\leq} &  \nu [\beta^t]^{u} \cdot (d^{t-1} - d^{t}) ^{u} +  (d^{t-1} - d^{t} )^{u} \cdot \overline{\rm{e}}^{ 1 - u }  \cdot (\tfrac{\beta^t}{\beta^0})^u \nn\\
&\overset{}{=} &  [\nu+ (\beta^0)^{-u}] \cdot [\beta^t]^{u} \cdot (d^{t-1} - d^{t}) ^{u} \nn\\
&\overset{\step{3}}{=} &  \nu' t^{pu} \cdot (d^{t-1} - d^{t}) ^{u}\nn\\
&\overset{\step{4}}{ \leq } &  \OO(t^{-\zeta}),\nn
\eeq
\noi where step \step{1} uses the fact that $\max_{x \in (0,\overline{\rm{e}}]} \frac{x}{x^{u}} \leq \overline{\rm{e}}^{1-u}$ if $u\in(0,1)$; step \step{2} uses $\beta^0\leq\beta^t$ and $u \in (0,1)$; step \step{3} uses the fact that $\beta^t\leq \nu'' t^p$ for some certain constant $\nu''$, and the choice that $\nu'= [\nu+(\beta^0)^{-u}]\nu''$; step \step{4} uses Lemma \ref{lemma:sigma:case:between:1:2:to:1} with $c=\nu'$.

\textbf{Part (c)}. $\tilde{\sigma} \in (\tfrac{1}{4},\tfrac{1}{2}]$. We let $u = \tfrac{ 1-\tilde{\sigma}}{\tilde{\sigma}} \in [1,3)$, and $\zeta = 1-pu = \frac{4 \tilde{\sigma}-1}{3\tilde{\sigma}} \in (0,\tfrac{2}{3}]$. 

%Using the definition of $e^t$, there exists a universal constant $\overline{\rm{e}}>0$ such that $e^t \leq \overline{\rm{e}}$.

We have from Inequality (\ref{eq:three:case:conv:ddd}):
\beq \label{eq:three:case:conv:ddd:AAA}
d^{t} &\leq&   d^{t-1} - d^{t} +   \nu \cdot [\beta^t]^u\cdot [d^{t-1} - d^{t}] ^{u} \cdot   \nn\\
&\overset{\step{1}}{ \leq} & (d^{t-1} - d^{t} ) \cdot  (1+ \overline{\rm{e}}^{u-1} \nu \cdot [\beta^t]^u)   \nn\\
&\overset{\step{2}}{ \leq } &  (d^{t-1} - d^{t} ) \cdot \nu' t^{p u} \nn\\
&\overset{}{ \leq} & d^{t-1}  \cdot \tfrac{ \nu' t^{p u} }{ \nu' t^{p u} + 1}\nn\\
&\overset{\step{3}}{\leq} &  \mathcal{O} ({1}/{\exp(t^\zeta)}), \nn
\eeq
\noi where step \step{1} uses the fact that $\tfrac{x^u}{x} \leq \overline{\rm{e}}^{u-1}$ for all $u \geq 1$, and $x = d^{t-1} - d^{t} = e^t \leq \overline{\rm{e}}$; step \step{2} uses $ (1+ \overline{\rm{e}}^{u-1} \nu \cdot [\beta^t]^u) \leq \nu' t^{p u}$ with $\nu'$ being some certain constant, which is implied by $\beta^t = \OO(t^p)$; step \step{3} uses Lemma \ref{lemma:sigma:case:between:1:4:to:1/2} with $c=\nu'$ and $q=pu$.

\textbf{Part (d)}. Finally, using the fact $\|\X^{T} - \X^\infty\|_{\fro} \leq \OO(d^T)$ as shown in Lemma \ref{app:lemma:rate:pre}(\bfit{b}), we finish the proof of this theorem.

\end{proof}

\section{Additional Experiments Details and Results} \label{app:exp}

%We provide further details of the experiments in Section \ref{app:sect:data:set} and additional results in Section \ref{app:add:exp:detail}.

\noi $\blacktriangleright$ \textbf{Datasets}. In our experiments, we utilize several datasets comprising both randomly generated and publicly available real-world data. These datasets are structured as data matrices $\mathbf{D}\in\Rn^{\dot{m}\times \dot{d}}$. They are denoted as follows: `mnist-$\dot{m}$-$\dot{d}$', `TDT2-$\dot{m}$-$d'$', `sector-$m'$-$d'$', and `randn-$\dot{m}$-$\dot{d}$', where ${\mathrm{randn}(m,n)}$ generates a standard Gaussian random matrix of size $m\times n$. The construction of $\mathbf{D}\in\Rn^{\dot{m}\times \dot{d}}$ involves randomly selecting $\dot{m}$ examples and $\dot{d}$ dimensions from the original real-world dataset, sourced from \url{http://www.cad.zju.edu.cn/home/dengcai/Data/TextData.html} and \url{https://www.csie.ntu.edu.tw/~cjlin/libsvm/}. Subsequently, we normalize each column of $\mathbf{D}$ to possess a unit norm and center the data by subtracting the mean, denoted as $\mathbf{D} \Leftarrow \mathbf{D}-\one \one\trans \mathbf{D}$.

\noi $\blacktriangleright$ \textbf{Additional experiment Results}.
We present additional experimental results in Figures \ref{fig:L0PCA:rho:10}, \ref{fig:L0PCA:rho:100}, and \ref{fig:L0PCA:rho:1000}. The figures demonstrate that the proposed {\sf OADMM} method generally outperforms the other methods, with {\sf OADMM-EP} surpassing {\sf OADMM-RR}. These results reinforce our previous conclusions.

\begin{figure}[!t]

\centering
\begin{subfigure}{.24\textwidth}\centering\includegraphics[width=1.12\linewidth]{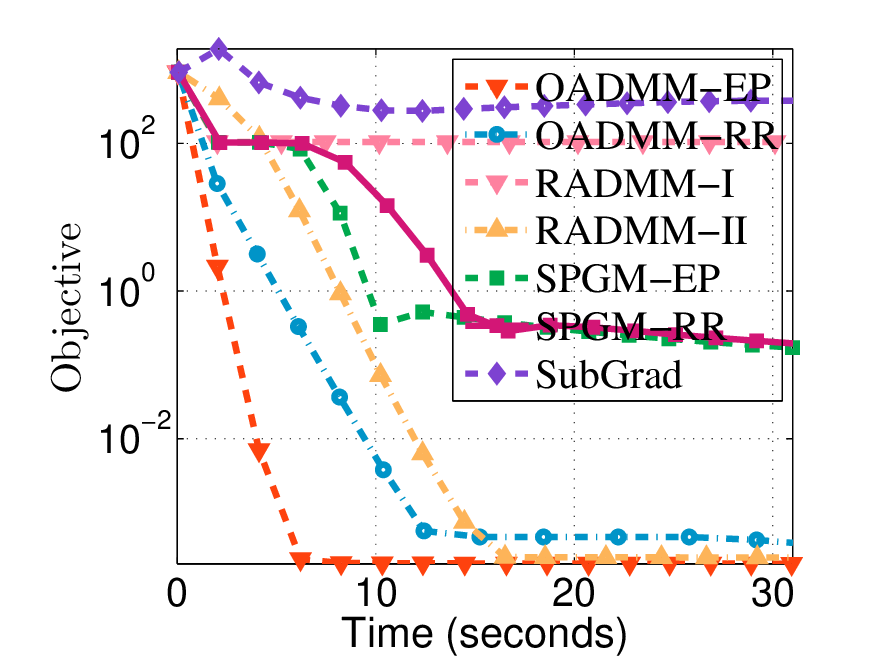}\caption{\scriptsize mnist-1500-500}\label{fig:sub1}\end{subfigure}
\begin{subfigure}{.24\textwidth}\centering\includegraphics[width=1.12\linewidth]{fig//demo_L0PCA10_12.eps}\caption{\scriptsize mnist-2500-500}\label{fig:sub2}\end{subfigure}
\begin{subfigure}{.24\textwidth}\centering\includegraphics[width=1.12\linewidth]{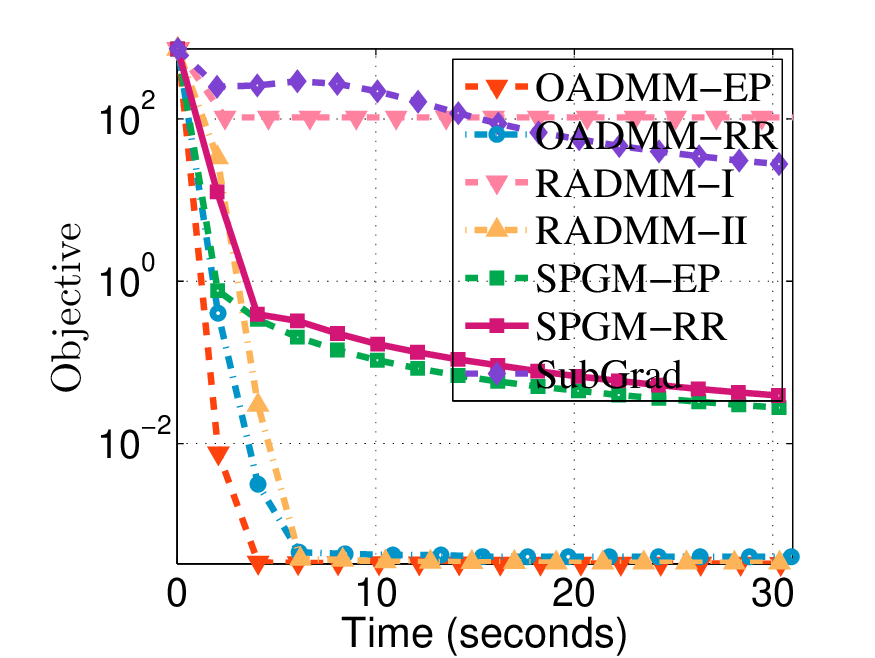}\caption{\scriptsize TDT2-1500-500}\label{fig:sub3}\end{subfigure}
\begin{subfigure}{.24\textwidth}\centering\includegraphics[width=1.12\linewidth]{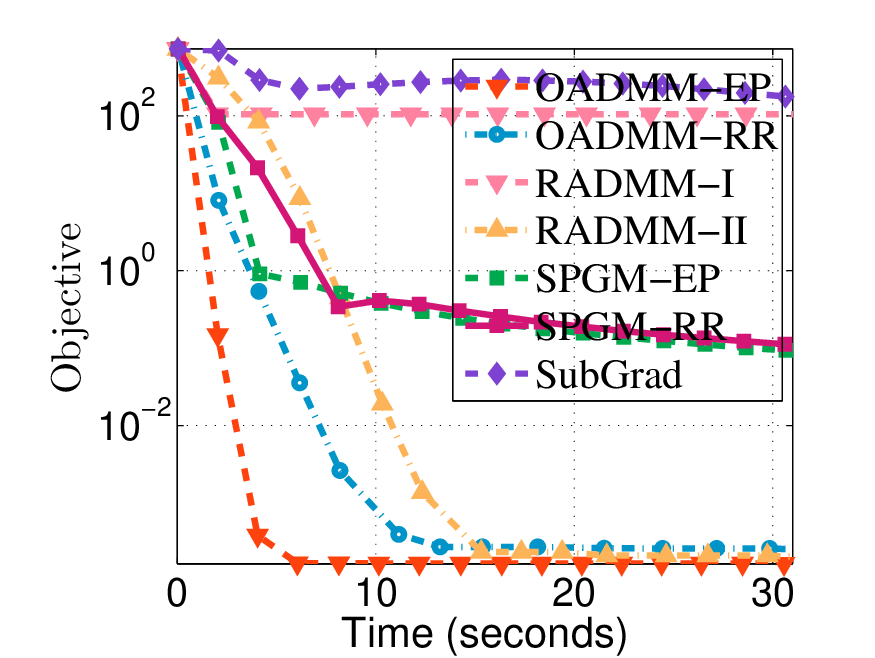}\caption{\scriptsize TDT2-2500-500}\label{fig:sub4}\end{subfigure}

\begin{subfigure}{.24\textwidth}\centering\includegraphics[width=1.12\linewidth]{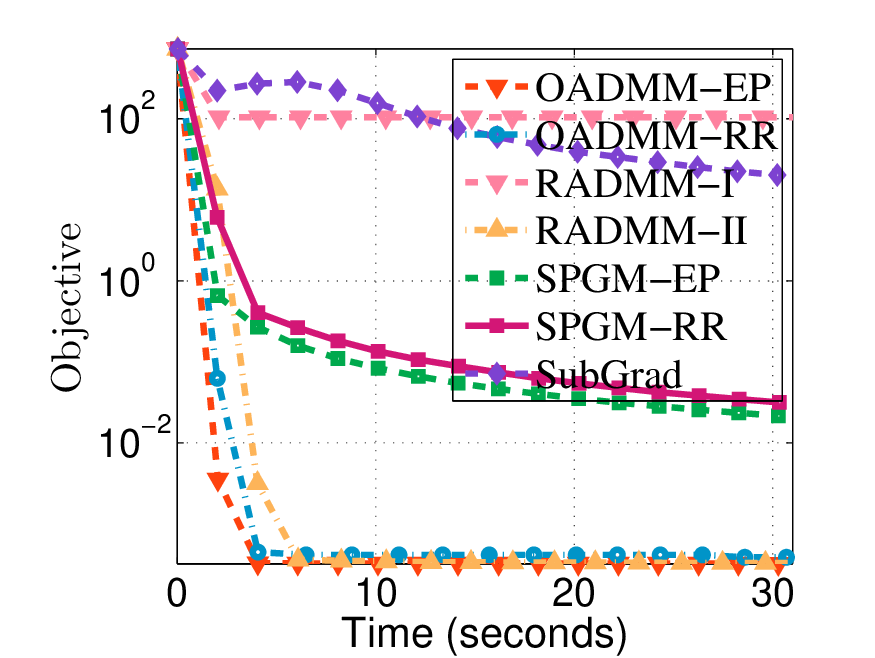}\caption{\scriptsize sector-1500-500}\label{fig:sub1}\end{subfigure}
\begin{subfigure}{.24\textwidth}\centering\includegraphics[width=1.12\linewidth]{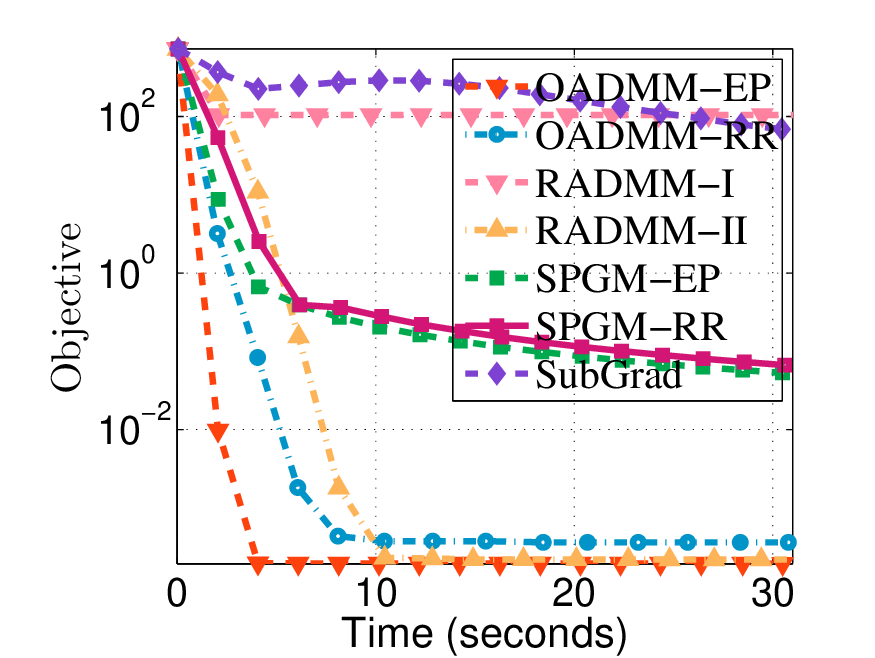}\caption{\scriptsize sector-2500-500}\label{fig:sub2}\end{subfigure}
\begin{subfigure}{.24\textwidth}\centering\includegraphics[width=1.12\linewidth]{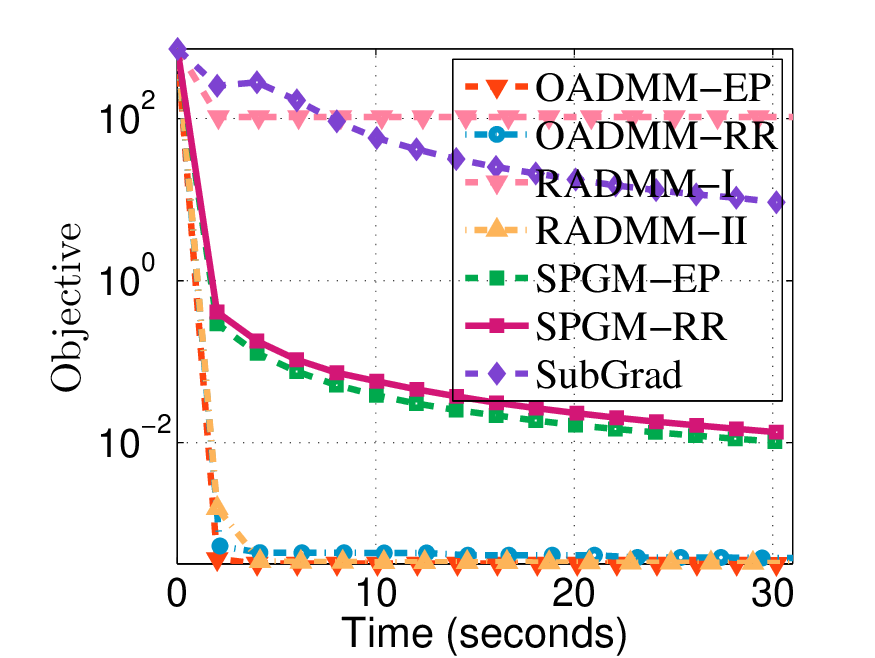}\caption{\scriptsize randn-1500-500}\label{fig:sub3}\end{subfigure}
\begin{subfigure}{.24\textwidth}\centering\includegraphics[width=1.12\linewidth]{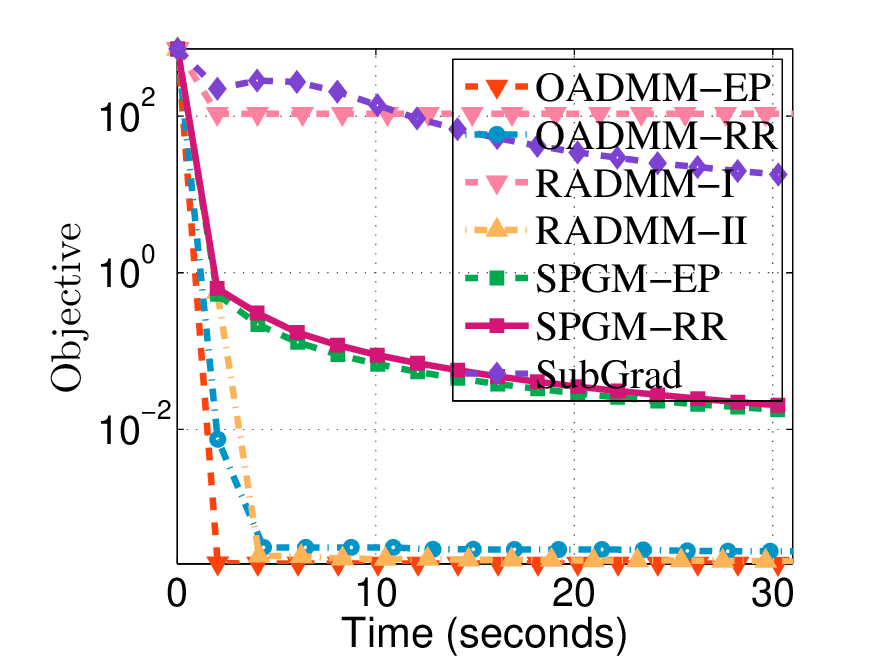}\caption{\scriptsize randn-2500-500}\label{fig:sub4}\end{subfigure}

\caption{The convergence curve of the compared methods with $\dot{\rho}=10$.}

\label{fig:L0PCA:rho:10}

\centering
\centering
\begin{subfigure}{.24\textwidth}\centering\includegraphics[width=1.12\linewidth]{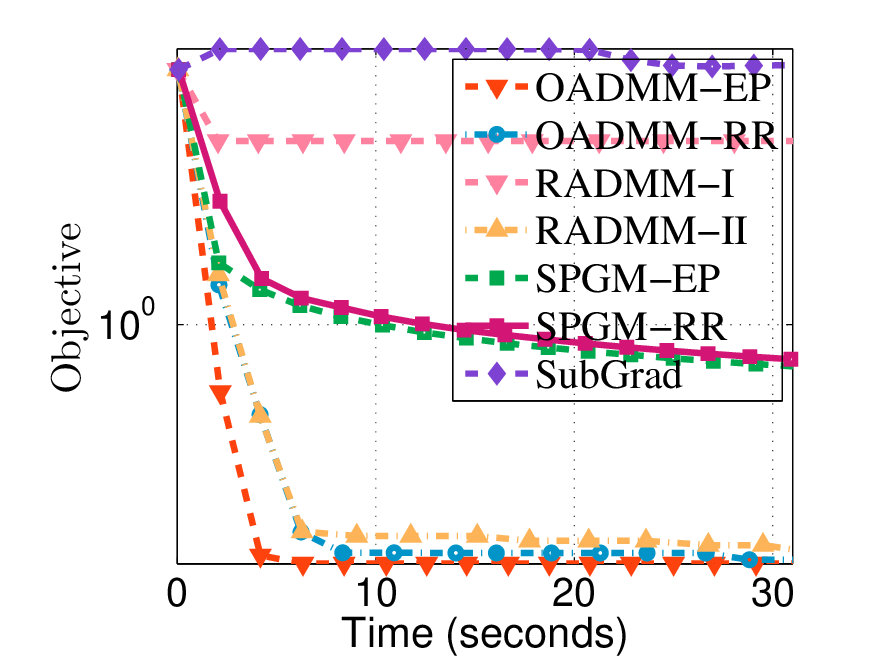}\caption{\scriptsize mnist-1500-500}\label{fig:sub1}\end{subfigure}
\begin{subfigure}{.24\textwidth}\centering\includegraphics[width=1.12\linewidth]{fig//demo_L0PCA100_12.eps}\caption{\scriptsize mnist-2500-500}\label{fig:sub2}\end{subfigure}
\begin{subfigure}{.24\textwidth}\centering\includegraphics[width=1.12\linewidth]{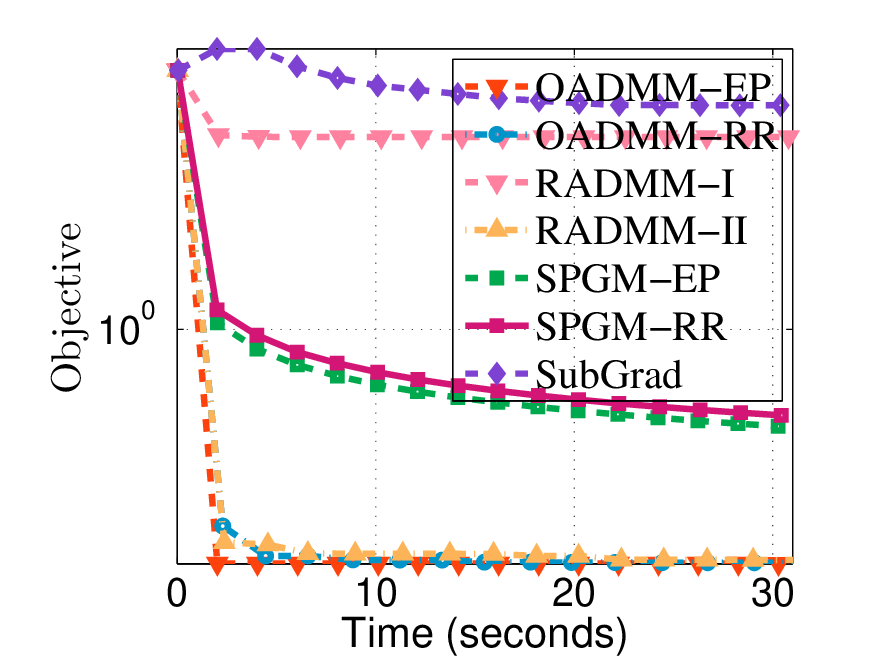}\caption{\scriptsize TDT2-1500-500}\label{fig:sub3}\end{subfigure}
\begin{subfigure}{.24\textwidth}\centering\includegraphics[width=1.12\linewidth]{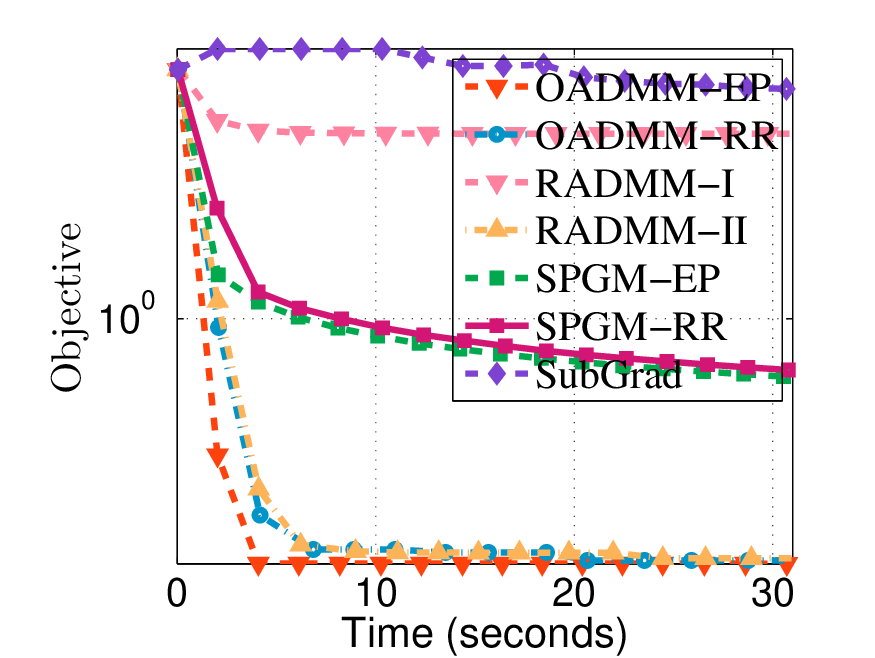}\caption{\scriptsize TDT2-2500-500}\label{fig:sub4}\end{subfigure}

\begin{subfigure}{.24\textwidth}\centering\includegraphics[width=1.12\linewidth]{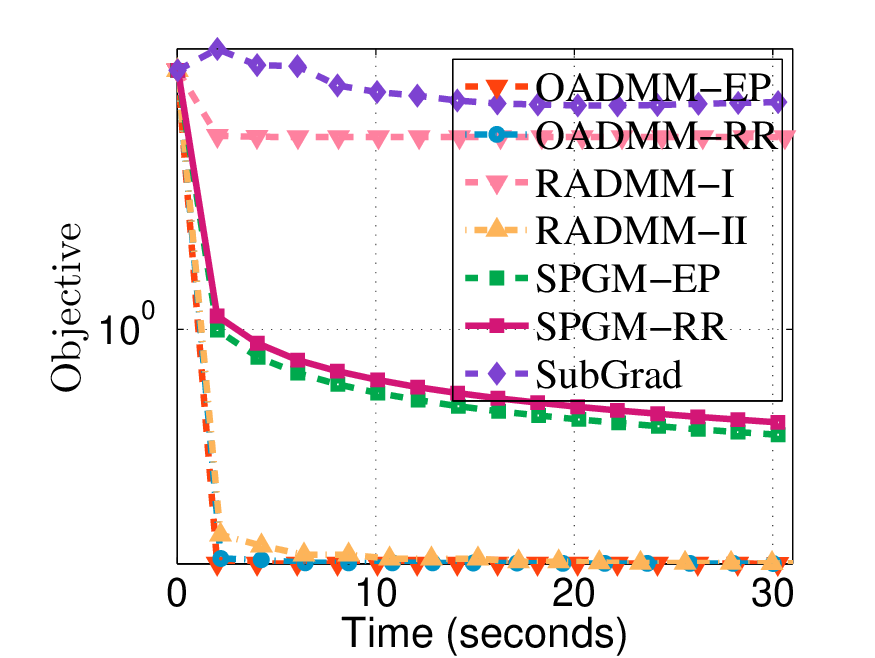}\caption{\scriptsize sector-1500-500}\label{fig:sub1}\end{subfigure}
\begin{subfigure}{.24\textwidth}\centering\includegraphics[width=1.12\linewidth]{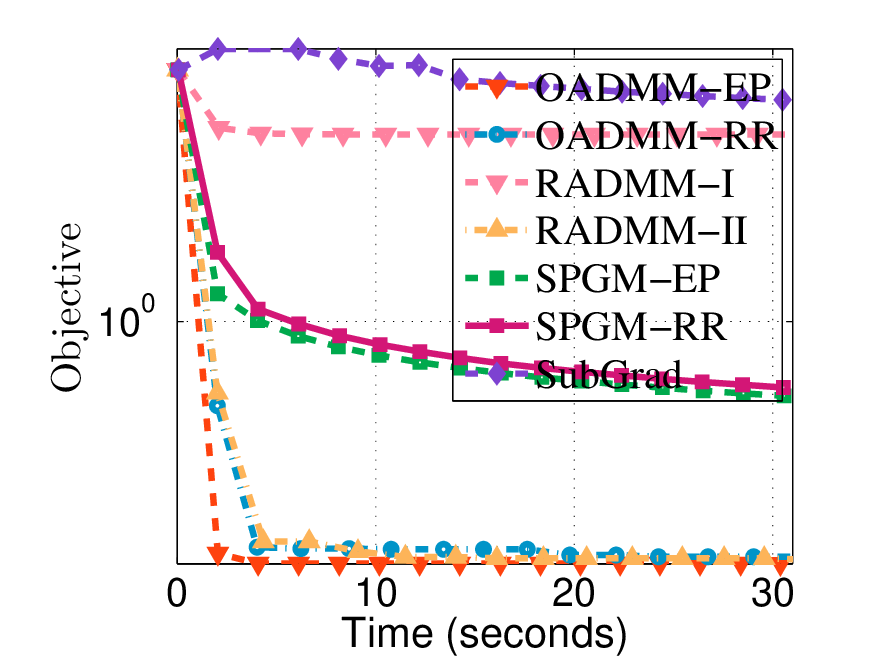}\caption{\scriptsize sector-2500-500}\label{fig:sub2}\end{subfigure}
\begin{subfigure}{.24\textwidth}\centering\includegraphics[width=1.12\linewidth]{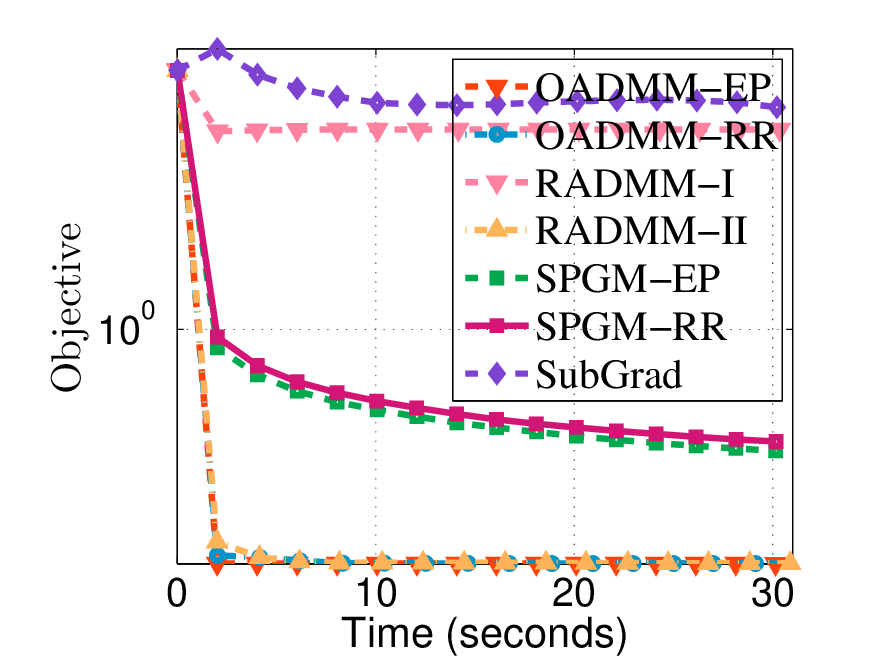}\caption{\scriptsize randn-1500-500}\label{fig:sub3}\end{subfigure}
\begin{subfigure}{.24\textwidth}\centering\includegraphics[width=1.12\linewidth]{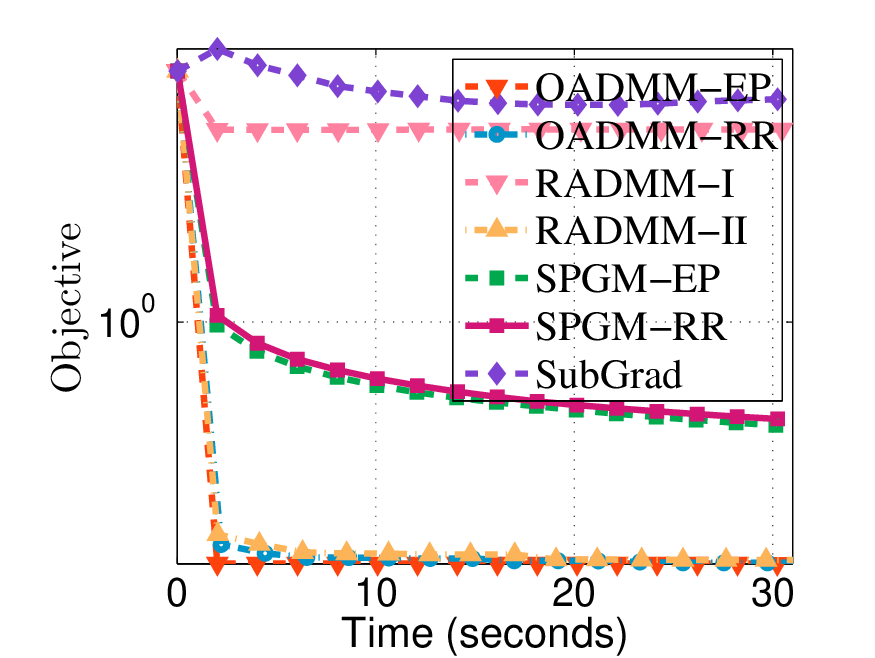}\caption{\scriptsize randn-2500-500}\label{fig:sub4}\end{subfigure}

\caption{The convergence curve of the compared methods with $\dot{\rho}=100$.}

\label{fig:L0PCA:rho:100}

\end{figure}

\begin{figure}[!t]
\centering

\centering
\begin{subfigure}{.24\textwidth}\centering\includegraphics[width=1.12\linewidth]{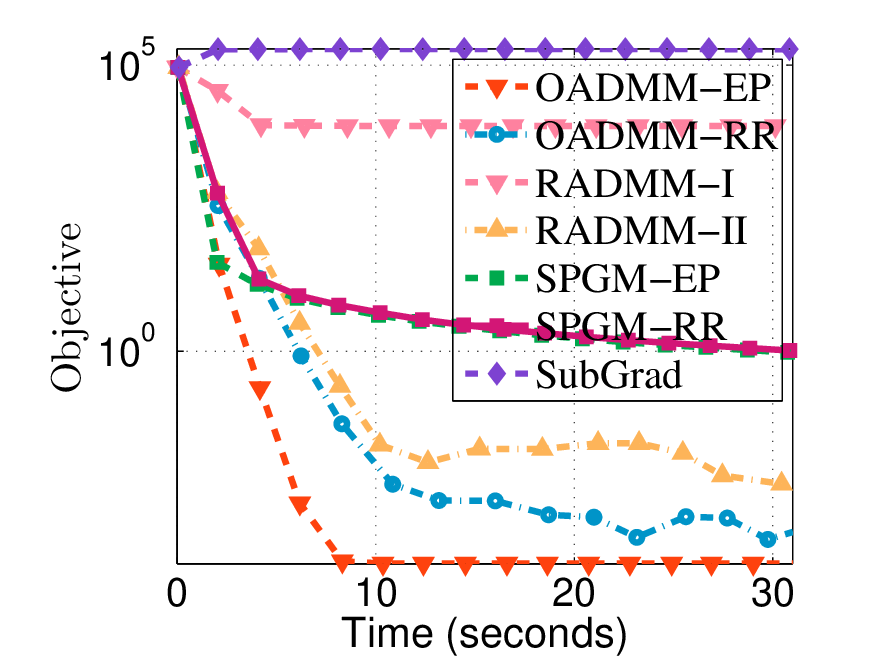}\caption{\scriptsize mnist-1500-500}\label{fig:sub1}\end{subfigure}
\begin{subfigure}{.24\textwidth}\centering\includegraphics[width=1.12\linewidth]{fig//demo_L0PCA1000_12.eps}\caption{\scriptsize mnist-2500-500}\label{fig:sub2}\end{subfigure}
\begin{subfigure}{.24\textwidth}\centering\includegraphics[width=1.12\linewidth]{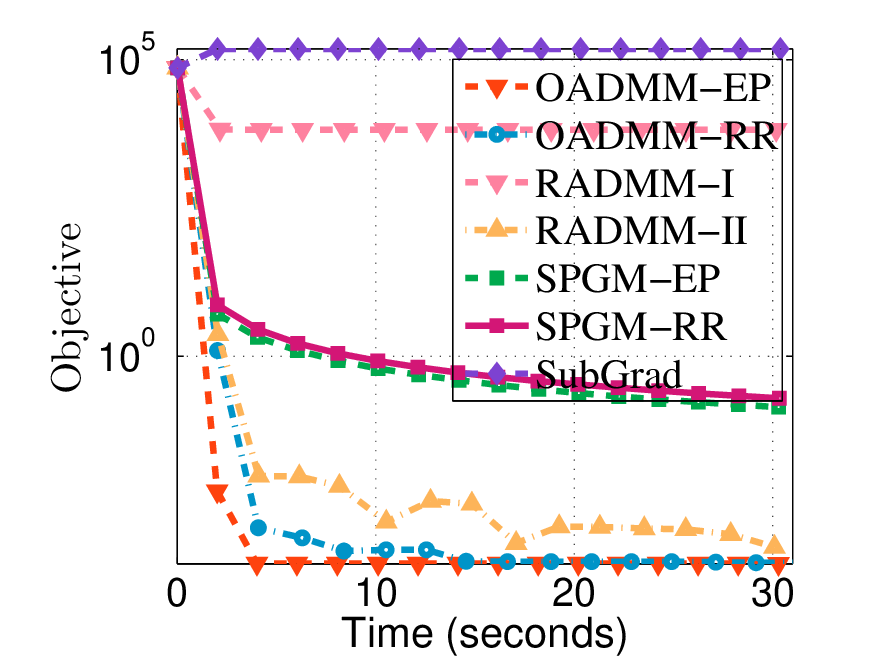}\caption{\scriptsize TDT2-1500-500}\label{fig:sub3}\end{subfigure}
\begin{subfigure}{.24\textwidth}\centering\includegraphics[width=1.12\linewidth]{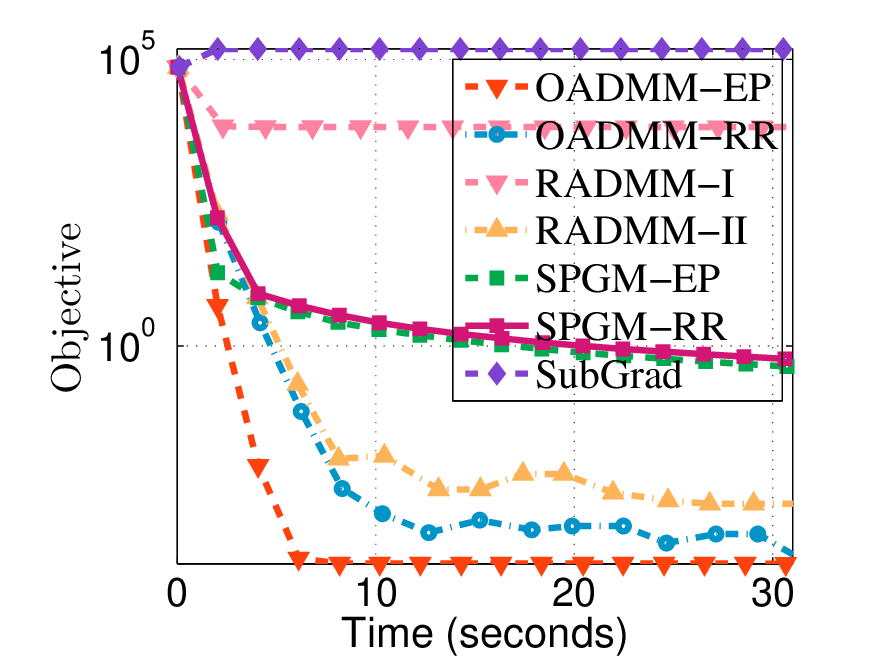}\caption{\scriptsize TDT2-2500-500}\label{fig:sub4}\end{subfigure}

\begin{subfigure}{.24\textwidth}\centering\includegraphics[width=1.12\linewidth]{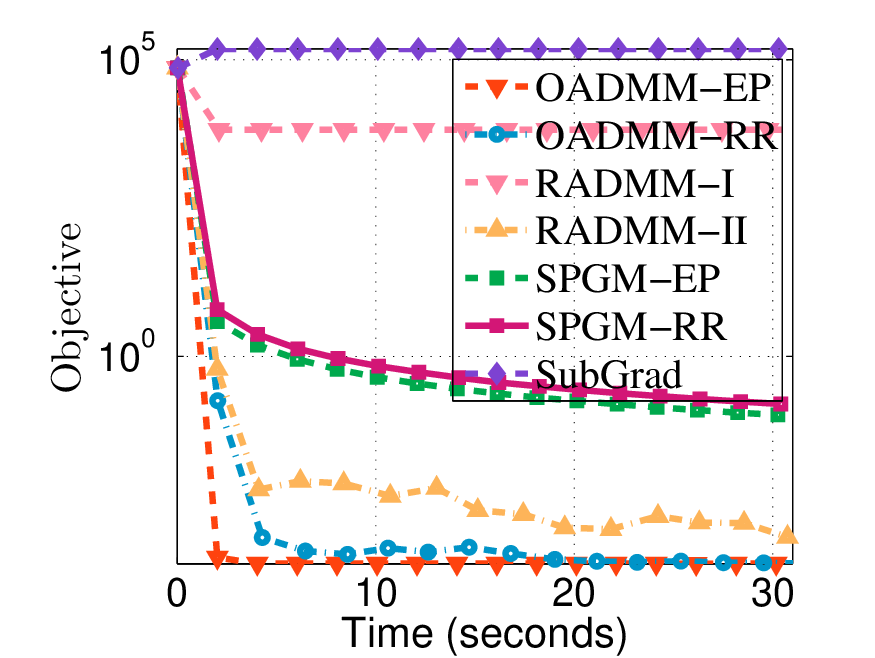}\caption{\scriptsize sector-1500-500}\label{fig:sub1}\end{subfigure}
\begin{subfigure}{.24\textwidth}\centering\includegraphics[width=1.12\linewidth]{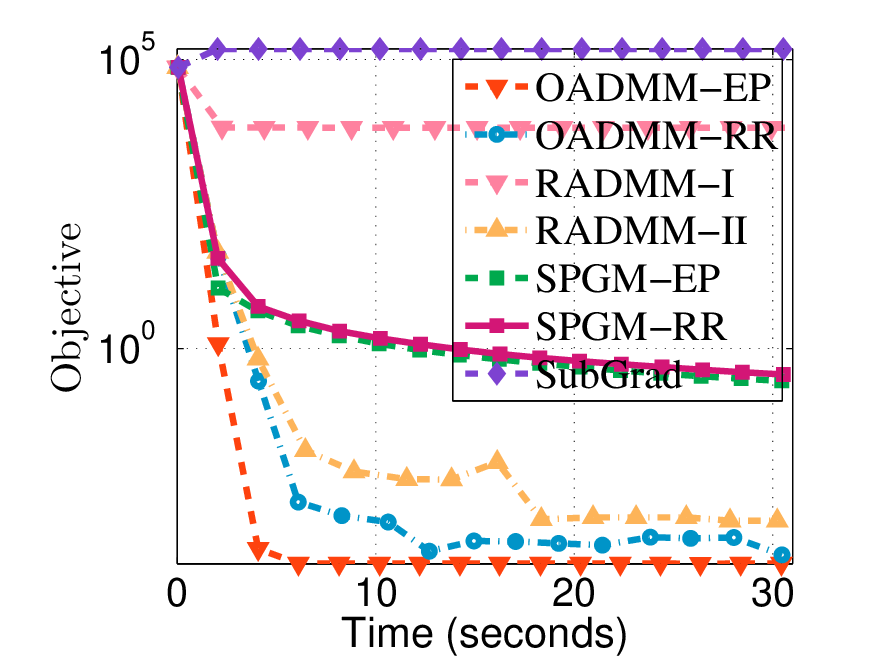}\caption{\scriptsize sector-2500-500}\label{fig:sub2}\end{subfigure}
\begin{subfigure}{.24\textwidth}\centering\includegraphics[width=1.12\linewidth]{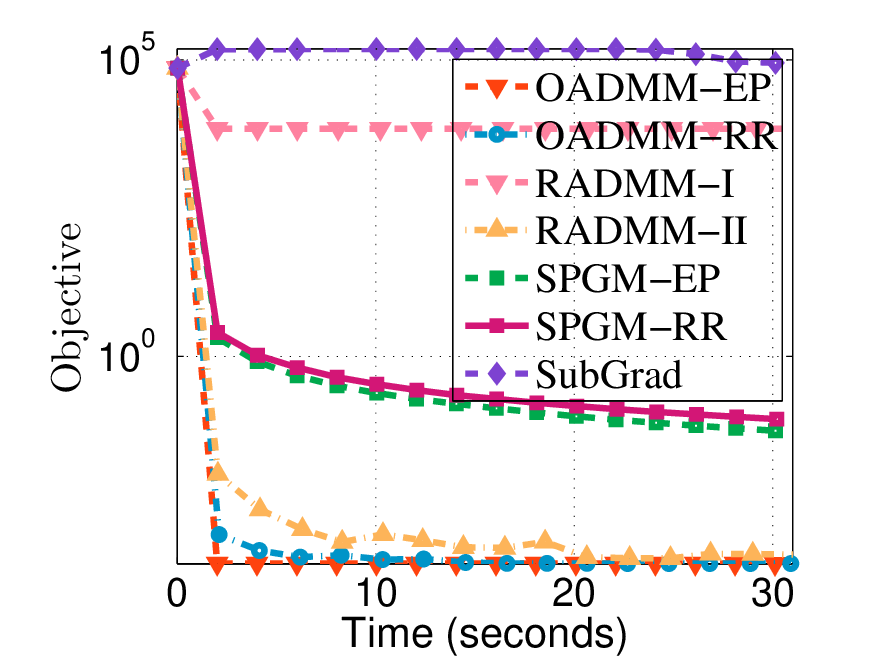}\caption{\scriptsize randn-1500-500}\label{fig:sub3}\end{subfigure}
\begin{subfigure}{.24\textwidth}\centering\includegraphics[width=1.12\linewidth]{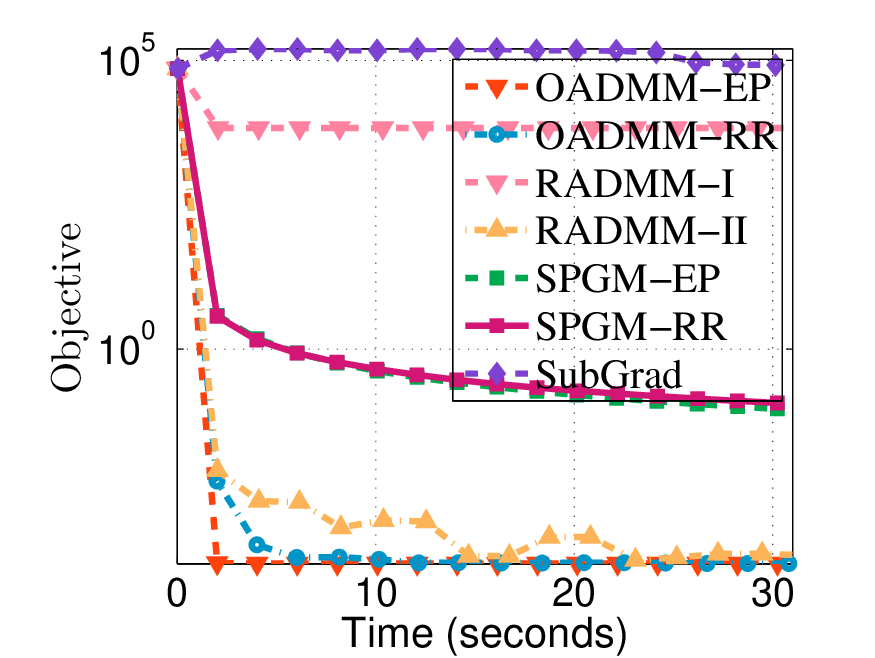}\caption{\scriptsize randn-2500-500}\label{fig:sub4}\end{subfigure}

\caption{The convergence curve of the compared methods with $\dot{\rho}=1000$.}

\label{fig:L0PCA:rho:1000}

\end{figure}

\end{document}